\documentclass[12pt]{article}
\usepackage{amsthm,amsfonts,amssymb,amsmath}
\usepackage{enumerate}
\usepackage{fullpage}
\usepackage[colorlinks=true]{hyperref}
\usepackage{xcolor}

\newcommand{\BC}{{\mathbb {C}}}

\newcommand{\BP}{{\mathbb {P}}}
\newcommand{\BQ}{{\mathbb {Q}}}
\newcommand{\BR}{{\mathbb {R}}}

\newcommand{\CA}{{\mathcal {A}}}

\newcommand{\CCC}{{\mathcal {C}}}
\newcommand{\CD}{{\mathcal {D}}}
\newcommand{\CE}{{\mathcal {E}}}

\newcommand{\CH}{{\mathcal {H}}}

\newcommand{\CL}{{\mathcal {L}}}
\newcommand{\CM}{{\mathcal {M}}}

\newcommand{\CO}{{\mathcal {O}}}

\newcommand{\CS}{{\mathcal {S}}}

\newcommand{\CU}{{\mathcal {U}}}
\newcommand{\CV}{{\mathcal {V}}}

\newcommand{\CX}{{\mathcal {X}}}
\newcommand{\CY}{{\mathcal {Y}}}

\newcommand{\OL}{{\overline{L}}}
\newcommand{\OM}{{\overline{M}}}
\newcommand{\ON}{{\overline{N}}}

\newcommand{\OD}{{\overline{D}}}

\newcommand{\OB}{{\overline{B}}}

\newcommand{\TL}{{\widetilde{L}}}
\newcommand{\TM}{{\widetilde{M}}}

\renewcommand{\OE}{{\overline{E}}}

\newcommand{\OTheta}{{\overline{\Theta}}}

\newcommand{\Fal}{{\mathrm{Fal}}}

\newcommand{\an}{{\mathrm{an}}}
\newcommand{\ran}{{\text{r-an}}}

\newcommand{\red}{{\mathrm{red}}}

\newcommand{\Div}{{\mathrm{Div}}}
\newcommand{\Divhat}{{\widehat{\mathrm{Div}}}}
\renewcommand{\div}{{\mathrm{div}}}

\newcommand{\Hom}{{\mathrm{Hom}}}

\newcommand{\intb}{{\mathrm{int}}}

\newcommand{\rank}{{\mathrm{rank}}}

\newcommand{\Pic}{\mathrm{Pic}}

\newcommand{\CaCl}{\mathrm{CaCl}}

\newcommand{\Picc}{{\mathcal{P}\mathrm{ic}}}
\newcommand{\Pichat}{\widehat{\mathrm{Pic}}}

\newcommand{\rmod}{{\mathrm{mod}}}

\DeclareMathOperator{\Spec}{Spec}

\newcommand{\vol}{{\mathrm{vol}}}

\newcommand{\wt}{\widetilde}
\newcommand{\wh}{\widehat}

\newcommand{\pair}[1]{\langle {#1} \rangle}

\newcommand{\ds}{\displaystyle}

\newcommand{\ol }{\overline}

\newcommand{\lra}{\longrightarrow}

\newcommand{\kkk}{Let $k$ be either $\ZZ$ or a field. }

\newcommand{\nef}{\mathrm{nef}}

\newcommand{\eff}{\mathrm{eff}}

\renewcommand{\d}{\mathrm{d}}

\newcommand{\CLL}{\overline{\mathcal L}}
\newcommand{\CMM}{\overline{\mathcal M}}
\newcommand{\CHH}{\overline{\mathcal H}}

\newcommand{\CDD}{\overline{\mathcal D}}
\newcommand{\CEE}{\overline{\mathcal E}}

\newcommand{\CC}{\mathbb{C}}
\newcommand{\RR}{\mathbb{R}}
\newcommand{\ZZ}{\mathbb{Z}}
\newcommand{\QQ}{\mathbb{Q}}

\newcommand{\PP}{\mathbb{P}}

\providecommand{\lb}{\mathcal{L}}

\providecommand{\lbb}{\overline{\mathcal{L}}}

\providecommand{\VCV}{\|\cdot\|}

\newtheorem{thm}{Theorem}[section]

\newtheorem{lem}[thm]{Lemma}
\newtheorem{prop}[thm]{Proposition}

\newtheorem{theorem}[thm]{Theorem}

\newtheorem{lemma}[thm]{Lemma}

\theoremstyle{definition}

\theoremstyle{remark}

\begin{document}

\title{Adelic Line Bundles, Arithmetic Positivity and Diophantine Geometry}
\author{Xinyi Yuan}
\date{}
\maketitle

\begin{abstract}
This article is an expository account of the basics of
adelic line bundles on quasi-projective varieties, arithmetic positivity of adelic line bundles, 
and applications of positivity to an equidistribution theorem and 
the uniform Bogomolov conjecture.  
\end{abstract}

\tableofcontents

\section{Introduction}

Diophantine geometry studies rational points (or integral points, algebraic points) of algebraic varieties over number fields. 
The most crucial tool to study this topic is heights of algebraic points originally introduced by Weil \cite{Wei51}.
Arakelov geometry, introduced by Arakelov \cite{Ara} and further developed by Deligne \cite{Del} and Gillet--Soul\'e \cite{GS}, is an intersection theory over arithmetic varieties over $\ZZ$ by combining intersection numbers on algebraic schemes and integrals on complex analytic spaces. 
In particular, heights of algebraic points are naturally interpreted as arithmetic intersection numbers in Arakelov geometry, and thus can be understood by geometric intuition. 

Among works in the 20th century, the following are probably the three most significant applications of Arakelov geometry to Diophantine geometry:
\begin{enumerate}
\item[(1)] the proof of the Mordell conjecture by Faltings \cite{Fal},  
\item[(2)] the discovery of an equidistribution theorem by Szpiro--Ullmo--Zhang \cite{SUZ},  
\item[(3)] the proof of the Bogomolov conjecture by Ullmo \cite{Ull} and Zhang \cite{Zh3}.
\end{enumerate}
In the 21st century, all these three topics have substantial developments, which include:
\begin{enumerate}
\item[(4)] the extension of the equidistribution theorem to semipositive metrics and all places by Chambert-Loir \cite{CL} and Yuan \cite{Yua},  
\item[(5)] the proof of the geometric Bogomolov conjecture by 
Gubler \cite{Gub07},  Yamaki \cite{Yam16, Yam18, Yam17b}, Cantat--Gao--Habegger--Xie \cite{CGHX}, and  Xie--Yuan \cite{XY22},
\item[(6)] the proof of the uniform Bogomolov conjecture by Dimitrov--Gao--Habegger \cite{DGH} and K\"uhne \cite{Kuh}, which implies the uniform Mordell conjecture by combining the work of Vojta \cite{Voj}. 
\end{enumerate}

To apply Arakelov geometry to Diophantine geometry and algebraic dynamics, it is very common to vary integral models of projective varieties or even take limiting information from an infinite sequence of integral models. 
For this purpose, Zhang \cite{Zha2} introduced an intersection theory of adelic line bundles on projective varieties over number fields. 
In fact, all the above works on equidistribution theorems and Bogomolov conjectures are based on this theory.
On the other hand, Zhang's theory is not applicable to similar problems on quasi-projective varieties over number fields, which appear in many natural situations involving moduli spaces, degenerations and varieties over finitely generated fields. 
For this purpose, Yuan--Zhang \cite{YZ} introduced an extension of the theory of \cite{Zha2} to quasi-projective varieties over number fields. 
As a consequence, Yuan--Zhang \cite{YZ} extended the equidistribution theorem to quasi-projective varieties, and Yuan \cite{Yua21} gave a theoretical proof of the uniform Bogomolov conjecture originally proved by 
Dimitrov--Gao--Habegger \cite{DGH} and K\"uhne \cite{Kuh}. 

In these works,  arithmetic positivity of adelic line bundles is the theoretical foundation for our applications of Arakelov geometry to Diophantine geometry. The most commonly used theorems in this direction are the arithmetic Hilbert--Samuel formula originally due to Gillet--Soul\'e \cite{GS2,GS3}, Bismut--Vasserot \cite{BV89}, and Zhang \cite{Zh1}, and the arithmetic Siu inequality originally due to Yuan \cite{Yua}.

The goal of this expository article is to introduce the above works of 
Yuan--Zhang \cite{YZ} and Yuan \cite{Yua21}. 
More precisely, we will cover the following topics.
\begin{enumerate}
\item[(a)] In \S\ref{sec adelic}, we review the basic theory of adelic line bundles, their analytification and intersection theory. 
\item[(b)] In \S\ref{sec positivity}, we sketch the arithmetic positivity theory of adelic line bundles. The notions include nefness, bigness, and volumes of adelic line bundles.
\item[(c)] In \S\ref{sec height}, we study height functions associated to adelic line bundles, and introduce fundamental inequalities relating arithmetic positivity to algebraic points of small heights. 
\item[(d)] In \S\ref{sec equi}, we prove an equidistribution theorem of small points over quasi-projective varieties. The proof is based on the theory in the previous three sections. 
\item[(e)] In \S\ref{sec admissible}, we introduce the admissible canonical bundle for a relative curve as an adelic line bundle, and  prove its bigness in the case of maximal variation. 
\item[(f)] In \S\ref{sec uniform}, we apply the bigness of the admissible canonical bundle to prove the uniform Bogomolov conjecture. 
\end{enumerate}
The contents of \S\ref{sec adelic}-\S\ref{sec equi} are from Yuan--Zhang \cite{YZ}, and the contents of \S\ref{sec admissible}-\S\ref{sec uniform} are from Yuan \cite{Yua21}. 

Almost all definitions and results over number fields in this article have direct counterparts over function fields. We only treat the case of number fields for simplicity, while many original references actually {cover} function fields. We leave interested readers to find them in the references or {work them out themselves}.

This article starts with the most basic intersection theory of hermitian line bundles on arithmetic varieties, but the article may be difficult to read without familiarity with basic Arakelov geometry (including Zhang's theory of adelic line bundles on projective varieties). For {quick access} to these topics, we refer {readers} to our previous expository article \cite{Yua12}.

\subsubsection*{Notations and terminology}\label{sec notation}

By a \emph{variety} over a field $k$, we mean an integral scheme,  separated of finite type over $k$. 
By a \emph{curve} over a field $k$, we mean a variety of dimension 1 over a field $k$.

By a \emph{line bundle} on a scheme, we mean an invertible sheaf on the scheme. 
We often write tensor products of line bundles additively, so $aL-bM$ means
$L^{\otimes a}\otimes M^{\otimes (-b)}$
for line bundles $L,M$ and integers $a,b$.

To treat adelic line bundles of \cite{YZ}, we follow the uniform terminology of \cite[\S1.5]{YZ}.
In particular, most of the time, our base ring $k$ is $\ZZ$ or a field.

By a \emph{number field} $K$, we mean a finite extension of $\QQ$. 
Denote by $M_K$ the set of places of $K$. 
For any place $v$ of $K$, normalize the absolute value $|\cdot|_v$ on the completion $K_v$ as follows.
\begin{enumerate}[(1)]
\item If $v$ is an archimedean place, take $|\cdot|_v$ to be the usual absolute value on $K_v\cong\RR,\CC$;
\item If $v$ is a non-archimedean place, set $|a|_v=\big(\#(O_{K_v}/aO_{K_v})\big)^{-1}$ for any $a\in O_{K_v}$. 
\end{enumerate}
The valuations satisfy the product formula 
$$
\prod_{v\in M_K} |a|_v^{\epsilon_v} =1, \quad \forall \ a\in K^\times. 
$$
Here $\epsilon_v=1$ for all real or non-archimedean places $v$; $\epsilon_v=2$ for all complex places $v$.

\subsubsection*{Acknowledgments}
The author is supported by the grant {No.} 12321001
from the National Science Foundation of China, and by the Xplorer Prize from the New Cornerstone Science Foundation.
The author would also like to {acknowledge support from} the China--Russia Mathematics Center.

\section{Adelic line bundles}
\label{sec adelic}

In \cite{Ara}, Arakelov introduced an arithmetic intersection theory of hermitian line bundles (and arithmetic divisors) on arithmetic surfaces. 
Following ideas of \cite{Del, Elk2}, Arakelov's construction can be generalized to top intersection numbers of hermitian line bundles on arithmetic varieties. 
In \cite{GS}, Gillet--Soul\'e introduced arithmetic Chow cycles on arithmetic varieties and extended the intersection theory to arithmetic Chow groups. 

All the above theories work on arithmetic varieties, but as we will see, these are inconvenient or insufficient for applications in Diophantine geometry and arithmetic dynamics. In \cite{Zha2}, Zhang introduced a notion of adelic line bundles on projective varieties over number fields.  
In \cite{YZ}, Yuan--Zhang further extended the theory to quasi-projective varieties over {number fields or constant fields}.
{The goal of this section is to sketch the theory of adelic line bundles} of \cite{YZ}. 

\subsection{Intersection of hermitian line bundles} \label{iccm2010:intersection}

Let us introduce some terminology of intersection of hermitian line bundles developed by Arakelov \cite{Ara}, Deligne \cite{Del}, and Gillet--Soul\'e \cite{GS}. 
We refer to \cite{YG} for more details.  

\subsubsection{Hermitian line bundles} 

Let ${\mathcal{X}}$ be an arithmetic variety of dimension $d+1$, i.e. a projective and flat integral scheme over ${\mathrm{Spec}}({\mathbb Z})$ of absolute dimension $d+1$. 

A \emph{hermitian line bundle} on ${\mathcal{X}}$ is a pair ${\overline{\mathcal{L}}}=({\mathcal{L}},{\|\cdot\|})$, where ${\mathcal{L}}$ is a line bundle on ${\mathcal{X}}$, and ${\|\cdot\|}$ is a smooth hermitian metric of ${\mathcal{L}}({\mathbb C})$ on the complex variety ${\mathcal{X}}({\mathbb C})$ invariant under the complex conjugation. 
If ${\mathcal{X}}({\mathbb C})$ is singular, the smoothness of the metric means that the metric is locally equal to the pull-back of a smooth metric via a closed embedding into a smooth complex manifold. 

The \emph{Chern form} $c_1(\CLL)$ of $\CLL$ is a $(1,1)$-form on $\CX(\CC)$ determined as follows. For any pair $(U,s_U)$ of an open subset $U$ of $\CX(\CC)$ together with an everywhere non-vanishing holomorphic section $s_U$ of $\CL(\CC)$ on  $U$, we have the restriction
$$c_1(\CLL)|_U=\frac{1}{\pi i}\partial\bar\partial \log \|s_U\|. $$
We say that the hermitian metric $\|\cdot\|$
is \emph{semipositive} if $c_1(\CLL)$ is 
{positive semi-definite} everywhere.

Denote by $\widehat{\mathrm{Pic}}({\mathcal{X}})$ the group of isometry classes of hermitian line bundles on ${\mathcal{X}}$.
There is an intersection pairing
$$
\widehat{\mathrm{Pic}}({\mathcal{X}})^{d+1} \longrightarrow {\mathbb R} 
$$
as follows. 

If $\dim {\mathcal{X}}=1$ and $\CX$ is normal, then ${\mathcal{X}}={\mathrm{Spec}}(O_K)$ for some number field $K$, and ${\mathcal{L}}$ is an $O_K$-module of rank $1$ and the hermitian metric ${\|\cdot\|}$ is a collection $(\|\cdot\|_\sigma)_{\sigma:K\to\mathbb{C}}$ of metrics on the complex line ${\mathcal{L}}_\sigma({\mathbb C})$. Let $s$ be any nonzero element of ${\mathcal{L}}$. The \emph{arithmetic degree} of ${\mathcal{L}}$ is defined as
$$\widehat{\mathrm{deg}}({\mathcal{L}})=\log\#({\mathcal{L}}/s\mathcal L)-\sum_{\sigma:K\to{\mathbb C}}\log\|s\|_\sigma.$$
It is independent of the choice of $s$ by the product formula. The intersection pairing is just the arithmetic degree map
$$
\widehat{\mathrm{deg}}: \widehat{\mathrm{Pic}}({\mathcal{X}}) \longrightarrow {\mathbb R}.
$$

In general, let ${\overline{\mathcal L}}_1, {\overline{\mathcal L}}_2, \cdots, {\overline{\mathcal L}}_{d+1}$ be $d+1$ hermitian line bundles on ${\mathcal{X}}$, and let $s_{d+1}$ be any nonzero rational section of ${\mathcal{L}}_{d+1}$ on ${\mathcal{X}}$. 
The \emph{arithmetic intersection number} is defined inductively by 
\begin{eqnarray*}
{\overline{\mathcal L}}_1\cdot {\overline{\mathcal L}}_2 \cdots {\overline{\mathcal L}}_{d+1}
= {\overline{\mathcal L}}_1\cdot {\overline{\mathcal L}}_2 \cdots {\overline{\mathcal L}}_{d} \cdot {\mathrm{wdiv}}(s_{d+1})
- \int_{{\mathcal{X}}({\mathbb C})} \log \|s_{d+1}\| c_1({\overline{\mathcal L}}_1)\cdots  c_1({\overline{\mathcal L}}_d).
\end{eqnarray*}
The right-hand side depends on the Weil divisor ${\mathrm{wdiv}}(s_{d+1})$ linearly, so it suffices to explain the case that ${\mathcal D}={\mathrm{wdiv}}(s_{d+1})$ is irreducible (and reduced). 

If ${\mathcal D}$ is horizontal in the sense that it is flat over ${\mathbb Z}$, then
$$
{\overline{\mathcal L}}_1\cdot {\overline{\mathcal L}}_2 \cdots {\overline{\mathcal L}}_{d} \cdot {\mathcal D}= {\overline{\mathcal L}}_1|_{\mathcal D}\cdot {\overline{\mathcal L}}_2|_{\mathcal D}\cdots {\overline{\mathcal L}}_{d}|_{\mathcal D}
$$
is an arithmetic intersection on ${\mathcal D}$ defined by the induction hypothesis. 
If ${\mathcal D}$ is a vertical divisor in the sense that it is a variety over $\mathbb F_p$ for some prime $p$, then 
$$
{\overline{\mathcal L}}_1\cdot {\overline{\mathcal L}}_2 \cdots {\overline{\mathcal L}}_{d} \cdot {\mathcal D}= ({\mathcal{L}}_1|_{\mathcal D}\cdot {\mathcal{L}}_2|_{\mathcal D}\cdots {\mathcal{L}}_{d}|_{\mathcal D})\log p.
$$
Here the intersection is the usual intersection on the projective variety $\mathcal D$ over $\mathbb F_p$.
The definition does not depend on the choice of the rational section $s_{d+1}$. 
It gives a symmetric and multi-linear intersection pairing
$$
\widehat{\mathrm{Pic}}({\mathcal{X}})^{d+1} \longrightarrow {\mathbb R}. 
$$
When $d=0$, it is just the arithmetic degree map.

In general, for any closed integral subscheme $\mathcal{Y}$ of ${\mathcal{X}}$, we denote 
$${\overline{\mathcal{L}}}^{\dim \CY}\cdot \mathcal{Y}
=(\overline{\mathcal{L}}|_\mathcal{Y})^{\dim \CY}.$$ 
By linearity, this notation extends to Chow cycles $\CY$ of $\CX$.

Let ${\mathcal{X}}$ be an arithmetic variety and ${\overline{\mathcal{L}}}=({\mathcal{L}},{\|\cdot\|})$ be a {hermitian line bundle} on ${\mathcal{X}}$. We say that ${\overline{\mathcal{L}}}$ is \emph{nef} if its hermitian metric is semipositive and the intersection number ${\overline{\mathcal{L}}}\cdot \mathcal{C}\geq0$ for any 1-dimensional closed integral subscheme $\mathcal{C}$ of ${\mathcal{X}}$.
As a consequence of the arithmetic Nakai--Moishezon criterion of Zhang \cite{Zh1}, this implies that 
 the intersection number ${\overline{\mathcal{L}}}^{\dim \CY}\cdot \mathcal{Y}\geq0$ for any closed integral subscheme $\mathcal{Y}$ of ${\mathcal{X}}$.
We refer to \cite[Prop. 2.3]{Mo1} and 
 \cite[Cor. A.4.3]{YZ} for this arithmetic version {of Kleiman's theorem}.

\subsubsection{Arithmetic divisors}

Let $\CX$ be an arithmetic variety. 
An  {\em arithmetic divisor on $\CX$} is a pair $\CDD=(\CD,g_\CD)$, 
where  $\CD$ is a Cartier divisor on $\CX$, and $g_\CD$ is a \emph{Green function} of $\CD(\CC)$ on $\CX(\CC)$, invariant under the action of {complex conjugation}. For the regularization of the Green function, we require that 
$$g_\CD:\CX(\CC)\setminus |\CD(\BC)|\lra \RR$$
 is smooth, and that 
for any pair $(U,f_U)$ of an open subset $U$ of $\CX(\CC)$ together with a rational function $f_U$ on $U$ with $\div(f_U)=\CD(\CC)|_U$, the function $g_{\CD}+\log|f_U|$ extends to a smooth function on $U$.

A {\em principal arithmetic divisor on $\CX$} is an arithmetic divisor of the form 
$$\wh\div(f):=(\div(f),-\log|f|)$$ 
for any rational function $f\in \QQ(\CX)^\times$ on $\CX$.

Denote by $\wh\Div(\CX)$ the group of arithmetic divisors on $\CX$, and by 
$\wh\Pr(\CX)$ the group of principal arithmetic divisors on $\CX$. 
Then we have the \emph{arithmetic divisor class group} 
$$
\wh\CaCl(\CX)=\wh\Div(\CX)/\wh\Pr(\CX).
$$

An arithmetic divisor $\CDD=(\CD,g_\CD) \in \Divhat(\CX)$ is \emph{effective}
(resp. \emph{strictly effective}) if
$\CD$ is an effective Cartier divisor on $\CX$ and the Green function $g_\CD\ge 0$
(resp. $g_\CD>0$) on $\CX(\BC)\setminus |\CD(\BC)|$.

There is a canonical map
$$
\wh\Div(\CX)\lra \wh\Pic(\CX), \quad\ \CDD\longmapsto \CO(\CDD),
$$
which induces an isomorphism
$$
\wh\CaCl(\CX)\lra \wh\Pic(\CX).
$$
The inverse image of a hermitian line bundle $\CLL$ is represented by the divisor 
$$
\wh\div(s)=\wh\div_{(\CX, \CLL)}(s):=(\div(s),-\log\|s\|),
$$
where $s$ is any nonzero rational section of $\CL$ on $\CX$.


\subsection{Adelic line bundles on quasi-projective varieties} \label{subsec adelic}

The goal of this subsection is to sketch the notion of adelic line bundles on quasi-projective varieties of Yuan--Zhang \cite{YZ}, which generalizes the more classical adelic line bundles of projective varieties over number fields of Zhang \cite{Zha2}.

\subsubsection{Adelic divisors}
\label{sec adelic divisor}

We take the uniform terminology of \cite[\S1.5]{YZ}. 
\kkk
Let $\CU$ be a quasi-projective and flat integral scheme over $k$.
Let us first recall the definition of adelic divisors on $\CU/k$.  

Let $\CX$ be a \emph{projective model} of $\CU$ over $k$, i.e., a projective integral scheme over $k$ with an open immersion $\CU\to\CX$ over $k$.
In the spirit of \cite[\S2.2]{YZ}, take the fiber product
$$
\wh\Div(\CX,\CU)=\wh\Div(\CX)_\QQ \times_{\Div(\CU)_\QQ} \Div(\CU),
$$
whose elements are arithmetic divisors of mixed coefficients. 
In the arithmetic case ({where} $k=\ZZ$), $\wh\Div(\CX)$ is the group of arithmetic divisors on $\CX$, where the Green's functions are assumed to be continuous (away from the singularities). 
In the geometric case (when $k$ is a field), $\wh\Div(\CX)$ means the usual 
$\Div(\CX)$.

By abuse of terminology, in the following, if $k$ is a field, then ``arithmetic divisor'' (resp. ``hermitian line bundle'') means
``divisor'' (resp. ``line bundle'').

Define the group of \emph{model adelic divisors} by
$$\wh\Div (\CU/k)_\rmod=\lim_{\substack{\lra\\ \CX}}\wh\Div(\CX,\CU).$$
Here the limit is over the system of projective models $\CX$ of $\CU$ over $k$.
An element of $\wh\Div (\CU/k)_{\rmod}$ is \emph{effective} if it is the image of an effective element of some $\wh\Div(\CX,\CU)$, where an element of $\wh\Div(\CX,\CU)$ is \emph{effective} if its images in $\wh\Div(\CX)_\QQ$ and $\Div(\CU)$ are both effective. 

Fix a \emph{boundary divisor} $(\CX_0,\CEE_0)$ of $\CU$,
i.e.,  a projective model $\CX_0$ of $\CU$ and a strictly effective arithmetic divisor $\CEE_0$ on $\CX_0$ such that the support of the finite part $\CE_0$ is exactly $\CX_0\setminus \CU$.
We have a \emph{boundary norm} 
$$\|\cdot\|_{\CEE_0}:\wh\Div (\CU/k)_\rmod
\lra [0,\infty]$$
by 
$$
\|\CDD\|_{\CEE_0}:=\inf\{\epsilon\in \BQ_{>0}: \ 
 -\epsilon \CEE_0 \leq
\CDD \leq  \epsilon \CEE_0\}.
$$
Here the inequalities are defined in terms of effectivity. 
It further induces a \emph{boundary topology} on $\wh\Div (\CU/k)_{\rmod}$, which does not depend on the choice of $(\CX_0,\overline\CE_0)$.

Let $\wh \Div  (\CU/k)$ be the \emph{completion} of $\wh \Div  (\CU/k)_{\rmod}$ with respect to the boundary topology. 
An element of $\wh \Div(\CU/k)$ is called an \emph{adelic divisor} on $\CU/k$.

By definition, an adelic divisor is represented by a Cauchy sequence in $\wh \Div  (\CU/k)_\rmod$, i.e., a sequence $\{\CDD_i\}_{i\geq 1}$ in $\wh \Div  (\CU/k)_\rmod$ satisfying the property that there is a sequence $\{\epsilon_i\}_{i\geq 1}$ of positive rational numbers converging to $0$ such that 
$$
 -\epsilon_i \CEE_0 \leq
\CDD_{i'}-\CDD_{i} \leq  \epsilon_i \CEE_0,\quad\ i'\geq i\geq 1.
$$

There is a canonical map
$$
\wh\Div(\CU/k) \lra \Div(\CU),\quad
\{\CDD_i\}_{i\geq 1}\longmapsto\CD_1|_{\CU}.
$$ 
We usually write $\CDD=\{\CDD_i\}_{i\geq 1}$ and $\CD=\CD_1|_{\CU}$, and call $\CD$ the \emph{underlying divisor} of $\CDD$.

\subsubsection{Adelic line bundles}

\kkk
Let $\CU$ be a quasi-projective and flat integral scheme over $k$. 

Let $\CX$ be a projective model of $\CU$ over $k$.
In the spirit of \cite[\S2.2]{YZ}, 
let $\wh\Picc(\CX)$ be the category of hermitian line bundles on $\CX$, 
and $\wh\Picc(\CX)_\QQ$ be the category of hermitian $\QQ$-line bundles on $\CX$. 
In the arithmetic case (when $k=\ZZ$), $\wh\Picc(\CX)$ is the category of hermitian line bundles with continuous metrics on $\CX$. 
In the geometric case (when $k$ is a field), $\wh\Picc(\CX)$ means the usual 
$\Picc(\CX)$.

As a convention, categories of various line bundles are defined to be groupoids; i.e., the morphisms in them are defined to be isomorphisms (or isometries) of the line bundles. 
To illustrate the category of various $\QQ$-line bundles, take $\wh\Picc(\CX)_\QQ$ for example. 
An object of $\wh\Picc(\CX)_\QQ$ is a pair $(a,\CL)$ (or just written as $a\CL$)
with $a\in \QQ$ and $\CL\in\wh\Picc(\CX)$, and a morphism of two such objects is defined to be an element of
$$\Hom(a\CL,a'\CL')=\varinjlim_m \Hom(am\CL, a'm\CL'),$$
where $m$ runs through positive integers such that $am$ and $a'm$ are both integers, so that $am\CL$ and $a'm\CL'$ are viewed as objects of $\wh\Picc(\CX)$, 
and ``$\Hom$'' on the right-hand side are viewed in $\wh\Picc(\CX)$.

Let $(\CX_0,\CEE_0)$ be a boundary divisor as above.
Define the \emph{category $\wh\Picc (\CU/k)$ of adelic line bundles} on $\CU/k$ as follows.
An object of $\wh\Picc (\CU/k)$ is a pair
$(\CL, (\CX_i,\overline \CL_i, \ell_{i})_{i\geq 1})$ where:
\begin{enumerate}[(1)]
\item $\CL$ is an object of $\Picc(\CU)$, i.e., a line bundle on $\CU$;

\item  $\CX_i$ is a projective model of $\CU$ over $k$;

\item  $\overline \CL_i$ is an object of $\wh\Picc(\CX_i)_\QQ$, i.e. a hermitian $\QQ$-line bundle on $\CX_i$;

\item $\ell_i:\CL\to \CL_i|_{\CU}$ is an isomorphism in $\Picc(\CU)_\QQ$, where $\CL_i$ is the underlying $\QQ$-line bundle of $\CLL_i$ on $\CX_i$.
\end{enumerate}
The sequence is required to satisfy the \emph{Cauchy condition} that
the sequence $\{\wh \div(\ell_i \ell_1^{-1})\}_{i\geq 1}$ is a 
Cauchy sequence in $\wh\Div(\CU/k)_\rmod$ under the boundary topology.
Here $\ell_i \ell_1^{-1}: \CL_1|_{\CU} \to \CL_i|_{\CU}$ is viewed as a rational section of the underlying line bundle of
$\CLL_i-\CLL_1$, so that $\wh \div(\ell_i \ell_1^{-1})$ is a well-defined element of $\wh\Div(\CU/k)_\rmod$. 

A morphism from an object $(\CL, (\CX_i,\overline \CL_i, \ell_{i})_{i\geq 1})$ to another 
$(\CL',(\CX_i',\overline \CL_i', \ell_{i}')_{i\geq 1})$ is an isomorphism  $\iota:\CL\to \CL'$ of the integral line bundles on $\CU$ such that the sequence 
$\{ \wh\div(\ell_i'\iota \ell_i^{-1}) \}_{i\geq1}$
of $\wh \Div (\CU/k)_\rmod$ converges to 0 in $\wh \Div (\CU/k)$
under the boundary topology.
Note that the model arithmetic divisor $\wh\div(\ell_i'\iota \ell_i^{-1})$ is equal to
$\wh\div(\ell_{i}'\ell_1'^{-1})-\wh\div(\ell_{i}\ell_1^{-1})+\wh\div(\ell_1'\iota \ell_1^{-1})
$.

An object of $\wh\Picc (\CU/k)$ is called an \emph{adelic line bundle} on $\CU$.
Define $\wh\Pic (\CU/k)$ to be the \emph{group} of isomorphism classes of objects of $\wh\Picc (\CU/k)$. 

There is a canonical forgetful functor
$$
\wh\Picc(\CU/k) \lra \Picc(\CU),\quad
(\CL, (\CX_i,\overline \CL_i, \ell_{i})_{i\geq 1})\longmapsto \CL
$$ 
We usually write $\CLL=(\CL, (\CX_i,\overline \CL_i, \ell_{i})_{i\geq 1})$ and call $\CL$ the \emph{underlying line bundle} of $\CLL$.

There is a canonical surjection
$$\wh\Div(\CU/k)\lra \wh\Pic(\CU/k), \quad
\CDD\longmapsto \CO(\CDD).$$
The kernel is the image of the group of principal arithmetic divisors on projective models of $\CU$.

An adelic  divisor is called \emph{effective} if it is equal to a limit of effective arithmetic divisors (of mixed coefficients). 
An adelic line bundle is called \emph{effective} if it is the image of an effective adelic divisor. 
An adelic line bundle (resp. adelic divisor) is called \emph{strongly nef} if it is isomorphic (resp. equal) to a limit of nef hermitian line bundles (resp. adelic divisors) under the boundary topology. 
An adelic line bundle  (resp. adelic divisor) $\CLL$ on $\CU$ is \emph{nef} if there exists a strongly nef adelic line bundle (resp. adelic divisor) $\CMM$ on $\CU$ such that $a\CLL+\CMM$ is strongly nef for all positive integers $a$.
An adelic line bundle (resp. adelic divisor) is \emph{integrable} if it is 
isomorphic (resp. equal) to the difference of two strongly nef ones.

Denote by $\wh\Picc(\CU/k)_\nef$ (resp. $\wh\Picc(\CU/k)_\intb$) the subcategory of 
$\wh\Picc(\CU/k)$ consisting of nef (resp. integrable) adelic line bundles.
Denote by $\wh\Pic(\CU/k)_\nef$ (resp. $\wh\Pic(\CU/k)_\intb$) the subsets of nef (resp. integrable) elements in $\wh\Pic(\CU/k)$.

If $\pi:\CU'\to \CU$ is a morphism of quasi-projective and flat integral schemes over $k$, then there are functorial pull-back maps
$$
\pi^*: \wh\Picc(\CU/k) \lra \wh\Picc(\CU'/k), \qquad
\pi^*: \wh\Pic(\CU/k) \lra \wh\Pic(\CU'/k).
$$
The maps send nef (resp. integrable) adelic line bundles to nef (resp. integrable) adelic line bundles. 

If the base ring $k$ is clear, we usually omit the dependence on $k$ of the groups or categories of adelic objects. For example,  
$\wh\Div(\CU/k)$, $\wh\Picc(\CU/k)$, $\wh\Pic(\CU/k)$ are written as 
$\wh\Div(\CU)$, $\wh\Picc(\CU)$, $\wh\Pic(\CU)$.

If $k$ is a field, to emphasize the geometric situation, we may also write 
$\wh\Div(\CU/k)$, $\wh\Picc(\CU/k)$, $\wh\Pic(\CU/k)$ as 
$\wt\Div(\CU/k)$, $\wt\Picc(\CU/k)$, $\wt\Pic(\CU/k)$.

\subsubsection{Intersection theory}

The intersection theory in \cite[\S4.1]{YZ} introduced absolute intersection numbers and Deligne pairings in the relative situation.
Let us first review the absolute case. 

\kkk
Let $\CU$ be a quasi-projective and flat integral scheme over $k$ of absolute dimension $d$. 
There is an absolute intersection pairing
$$\wh\Pic (\CU/k)_\intb^{d}\lra \BR, \qquad  (\CHH_1, \cdots, \CHH_{d})\longmapsto\CHH_1 \cdots \CHH_{d},$$
defined as limits of the arithmetic (or geometric) intersection pairings {on} projective models.
In particular, the top intersection numbers of nef adelic line bundles are {nonnegative}.

If $d=1$, we actually have $\wh\Pic (\CU/k)_\intb=\wh\Pic (\CU/k)$. In this case, the intersection pairing is just a degree map 
$$\wh\deg: \wh\Pic (\CU/k)\lra \BR.$$

\subsubsection{Deligne Pairing}
\label{sec Deligne pairing}

To recall the Deligne pairing for adelic line bundles, we first consider more classical situations.

Let $f:X\to Y$ be a finite and flat morphism of noetherian schemes. 
Then $f_*\CO_X$ is locally free over $\CO_Y$ and thus there are norm maps 
$f_*\CO_X\to \CO_Y$ and $f_*(\CO_X^\times)\to \CO_Y^\times$. 
To define the map, it suffices to consider the local situation that $X=\Spec A$ and $Y=\Spec B$, where $A$ is a free $B$-module of finite rank. Then 
the norm of an element $a\in A$ is just the determinant of the multiplication map $a:A\to A$ of $B$-modules. 
This eventually gives a canonical norm functor 
$$
N_{X/Y}: \Picc(X)\lra \Picc(Y). 
$$
This is introduced in detail in \cite[II, \S6.5]{EGA}. 

The Deligne pairing can be viewed as a high-dimensional generalization of the norm functor. To illustrate an idea, we consider the nice situation that 
$f:X\to Y$ is a projective and flat morphism of relative dimension $n\geq 0$ of smooth  varieties $X$ and $Y$ over a field $k$.
Then there is a composition 
$$
\Pic(X)^{n+1}\lra \mathrm{CH}^{n+1}(X) \stackrel{f_*}\lra \mathrm{CH}^1(Y) \stackrel{\sim}\lra \Pic(Y).
$$
We will see that the Deligne pairing is a functorial pairing refining this composition significantly. 

Now we relax the condition drastically. 
Let $f: X\to Y$ be a projective and flat morphism of noetherian schemes of pure relative dimension $n$. The Deligne pairing is a multi-linear functor
$$\Picc (X)^{n+1}\lra \Picc (Y), \quad
(L_1,\cdots,L_{n+1})
\longmapsto f_*\pair{L_1, \cdots, L_{n+1}}.
$$
The functor refines the above two {constructions}. It satisfies many natural functorial properties, including the base change property, the multi-linearity, the symmetry, and the induction formula. 
The existence of the Deligne pairing for arbitrarily singular $Y$ is a miracle in some sense.

For a brief history of {the pairing}, the case $n=0$ is just the norm functor $N_{X/Y}$. 
Deligne \cite{Del} constructed the functor for $n=1$ and speculated a similar pairing for general $n$. 
Deligne's major motivation is to formulate an arithmetic Riemann--Roch theorem for families of curves. 
For general $n$, the pairing was constructed by Elkik \cite{Elk1} for any $f$ which is projective, flat, and further Cohen--Macaulay,  by {Mu\~noz Garcia} \cite{MG} for any $f$ which is projective, equi-dimensional and of finite Tor-dimension (which implies the projective and flat case), and by 
Shiquan Li \cite{Li24}  for any $f$ which is projective and equi-dimensional. 

If $X$ and $Y$ are smooth varieties over $\CC$ and $f$ is smooth, and if $L_1,\cdots,L_{n+1}$ are endowed with smooth hermitian metrics, then the metrics transfer to a canonical smooth hermitian metric on 
$\pair{L_1, \cdots, L_{n+1}}$, as constructed by Deligne \cite{Del} and Elkik \cite{Elk2}. 
The metric construction was further generalized to the projective and flat case by \cite{YZ}, and in this case, the Deligne pairing transfers continuous metrics to continuous metrics. 
Finally, a further limit process by \cite{YZ} transfers the construction to adelic line bundles as follows.

\kkk
Let $f:\CU\to \CV$ be a projective flat morphism of relative dimension $n$ of 
quasi-projective and flat integral schemes over $k$. 
Assume that $\CV$ is normal.
There is a relative intersection pairing
$$f_*:\wh \Picc (\CU/k)_\intb^{n+1}\lra \wh\Picc (\CV/k)_\intb, \qquad 
(\CLL_1, \cdots, \CLL_{n+1})\longmapsto f_*\pair{\CLL_1, \cdots, \CLL_{n+1}}.
$$
This is defined as the limit of the Deligne pairing.
Moreover, the pairing is compatible with base changes of the form $\CV'\to\CV$, where $\CV'$ is a normal quasi-projective and flat integral {scheme} over $k$ 
such that $\CU'=\CU\times_\CV\CV'$ is also integral.

The pairing of nef adelic line bundles is still nef. 
For simplicity, we may abbreviate $f_*\pair{\CLL_1, \cdots, \CLL_{n+1}}$
as $\pair{\CLL_1, \cdots, \CLL_{n+1}}$ if no confusion can occur.

Finally, the Deligne pairing refines the absolute intersection pairing 
in the sense that if $\dim \CV=1$, then the composition 
$$
\wh \Picc (\CU/k)_\intb^{n+1}
\stackrel{f_*}{\lra} \wh\Picc (\CV/k)_\intb
\stackrel{\wh\deg}{\lra}\RR
$$
is equal to the absolute {intersection} pairing.

\subsubsection{Quasi-projective varieties over number fields}

By further direct limits, the above definitions and notations are extended to 
flat and essentially quasi-projective integral schemes over $k$ in \cite{YZ}. 
The notion of ``essentially quasi-projective scheme'' is introduced in \cite[\S2.3]{YZ}, which can be realized as an intersection of (possibly infinitely many) open subschemes of a quasi-projective scheme.
We do not need this generality in this paper, but we only need the following case of quasi-projective varieties over a number field.

Here we restrict to the arithmetic case $k=\ZZ$. 
Let $K$ be a number field.
Let $X$ be a quasi-projective variety over $K$, which is viewed as a scheme over $\ZZ$. 
By a \emph{quasi-projective model} $\CU$ of $X$ over $\ZZ$, we mean a quasi-projective and flat integral scheme $\CU$ over $O_K$, together with an open immersion $X\to \CU_K$
over $K$.
Define 
$$
\wh\Div(X/\ZZ)=  \varinjlim_{\CU} \wh\Div(\CU/\ZZ),$$
$$
\wh\Pic(X/\ZZ)=  \varinjlim_{\CU} \wh\Pic(\CU/\ZZ),
$$ 
$$
\wh\Picc(X/\ZZ)=  \varinjlim_{\CU} \wh\Picc(\CU/\ZZ).$$
Here the limits are over all quasi-projective models $\CU$ of $X$ over $\ZZ$.
For the last direct limit, an object of $\wh\Picc(X/\ZZ)$ is a pair $(\CLL, \CU)$, where $\CU$ is a quasi-projective model of $X$ over $\ZZ$ and $\CLL$ is an object of $\wh\Picc(\CU/\ZZ)$. 
A morphism $(\CLL, \CU)\to (\CLL', \CU')$ between two objects of $\wh\Picc(X/\ZZ)$ is an isomorphism $\iota: \CL|_X\to \CL'|_X$ in $\Picc(X)$ satisfying the property that for some quasi-projective model $\CV$ of $X$ over $\ZZ$ endowed with open immersions $\psi:\CV\to \CU$ and $\psi':\CV\to \CU'$ extending the identity morphism $X\to X$, the isomorphism $\iota: \CL|_X\to \CL'|_X$ can be extended to an isomorphism $\CL|_\CV\to \CL'|_\CV$ in $\Picc (\CV)$ and induces
an isomorphism $\CLL|_\CV\to \CLL'|_\CV$ in $\wh \Picc (\CV/\ZZ)$. 
Here we take the convention $\CL|_X=(\CL|_\CU)|_X$, and
if $\CLL=(\CL, (\CX_i,\overline \CL_i, \ell_{i})_{i\geq 1})$ in $\wh\Picc (\CU/\ZZ)$, then $\CLL|_\CV=(\CL|_\CV, (\CX_i,\overline \CL_i, \ell_{i}|_\CV)_{i\geq 1})$ in 
$\wh\Picc (\CV/\ZZ)$.
 
An element of $\wh\Div(X/\ZZ)$ is called an \emph{adelic divisor} on $X/\ZZ$;
an object of $\wh\Picc(X/\ZZ)$ is called an \emph{adelic line bundle} on $X/\ZZ$. 
Again, we often abbreviate 
$$\wh\Div(X/\ZZ), \quad
\wh\Picc(X/\ZZ), \quad
\wh\Pic(X/\ZZ)$$  
as 
$$\wh\Div(X), \quad
\wh\Picc(X), \quad
\wh\Pic(X).$$ 

The notions of effectivity, nefness and integrability of adelic line bundles (or adelic divisors) are also transferred to quasi-projective varieties over $K$ by taking direct limits. 
The intersection pairing
$$\wh\Pic(X)_{\intb} ^{\dim X+1}\lra \RR$$
is valid in the current situation by taking direct limits.
Similarly, the Deligne pairing 
$$\wh\Picc(X)_{\intb} ^{n+1}\lra \wh\Picc(Y)_{\intb}$$
is valid for any projective and flat morphism $f:X\to Y$ of quasi-projective $K$-varieties of relative dimension $n$.

There is a functorial map 
$$
\wh\Pic(X/\ZZ) \lra \wt\Pic(X/K), \quad \OL\longmapsto \wt L. 
$$
The image $\wt L$ is called the \emph{geometric part} of $\OL$. 
The map is induced via a limit process by the natural map
$$
\wh\Pic(\CX) \lra \Pic(\CX_K), \quad \CLL\longmapsto \CL_K. 
$$
for arithmetic varieties $\CX$ over $O_K$.

\subsection{Analytification by Berkovich spaces}

{Adelic line bundles have} an interpretation in terms of metrized line bundles on Berkovich spaces. This subsection sketches this interpretation. 

\subsubsection{Berkovich spaces}

Berkovich spaces are best known as analytic spaces associated with varieties over non-archimedean fields, whose foundation was introduced by Berkovich \cite{Ber1}. 
By Berkovich \cite[\S1]{Ber2}, the base fields are relaxed to be Banach rings, and the old construction works similarly. 
The {latter} case plays an important role in the theory of adelic line bundles in \cite{YZ}.
Here we review the basics of Berkovich spaces following the terminology of 
\cite{YZ}. 

Let $(R,|\cdot|_\mathrm{Ban})$ be a commutative Banach ring. Namely, $R$ is a commutative ring with unity 1, and $|\cdot|_\mathrm{Ban}:R\to \RR_{\geq0}$ is a map satisfying the following properties:
\begin{enumerate}[(1)]
\item (norm property) $|a|_\mathrm{Ban}>0$ for all nonzero $a\in R$;
\item (triangle inequality) $|a+b|_\mathrm{Ban}\leq |a|_\mathrm{Ban}+|b|_\mathrm{Ban}$ for all  $a,b\in R$;
\item (sub-multiplicativity) $|ab|_\mathrm{Ban}\leq |a|_\mathrm{Ban}\cdot |b|_\mathrm{Ban}$ for all  $a,b\in R$;
\item (completeness) $R$ is complete under the topology induced by $|\cdot|_\mathrm{Ban}$.
\end{enumerate}
In our applications, we will always have one of the following 4 cases:
\begin{enumerate}[(a)]
\item $R$ is an archimedean field, i.e. $R$ is isomorphic to $\RR$ or $\CC$ with the standard absolute value;
\item $R$ is a non-archimedean field, i.e. a complete field with a non-trivial non-archimedean valuation;
\item $R$ is a constant field,  i.e., $R$ is a field and $|\cdot|_\mathrm{Ban}$ is the trivial valuation $|\cdot|_0$;
\item $R=\ZZ$, and $|\cdot|_\mathrm{Ban}$ is the usual archimedean absolute value $|\cdot|_\infty$.
\end{enumerate}
To refer to the last two cases, in the uniform terminology of \cite[\S1.5]{YZ}, we will always say that \emph{let $k$ be either $\ZZ$ or a field}.

Let $A$ be a commutative ring over $R$. 
Then the 
\emph{Berkovich space} $\CM(A/R)$, sometimes abbreviated as $\CM(A)$,  is the set of multiplicative semi-norms on $A$ whose restriction to $R$ is bounded by 
$|\cdot|_\mathrm{Ban}$. 
Namely, a point $x\in \CM(A/R)$ is given by a map $|\cdot|_x: A\to \RR_{\geq 0}$ satisfying:
\begin{enumerate}[(1)]
\item (boundedness) $|a|_x\leq |a|_\mathrm{Ban}$ for any $a\in R$,
\item (triangle inequality) $|f+g|_x \leq |f|_x+|g|_x, \ \forall f,g\in A$,
\item (multiplicativity) $|fg|_x= |f|_x\cdot |g|_x, \ \forall f,g\in A$.
\end{enumerate}
Any $f\in A$ induces a map
$$|f|: \CM(A/R) \longrightarrow \RR, \quad x\longmapsto |f|_x.$$
Endow $\CM(A/R)$ with the coarsest topology such that $|f|$
is continuous for all $f\in A$.


Let $X$ be a scheme over $R$. 
Assume that $X$ is covered by an affine open cover $\{\Spec A_i\}_i$.
Then the \emph{Berkovich space} $(X/R)^\an$, sometimes abbreviated as $X^\an$, is the union of $\CM(A_i/R)$, glued canonically. 
The topology of $(X/R)^\an$ is the weakest one such that each $\CM(A_i/R)$ is an open subspace of $(X/R)^\an$.

Note that in the affine case $X=\Spec A$, the definition gives 
$(X/R)^\an=\CM(A/R)$. 
We have the following extra concepts.

\begin{enumerate}[(1)]

\item \emph{Residue field.}
For each $x\in \CM(A/R)$, the corresponding semi-norm $|\cdot|_x$ induces a norm on the integral domain $A/\ker(|\cdot|_x)$. The completion of the fraction field of  $A/\ker(|\cdot|_x)$ is called the \emph{residue field} of $x$ and denoted by $\CH_x$. Denote by $|\cdot|$ the valuation (multiplicative norm) on $\CH_x$ induced by $|\cdot|_x$. 
Then $|\cdot|_x:A\to \RR$ is equal to the composition 
$$A\lra \CH_x\overset{|\cdot|}{\lra} \BR.$$
We write the first map as $f\mapsto f(x)$, which is compatible with the convention $|f|_x=|f(x)|$.
The notation $\CH_x$ generalizes to any scheme $X$ over $R$.

\item \emph{Contraction.}
There is a canonical contraction map $\kappa: (X/R)^\an\to X$. It suffices to describe it in the case $X=\Spec A$. 
For each $x\in \CM(A)$, 
the kernel of the map $|\cdot|_x:A\to \RR$ is a prime ideal of $A$, and thus defines an element $\kappa(x)\in \Spec A$. 

\item \emph{Functoriality.}
Any morphism $\phi:X\to Y$ over $R$ induces a continuous map 
$$\phi^\an:(X/R)^\an\lra (Y/R)^\an.$$ 
For any point $v\in Y^\an$, the \emph{fiber} 
$$X_v^\an=(X/R)_v^\an=(\phi^\an)^{-1}(v),$$
 is a subspace of $(X/R)^\an$, and is canonically homeomorphic to the Berkovich space $(X_{\CH_v}/\CH_v)^\an$.
\end{enumerate}
By \cite[Lem. 1.1, Lem. 1.2]{Ber2}, we have the following basic topological properties: 
\begin{enumerate}[(1)]
\item If $X$ is separated and of finite type over $R$, then $(X/R)^\an$ is Hausdorff.
\item If $X$ is of finite type over $R$, then $(X/R)^\an$ is locally compact.
\item If $X$ is projective over $R$, then $(X/R)^\an$ is compact.
\end{enumerate}

In the definition, we do not assume that $X$ is reduced. 
This does not bring much trouble, since there is always a canonical 
homeomorphism $(X_{\red}/R)^\an\to (X/R)^\an$, where $X_\red$
denotes the reduced structure of $X$. 
Moreover, we can further write 
$(X_{\red}/R)^\an$ as a union of the Berkovich spaces associated to the irreducible components of $X_{\red}$.

Assume that $R$ is an archimedean field as in case (a), and assume that $X$ is quasi-projective over $R$. 
If $R=\CC$, then $(X/\CC)^\an$ is homeomorphic to the complex analytic space $X(\CC)$. 
If $R=\RR$, then $(X/\RR)^\an$ is  homeomorphic to the quotient of the complex analytic space $X(\CC)$ by the action of the complex conjugation.

\subsubsection{Arithmetic divisors and metrized line bundles}

Let $X$ be an integral scheme over commutative Banach ring $R$.
Let $X^\an=(X/R)^\an$ be the Berkovich space defined above.

Let $D$ be a Cartier divisor on $X$. By a {\em Green function} of the divisor $D$ on $X^\an$, we mean a continuous function $g: X^\an\setminus |D|^\an \to \RR$ with logarithmic singularity along $D$ in the sense that, for any rational function $f$ on a Zariski open subset $U$ of $X$ satisfying $\div(f)=D|_U$, the function $g+\log |f|$ can be extended to a continuous function on $U^\an$.

The pair $\overline D=(D, g)$ is called an {\em arithmetic divisor} on $X^\an$.  
An {arithmetic divisor} is called {\em effective} if $D$ is an effective Cartier divisor on $X$ and $g\geq 0$ on $X^\an\setminus |D|^\an$.
An {arithmetic divisor} is called {\em principal} if it is of the form 
$$\wh\div_{X^\an}(f):=(\div(f), -\log |f|)$$ 
for some nonzero rational function $f$ on $X$.

Denote by $\wh \Div (X^\an)$ the group of arithmetic divisors on $X^\an$, and by $\wh \Pr (X^\an)$ the group of {principal arithmetic divisors} on $X^\an$.
Denote the class group of arithmetic divisors as
$$\wh \CaCl (X^\an):=\wh \Div (X^\an)/\wh \Pr (X^\an).$$

Notice that for any arithmetic divisor $\overline D=(D, g)$ on $X^\an$, 
the algebraic part $D$ is a Cartier divisor on $X$ (instead of $X^\an$), and $g$ is a function on $X^\an$. 
We take this ad hoc definition to avoid defining general Cartier divisors on $X^\an$ {because of} the lack of a good theory of analytic functions on $X^\an$.

We can also define metrized line bundles. 
Let $L$ be a line bundle on $X$. 
At each point $x\in X^\an$, denote by $\bar x$ the image of $x$ in $X$. 
The \emph{fiber} $L^\an(x)$ of $L$ at $x$ is defined to be the $\CH_x$-line $L(\bar x)\otimes_{k(\bar x)} \CH_x$, or equivalently the completion of the fiber $L(\bar x)$ of $L$ on $\bar x$ for the  semi-norm $|\cdot|_x$.
By a \emph{metric} $\|\cdot \|$ of $L$ on $X^\an$ we mean a continuous metric on $\coprod_{x\in X^\an} L^\an(x)$ compatible with the semi-norms
 on $\CO_X$. More precisely, to each point $x\in X^\an$, we assign a norm $\|\cdot \|_x$ on the $\CH_x$-line
 $L^\an(x)$ which is compatible with the norm $|\cdot |_x$ of $\CH_x$ in the sense that
 $$\|f\ell\|_x=|f|_x\cdot \|\ell\|_x, \qquad f\in \CH_x, \quad \ell \in L^\an(x).$$
 We always assume that the metric $\|\cdot \|$ on $L$ is {\em continuous} in the sense that, for any section $\ell$ of $L$ on a Zariski open subset $U$ of $X$, the function $\|\ell (x)\|=\|\ell (x)\|_x$ is continuous in $x\in U^\an$.
 
The pair $(L,\|\cdot\|)$ above is called a \emph{metrized line bundle} on $X^\an$.
An \emph{isometry} from a metrized line {bundle} $(L,\|\cdot\|)$ to another one $(L',\|\cdot\|')$ is an isomorphism $i:L\to L'$ of line bundles on $X$ such that $\|\cdot\|=i^*\|\cdot\|'$.

Denote by $\wh \Picc (X^\an)$ the \emph{category} of metrized line bundles on $X^\an$, where the morphisms are isometries. 
Denote by $\wh \Pic (X^\an)$ the \emph{group} of isometry classes of metrized line bundles on $X^\an$.

There is a canonical isomorphism 
$$\wh \CaCl (X^\an) \lra \wh \Pic (X^\an).$$
In fact, given any arithmetic divisor, $(D, g)$ on $X^\an$, the term $e^{-g}$ defines a metric on $\CO (D)$, and thus we obtain a metrized line bundle on $X^\an$.
Conversely, for any metrized line bundle $(L,\|\cdot \|)$ on $X^\an$, if $s$ is a rational section of $L$, then 
$$\wh\div_{X^\an}(s):=(\div(s),-\log \|s\|)$$ 
defines an arithmetic divisor on $X^\an$. 
Both processes keep the properties of being norm-equivariant.

\subsubsection{Analytification map}

\kkk
If $k=\ZZ$, view it as a Banach ring under the usual archimedean absolute value; if $k$ is a field, view it as a Banach ring with the trivial norm.
Let $\CU$ be a quasi-projective and flat integral scheme over $k$. 
We have the {Berkovich analytic space} $\CU^\an=(\CU/k)^\an$ as above.
It is Hausdorff, path-connected and locally compact.

By \cite[Prop. 3.3.1, Prop. 3.4.1]{YZ}, there are injective analytification maps
$$
\wh\Div(\CU/k)\lra \wh\Div(\CU^\an),
$$
$$
\wh\Pic(\CU/k)\lra \wh\Pic(\CU^\an),
$$
and a fully faithful analytification functor
$$
\wh\Picc(\CU/k)\lra \wh\Picc(\CU^\an).
$$
Note that the left-hand sides are the adelic objects defined as limits of model objects, and the right-hand sides are objects on the analytic space $\CU^\an$. 

To introduce the idea, we only explain the analytification map
$$
\wh\Picc(\CU/k)\lra \wh\Picc(\CU^\an).
$$
We first consider the case that $\CU=\CX$ is projective over $k$. 
Let $\CLL$ be a hermitian line bundle on $\CX$.
We need to define a metric of $\CL$ on $\CX^\an$. 
{The metrics} of the fibers of $\CL$ at the archimedean points $x\in\CX^\an$ are given by the original hermitian metric.
For the metric of $\CL$ at a non-archimedean point $x\in \CX^\an$,
let $\phi_x^\circ:\Spec O_{\CH_x}\to \CX$ be the $k$-morphism extending the $k$-morphism 
$\phi_x:\Spec \CH_x\to \CX$ under the valuative criterion. 
Then $(\phi_x^\circ)^*\CL$ is a free module over $O_{\CH_x}$ of rank 1.
Let $s_x$ be {a basis} of this free module. 
Define the metric of $\CL(x)=\phi_x^*\CL$ by setting $\|s_x\|=1$.
It takes some {effort} to prove that this metric is continuous on $\CX^\an$. 

If $\CU$ is quasi-projective over $k$, the functor is obtained by approximation as follows. 
Recall that an object of $\wh\Picc (\CU/k)$ is a Cauchy sequence $\CLL=(\CL,(\CX_i,\overline \CL_i, \ell_{i})_{i\geq1})$. 
Note that each $\overline\CL_i$ induces a metric $\|\cdot\|_i^*$ of $\CL_{i}$ on $\CX_i^\an$ by the above projective case.
By the isomorphism $\ell_{i}:\CL\to \CL_i|_\CU$, and by restriction, we get a metric $\|\cdot\|_i$ of $\CL$ on $\CU^\an$.  
The Cauchy condition implies that these metrics converge pointwise to a continuous metric $\|\cdot\|$ of $\CL$ on $\CU^\an$. 
Then $\CLL^\an:=(\CL,\|\cdot\|)$ defines an object of 
$\wh \Picc(\CU^\an)$, which is the desired image of the functor.

\subsubsection{Example: invariant adelic line bundles} 
\label{sec adelic dynamics}

A \emph{polarized algebraic dynamical system} (or \emph{dynamical system}) over a noetherian scheme $S$ consists of a triple $(X,f,L,q)$ where:
\begin{enumerate}[(1)]
\item $X$ is a projective and flat scheme with integral fibers over $S$,
\item $f:X\to X$ is an algebraic morphism over $S$,
\item $L$ is a line bundle on $X$, relatively ample over $S$,  and polarizing $f$ in the sense that $f^*L\cong L^{\otimes q}$ for some integer $q>1$. 
\end{enumerate}
While the case $S=\Spec F$ for a field $F$ gives a single dynamical system over $F$, the general case gives an algebraic family of dynamical systems. 

Now we consider our adelic situation. 
\kkk
Let $S$ be a flat and quasi-projective integral scheme over $k$. 
Let $(X,f,L,q)$ be a dynamical system over $S$. 
Fix an isomorphism $\tau:f^*L\to L^{\otimes q}$. 
In \cite[\S6.1.1]{YZ}, Yuan--Zhang constructed a canonical extension of $L$ to an adelic line bundle $\OL_f$ over $X/k$ which is 
\emph{$f$-invariant} in the sense that the isomorphism $\tau:f^*L\to L^{\otimes q}$
extends to an isomorphism
$f^* \overline L_f\to \overline L_f^{\otimes q}$ in $\wh \Picc(X/k)$.
Moreover, 
$\OL_f$ is nef over $X/k$. 

To illustrate the idea, choose a projective model $\pi:\CX\rightarrow\CS$ of $X\to S$, i.e. a projective model $\CS$ of $S$ over $k$ and a flat morphism 
$\pi:\CX\rightarrow\CS$ of projective varieties over $k$ whose base change by $S\to \CS$ is isomorphic to $X\to S$. 
Choose a hermitian line bundle
$\overline\CL=(\CL, \|\cdot\| )$ on $\CX$ such that $(\CX_S, \CL_S)\simeq(X,L)$.

For each positive integer $i$, consider the composition $X \stackrel{f^i}{\rightarrow} X \rightarrow \CX$. Denote the normalization of the composition by $f_i: \CX_i\to \CX$, and denote the induced map to $\CS$ by $\pi_i: \CX_i\to \CS$.
Denote $\overline \CL_i =q^{-i} f_i^* \overline\CL$, which lies in 
$\wh\Picc(\CX_i)_{\BQ}$.
With some extra effort, we can complete the 
 sequence $\{(\CX_i, \overline \CL_i)\}_{i\geq 1}$ to an adelic line bundle 
$\OL_f=(\CL_\CV,(\CX_i, \overline \CL_i,\ell_i)_{i\geq 1})$ on a quasi-projective model $\CU$ of $X$ over $k$.

Denote by $\OL_f^\an=(L, \|\cdot\|_f)$ the analytification of $\OL_f$, which is a metrized line bundle on $X^\an$. 
We can also describe the metric $\|\cdot\|_f$ of $L$ on $X^\an$. 
In fact, the isomorphism $\tau:f^*L\to L^{\otimes q}$ induces an isomorphism $f^* \overline L_f\to \overline L_f^{\otimes q}$, which in turn induces an isometry 
$$
\tau: f^*(L, \|\cdot\|_{f})\to (L, \|\cdot\|_{f})^{\otimes q}.
$$
This uniquely determines the metric $\|\cdot\|_f$. 
Therefore, we can construct $\|\cdot\|_{f}$ via Tate's limit
$$
\|\cdot\|_{f}= \lim_{n\to \infty} \left((\tau \circ f^*)^n\|\cdot\|_{0}\right)^{\frac{1}{q^n}},
$$
where $\|\cdot\|_{0}$ is any continuous metric of $L$ on $X^\an$. 
To see the convergence and the uniqueness of the limit, it suffices to consider the restriction of the metric to 
the fiber $X_s^\an$ of $X^\an$ over every $s\in S^\an$. 
Here $X_s^\an$ is the Berkovich space of $X\times_S \Spec \CH_s$ over the valuation field $\CH_s$. 
The problem is reduced to the dynamical system $(X_{\CH_s}, f_{\CH_s}, L_{\CH_s},q)$ over the valuation field ${\CH_s}$.
This can be treated as in \cite[\S4]{Yua12} or \cite[Lem. 8.2.3]{YG}.

\subsubsection{Quasi-projective varieties over number fields}

Let $K$ be a number field. 
Denote by $M_K$ the set of places of $K$. We normalize the absolute value $|\cdot|_v$ for every $v\in M_K$ as in \S\ref{sec notation}.

Let $X$ be a quasi-projective variety over $K$. 
Recall that the adelic divisors and adelic line bundles on $X/\ZZ$ are defined as {direct limits} of the corresponding objects on 
quasi-projective models $\CU$ of $X$ over $\ZZ$.
Then we still have injective analytification maps
$$
\wh\Div(X/\ZZ)\lra \wh\Div(X^\an),
$$
$$
\wh\Pic(X/\ZZ)\lra \wh\Pic(X^\an),
$$
and a fully faithful analytification functor
$$
\wh\Picc(X/\ZZ)\lra \wh\Picc(X^\an).
$$
Here the analytic space 
$X^\an=(X/\ZZ)^\an$ is defined as above.

Recall that $\CM(\ZZ)=(\Spec \ZZ)^\an$ is the set of all multiplicative semi-norms of $\ZZ$, since the boundedness condition is automatic in this case. 
This space is too large and contains many ``redundant'' points, and as a result $X^\an$ is also too large. Now we introduce the restricted versions.

Define the \emph{restricted analytic space $X^\ran =(X/K)^\ran$ associated to} $X/K$ to be the disjoint union
$$
X^\ran =\coprod_{v\in M_K} X_v^\an,
$$
where $X_v^\an=(X_{K_v}/K_v)^\an$ is 
the Berkovich space associated to $X_{K_v}$ over the complete field $K_v$.
The topology on $X^\ran$ is induced by the disjoint union so that 
each $X_v^\an$ is both open and closed in $X^\ran$.
There is a natural injection $X^\ran\to X^\an$, under which 
$X^\ran$ is viewed as a closed subspace of $X^\an$.

Define an \emph{arithmetic divisor} on $X^\ran$ to be a pair $(D,g_D)$, where $D$ is a Cartier divisor on $X$, and $g_D$ is a \emph{Green function} of $D$ on $X^\ran$, i.e. a continuous function $g:X^\ran\setminus |D|^\ran \to \RR$ with logarithmic singularity along $D$ in the sense that, for any rational function $f$ on a Zariski open subset $U$ of $X$ satisfying $\div(f)=D|_U$, the function $g+\log |f|$ can be extended to a continuous function on $U^\ran$.

An {arithmetic divisor} on $X^\ran$ is called {\em principal} if it is of the form 
$$\wh\div_{X^\ran}(f):=(\div(f), -\log |f|)$$ 
for some nonzero rational function $f$ on $X$.

Denote by $\wh\Div(X^\ran)$ (resp. $\wh\Pr(X^\ran)$) the \emph{group} of arithmetic divisors (resp. principal arithmetic divisors) on $X^\ran$. 
Define 
$$
\wh\CaCl(X^\ran):=\wh\Div(X^\ran)/\wh\Pr(X^\ran). 
$$

Define a \emph{metrized line bundle} on $X^\ran$ to be a pair $(L,\|\cdot\|)$, where $L$ is a line bundle on $X$, and $\|\cdot\|$ is a continuous metric {on the fibers} of $L$ on $X^\ran$. 
This is similar to the original case; we omit the details.

Denote by $\wh\Pic(X^\ran)$ the \emph{group} of isometry classes of metrized line bundles on $X^\ran$, and $\wh\Picc(X^\ran)$ the \emph{category} of metrized line bundles on $X^\ran$ whose morphisms are isometries.
There is a canonical isomorphism
$$
\wh\CaCl(X^\ran)\lra \wh\Pic(X^\ran). 
$$

By the natural injection $X^\ran\to X^\an$, we have canonical maps 
$$\wh\Div (X^\an) \lra \wh \Div (X^\ran),$$
$$\wh\Pic (X^\an) \lra \wh \Pic (X^\ran),$$
$$\wh \Picc (X^\an)\lra \wh\Picc (X^\ran).$$
By composition, we obtain analytification maps
$$\wh\Div (X) \lra \wh \Div (X^\ran),$$
$$\wh\Pic (X) \lra \wh \Pic (X^\ran),$$
$$\wh \Picc (X)\lra \wh\Picc (X^\ran).$$
These maps are still injective by \cite[Prop. 3.5.1]{YZ}. 

In explicit terms, a metrized line bundle on $X^\ran$ is a pair 
$$\OL=(L,   (\|\cdot\|_v)_{v\in M_K})$$
where $L$ is a line bundle on $X$, and $\|\cdot\|_v$ is a continuous metric of $L_{K_v}$ on $X_v^\an$.
This interpretation is close to the original notion of adelic line bundles on projective varieties developed by Zhang \cite{Zha2}. 
We refer to \cite[\S3.5]{YZ} for more details of this connection, and an explanation of the equivalence of the notions for projective $X$.  

We have the following nice result derived from the definitions.

\begin{lemma} \label{iccm2010:degree2}
Suppose $X=\Spec(K)$ for a number field $K$. Let $\OL$
be an adelic line bundle on $X$. 
Then $\OL$ induces a metrized line bundle 
$(L, (\|\cdot\|_v)_{v\in M_K})$ on $X^\ran$. 
Note that the line bundle $L$ on $X$ is just a vector space over $K$ of dimension one.
We simply have 
$$
\wh{\deg}(\overline L)= -\sum_{v\in M_K}\epsilon_v \log \|s\|_v. 
$$
Here $s$ is any non-zero element of $L$, and the degree is independent of the choice of $s$ by the product formula. 
\end{lemma}

\section{Arithmetic positivity}
\label{sec positivity}

In algebraic geometry, the linear series of a line bundle is closely related to positivity properties of the line bundle in intersection theory. The involved positivity notions of line bundles include effectivity, ampleness, nefness, and bigness. We refer to Lazarsfeld \cite{Laz04, Laz04b} for a thorough introduction in this area. 

Most of these positivity results can be proved (by much more effort) for hermitian line bundles on arithmetic varieties, and further for adelic line bundles on quasi-projective varieties. In this section, we are going to review some of these positivity results. Their applications will be given in the next sections.

\subsection{Positivity in algebraic geometry}

We first recall some basic facts on positivity in algebraic geometry. 
The basic references for this topic are  Debarre \cite{Deb01} and
Lazarsfeld \cite{Laz04, Laz04b}. 

Throughout this subsection, let $X$ be a projective variety of dimension $d\geq 1$ over a field $k$. 
Denote by $\Pic(X)$ the group of isomorphism classes of line bundles on $X$. 
We have an intersection pairing
\[
\Pic(X)^d \longrightarrow \ZZ.
\]
For a closed subvariety $Y$ of dimension $e$ in $X$ and for 
line bundles $L_1,\dots, L_e$ on $X$, 
we usually write 
$$
L_1\cdot L_2\cdots L_e\cdot Y=   (L_1|_Y)\cdot (L_2|_Y) \cdots (L_e|_Y). 
$$
If $L_1= L_2=\cdots =L_e=L$, then the left-hand side is further written as $L^e\cdot Y$. 

\subsubsection{Ampleness and nefness}

Let $L$ be a line bundle on $X$. Denote $h^0(X,L)=\dim_k H^0(X,L)$. 
We say that $L$ is \emph{effective} if $h^0(X,L)>0$; i.e. $L$ is linearly equivalent to an effective Cartier divisor on $X$.

We say that $L$ is \emph{ample} if for every coherent sheaf $F$ on $X$, the sheaf $F\otimes L^{\otimes n}$ is globally generated for all sufficiently large $n$.  The Nakai--Moishezon criterion asserts that  $L$ is ample if and only if
\[
L^{\dim Y}\cdot Y>0
\]
for every closed subvariety $Y$ of $X$.
See \cite[\S1, Thm. 1.21]{Deb01}
or \cite[Thm. 1.2.23]{Laz04}. 

We say that $L$ is \emph{nef} if
\[
L^{\dim Y}\cdot Y\geq 0
\]
for every closed subvariety $Y$ of $X$.  By Kleiman's theorem, 
it is enough to check this condition for curves $Y$.  
See \cite[\S1, Thm. 1.26]{Deb01} or \cite[Thm. 1.4.9, Exa. 1.4.16]{Laz04}.

\subsubsection{Volumes and bigness}

Let $L$ be a line bundle on $X$. Define the \emph{volume}
\[
   \vol(L)=   \vol(X, L)=\lim_{n\to\infty} \frac{d!}{n^d}h^0(X, nL).
\]
Here the additive notation $nL$ means $L^{\otimes n}$, and we will take similar convention in the following. 
The existence of the limit is a consequence of Fujita's approximation theorem.  
See \cite[Thm. 11.4.7]{Laz04b} or \cite[Thm. A]{LM}.  

If $L$ is nef, then the Hilbert--Samuel formula gives the especially simple identity
\[
   \vol(L)=L^d.
\]
See \cite[\S1, Prop. 1.31]{Deb01} or \cite[Cor. 1.4.41]{Laz04}.

We say that $L$ is \emph{big} if $\vol(L)>0$.  
Thus, for nef line bundles, bigness is equivalent to positivity of the top self-intersection number.
Moreover, ample line bundles are always big. 

If $L$ and $M$ are nef line bundles on $X$, then Siu's inequality asserts that
\[
   \vol(L- M)\geq L^d-d L^{d-1}\cdot M.
\]
See \cite[Cor. 1.2]{Siu} or \cite[Thm. 2.2.15]{Laz04}.
The theorem is very useful due to the precise bound of the right-hand side, and the bound is accurate if $M$ is ``much smaller'' than $L$.
Its arithmetic analogue will be a main tool of this paper.

\subsubsection{More properties of volumes}

To introduce more results, it will be convenient to employ the notion of 
$\QQ$-line bundles, i.e. objects of $\Picc(X)_\QQ$. 
For the purpose here, it suffices to consider 
a \emph{$\QQ$-line bundle on $X$}
as an element of $\Pic(X)_\QQ=\Pic(X)\otimes_\ZZ\QQ$. 
We extend the previous positivity notions to $\QQ$-line bundles as follows. 

For a $\QQ$-line bundle $L$ on $X$, we say that $L$ is \emph{ample} (resp. \emph{effective}, \emph{nef}, or \emph{big}) 
if $L=aL_0$ in $\Pic(X)_\QQ$ for some strictly positive rational number $a$ and some $L_0\in \Pic(X)$ which is \emph{ample} (resp. \emph{effective}, \emph{nef}, or \emph{big}). 

We extend the volume function $\vol:\Pic(X)\to \RR$
to a volume function 
$$\vol:\Pic(X)_\QQ\lra \RR$$ 
by the homogeneity property
$$
\vol(aL_0):=a^d \vol(L_0), \quad a\in \QQ_{>0}, \quad L_0\in \Pic(X).
$$
Then a $\QQ$-line bundle $L$ on $X$ is big if and only if $\vol(L)>0$. 

Similarly, a (Cartier) \emph{$\QQ$-divisor on $X$}
is an element of $\Div(X)_\QQ$. 
All the above definitions extend to  $\QQ$-divisors by the same method. 

The notion of $\QQ$-line bundles brings lots of conveniences to the theory. For example, we can ``realize'' nef line bundles as 
``limits'' of ample $\QQ$-line bundles. In fact, let $L$ be 
a nef line bundle on $X$, and let $A$ be 
any ample line bundle on $X$. Then $L$ is the ``limit'' of $L+tA$ for {positive rational numbers} $t\to 0$, where each term $L+tA$ is ample. 

Finally, we list many other properties of the volume function in the following. 

\begin{thm} \label{geo volume}
Let $X$ be a projective variety of dimension $d>0$ over a field $k$.
Let $L$ and $M$ be line bundles on $X$. Then the following are true. 
\begin{enumerate}[(1)]
\item \textnormal{(birational invariance)}
If $\psi:X'\to X$ is a birational morphism of projective varieties over $k$, then $\vol(\psi^*L)=\vol(L)$.

\item \textnormal{(log-concavity)}  If $ L$ and $ M$ are effective, then 
$$ \vol ( L+ M)^{ 1/d}\ge  \vol ( L)^{1/d}+\vol (M)^{1/d}.$$

\item \textnormal{(Fujita approximation)}
If $L$ is big, then for any $\epsilon>0$, there exist a projective variety $X'$ over $k$ with a birational morphism $\psi:X'\to X$
and an ample $\QQ$-line bundle $ A$ on $X'$ such that 
$\psi^*L-A$ is effective on $X'$ and  
$$ \vol(A)\geq\vol(L)-\epsilon. $$

\item \textnormal{(continuity)} The volume function $\vol:\Pic(X)_\QQ\to \RR$ is continuous in the sense that for $t\in \QQ$, 
$$\lim_{t\to 0 }\vol (L+tM)=\vol (L).$$

\item \textnormal{(differentiability)}
If $L$ is big, then for $t\in \QQ$, the limit
$$\frac \d{\d t} \Big |_{t=0}\vol (L+tM)=
\lim_{t\to 0 } \frac{1}{t}(\vol (L+tM)-\vol (L))
$$
exists. 
If $L$ is big and nef, then for $t\in \QQ$, 
$$\frac \d{\d t} \Big |_{t=0}\vol (L+tM)=d \,  L^{d-1} \cdot M.$$
\end{enumerate}
\end{thm}

\begin{proof}
For (1) and (4), see  \cite[Prop. 2.2.43, Thm. 2.2.44]{Laz04}.
For (2) and (3), see \cite[Thm. 11.4.9, Thm. 11.4.4]{Laz04b}. 
Part (5) is proved by Boucksom--Favre--Jonsson in \cite[Thm. A]{BFJ09}, and we refer to the loc. cit. for an expression of the derivative for a general big $L$ in terms of the positive intersection number.
\end{proof}

\subsection{Effective sections and volumes}

Here we introduce effective sections and volumes of adelic line bundles, and then list many important properties of the volume functions.

\subsubsection{Volumes of Hermitian line bundles} \label{iccm2010:volume}

Let $\lbb$ be a hermitian line bundle on an arithmetic variety $\CX$ of dimension $d$.
The corresponding arithmetic linear series is the finite set
$$
H^0(\CX,\lbb):= \left\{s \in H^0(\CX, \lb): \|s\|_{\sup}\leq 1\right\}. 
$$
Here 
$$\|s\|_{\sup}:=\sup_{z\in \CX(\CC)}\|s(z)\|$$
is the usual supremum norm on $H^0(\CX,\lb)_{\CC}$.
An element of $H^0(\CX,\lbb)$ is called \emph{an effective section of} $\lbb$ on $\CX$, 
and $\lbb$ is called \emph{effective} if $H^0(\CX,\lbb)$ is nonzero.

Denote
$$h^0(\lbb)=\log \# H^0(\CX,\lbb).$$
The \emph{volume of} $\lbb$ is defined as
$$\wh\vol(\lbb)=\wh\vol(\CX,\lbb)=\lim_{n\to\infty} \frac{h^0(n\lbb)}{n^{d}/d!}.$$
The limit exists by Chen \cite{Che08}; and see \cite{Yu2} for a different proof using Okounkov bodies.
We say that $\lbb$ is \emph{big} if 
$\wh\vol(\lbb)>0$. 

Most properties of the volume function listed above {work} for hermitian line bundles, and also {work} for adelic line bundles by taking {limits}, so we will skip the statements for hermitian line bundles but make formal statements for adelic line bundles in the following.

\subsubsection{Volumes of adelic line bundles}

Here we recall the notions of volume and bigness of adelic line bundles in \cite[\S5.1-5.2]{YZ}. 

\kkk
If $k$ is a field, let $X$ be  {a quasi-projective variety} over $k$ of dimension $d$. If $k=\ZZ$, let $X$ be either a quasi-projective and flat integral scheme over $k$ of absolute dimension $d$, or a quasi-projective variety over a number field $K$ of dimension $d-1$. 

Let $\OL$ be an adelic line bundle on $X$ with underlying line bundle $L$ on $X$.  
Define 
$$\wh H^0(X, \OL):=\{s\in H^0(X, L): \|s(x)\|\leq 1,\, \forall x\in X^\an\}.$$
Here the metric $\|s(x)\|$ on the Berkovich space $X^\an$ is defined via the analytification 
functor
$\wh\Picc(X)\to \wh\Picc(X^\an)$. 
An element of $\wh H^0(X, \OL)$ is called an \emph{effective section of} 
{$\OL$} on $X$, 
and $\OL$ is called \emph{effective} if $\wh H^0(X, \OL)$ is nonzero.

In the case $k=\ZZ$ with $X$ a quasi-projective variety over a number field $K$ of dimension $d-1$, we further have 
$$\wh H^0(X, \OL)=\{s\in H^0(X, L): \|s(x)\|\leq 1,\, \forall x\in X^\ran\}.$$
Then we can further write  
$$\wh H^0(X, \OL)=\{s\in H^0(X, L): \|s\|_{v,\sup}\leq 1,\, \forall v\in M_K\}.$$
Here the \emph{supremum norm} at $v\in M_K$ is given by 
$$
\|s\|_{v,\sup}=\sup_{x\in X_v^\an} \|s(x)\|,\quad s\in H^0(X_{K_v}, L_{K_v}). 
$$

If $k=\ZZ$, then $\wh H^0(X, \OL)$ is a finite set, and we denote 
$$\hat h^0(X,\OL):=\log\#\wh H^0(X, \OL);$$ 
if $k$ is a field, then $\wh H^0(X, \OL)$ is a finite-dimensional vector space over $k$, and we denote 
$$\hat h^0(X,\OL):=\dim_k \wh H^0(X, \OL).$$

Define the \emph{arithmetic volume}
$$
\wh\vol(\OL)
=\wh\vol(X,\OL)
=\lim_{n\to \infty} \frac{d!}{n^{d}}\hat h^0(X, n\OL).
$$
The adelic line bundle $\OL$ is said to be \emph{big} on $X$ if 
$\wh\vol(\OL)>0$. 

The following result is from \cite[Thm. 5.2.1]{YZ}, which converts the convergence of the arithmetic volumes for adelic line bundles to those for hermitian line bundles.

\begin{thm}  \label{limit}
\begin{enumerate}[(1)]
\item The limit 
$\wh\vol(X,\OL)$
exists. 
\item
If $\OL$ is represented by an adelic line bundle 
$(\CL,(\CX_i,\overline \CL_i, \ell_{i})_{i\geq1})$ on $\CU$ for a quasi-projective model $\CU$ of $X$ over $k$, 
then 
$$
\wh\vol(X,\OL)
=\lim_{i\to \infty} \wh\vol(\CX_i,\CLL_i).
$$
\end{enumerate}
\end{thm}

On the right-hand side, $\wh\vol(\CX_i,\CLL_i)$ is the volume of $\CLL_i$ as a hermitian $\QQ$-line bundle on $\CX_i$, defined by homogeneity.
By the theorem, the definition of $\wh\vol(X,\OL)$ extends to adelic $\QQ$-line bundles on $X$ by homogeneity.

As in the geometric case, we have the following arithmetic Hilbert--Samuel formula from \cite[Thm. 5.2.2]{YZ}. 

\begin{thm}[arithmetic Hilbert--Samuel formula]  \label{HS}
Let $\OL$ be a nef adelic line bundle on $X$. Then
$$
\wh\vol(\OL) = \OL^d.
$$ 
\end{thm}

The theorem is reduced to the projective case by Theorem \ref{limit}. 
{Historically}, the corresponding result for hermitian line bundles on arithmetic varieties is a combination of the arithmetic Riemann--Roch theorem of Gillet--Soul\'e \cite{GS2}, an estimate of analytic torsions by Bismut--Vasserot \cite{BV89},  the Riemann--Roch theorem on lattice points by Gillet--Soul\'e \cite{GS3}, the arithmetic Nakai--Moishezon theorem by Zhang \cite{Zh1}, and some further refinements by Zhang \cite{Zh1} and Moriwaki \cite{Mo3}. 
See \cite[Cor. 2.7]{Yua} for more details.

We also have the following arithmetic Siu inequality from \cite[Thm. 5.2.2]{YZ}, where the case of hermitian line bundles is proved by Yuan \cite{Yua}.
 
\begin{thm}[arithmetic Siu inequality]  \label{bigness}
Let $\OL$ and $\OM$ be nef adelic line bundles on $X$. Then 
$$
\wh\vol(\OL-\OM) \geq \OL^d-d\,\OL^{d-1}\OM.
$$
\end{thm}

\subsubsection{More properties of volumes}

\kkk
Let $X/k$ and $d$ be as above.
Then $X$ is either quasi-projective over $k$ or quasi-projective over a number field $K$ (for $k=\ZZ$), and $d$ is the dimension of a quasi-projective model of $X$ over $k$. 
The definition of $\wh\vol(\OL)$ extends to adelic $\QQ$-line bundles on $X$ by homogeneity.

We have the following arithmetic counterpart of Theorem \ref{geo volume} from \cite[\S5.2.5]{YZ}. 

\begin{thm} \label{arith volume}
Let $\OL$ and $\OM$ be adelic line bundles on $X$. Then the following are true. 
\begin{enumerate}[(1)]
\item \textnormal{(birational invariance)}
If $\psi:X'\to X$ is a birational morphism over $k$, then $\wh\vol(\psi^*\OL)=\wh\vol(\OL)$.

\item \textnormal{(log-concavity)}  If $ \OL$ and $ \OM$ are effective, then 
$$ \wh\vol ( \OL+ \OM)^{ 1/d}\ge  \wh\vol ( \OL)^{1/d}+\wh\vol (\OM)^{1/d}.$$

\item \textnormal{(arithmetic Fujita approximation)}
If $\OL$ is big, then for any $\epsilon>0$, there exist a birational morphism $\psi:X'\to X$
and a nef adelic $\QQ$-line bundle $\ol A$ on $X'$ such that 
$\psi^*\OL-\ol A$ is effective on $X'$ and  
$$ \wh\vol(\ol A)\geq\wh\vol(\OL)-\epsilon. $$

\item \textnormal{(continuity)} The volume function $\wh\vol:\wh\Pic(X)_\QQ\to \RR$ is continuous in the sense that for $t\in \QQ$, 
$$\lim_{t\to 0 }\wh\vol (\OL+t\OM)=\wh\vol (\OL).$$

\item \textnormal{(differentiability)}
If $\OL$ is big, then for $t\in \QQ$, the limit
$$\frac \d{\d t} \Big |_{t=0}\wh\vol (\OL+t\OM)=
\lim_{t\to 0 } \frac{1}{t}(\wh\vol (\OL+t\OM)-\wh\vol (\OL))
$$
exists. 
If $\OL$ is big and nef, then for $t\in \QQ$, 
$$\frac \d{\d t} \Big |_{t=0}\wh\vol (\OL+t\OM)=d \,  \OL^{d-1} \cdot \OM.$$
\end{enumerate}
\end{thm}

\begin{proof}
All these results can be reduced to the projective case by Theorem \ref{limit}. 
Then the results for the case that $k$ is a field are just Theorem \ref{geo volume}. 
Now we explain the corresponding results for hermitian line bundles on arithmetic varieties. 
Parts (1) and (4) for hermitian line bundles are due to Moriwaki \cite[Thm. 4.3, Thm. B]{Mo3}. 
Part (2) for hermitian line bundles is due to Yuan \cite[Thm. B]{Yu2}. 
Part (3) for hermitian line bundles {was proved independently by Chen \cite{Che10} and Yuan \cite{Yu2}}. 
Part (5) for hermitian line bundles is due to Chen \cite{Che11}. 
\end{proof}

\section{Height functions}
\label{sec height}

In this section, we study height functions {in terms of both classical terminology and adelic line bundles}. The key results of this section are the fundamental inequalities, which will be useful for our {applications} later. We will also see that these fundamental inequalities become {very powerful when combined} with positivity results.

\subsection{Classical height functions}

Let us briefly recall the notion of Weil heights on projective varieties, canonical heights on polarized algebraic dynamical systems, 
 and N\'eron--Tate heights on abelian varieties and curves. 
 The basic references are \cite{Ser, HS, BG06, Yua12, YG}. 

Let $K$ be a number field throughout this subsection. 
Denote by $M_K$ the set of places of $K$. We normalize the absolute value $|\cdot|_v$ for every $v\in M_K$ as in \S\ref{sec notation}.

\subsubsection{Weil heights on projective varieties}

Let ${\mathbb P}^n$ be the projective space of dimension $n$ over $K$.
The \emph{standard height function $h: {\mathbb P}^n(\overline K)\to {\mathbb{R}}$} is defined to be
$$h(x_0, x_1, \cdots, x_n)=\frac{1}{[K':{\mathbb Q}]}\sum_{w\in M_{K'}} \epsilon_w \log \max\{|x_0|_w, |x_1|_w,\cdots, |x_n|_w \},$$
where $K'$ is a finite extension of $K$ containing all the coordinates $x_i$. It is independent of the choice of the homogeneous {coordinates} by the product formula.

Let $X$ be a projective variety over $K$, and let $L$ be an ample line bundle on $X$. 
Let $i: X\to {\mathbb P}^n$ be any morphism such that $i^*\mathcal O(1)\cong L^{\otimes e}$ for some positive integer $e\geq 1$. 
We obtain a height function 
$$h_{L,i}= \frac1e h\circ i : X(\overline K)\longrightarrow {\mathbb{R}}$$ 
as the composition of 
$i: X(\overline K)\to {\mathbb P}^n(\overline K)$ and $\frac1e h: {\mathbb P}^n(\overline K)\to {\mathbb{R}}$. 
It depends on the choices of {$e$ and $i$}.

More generally, let $L$ be any line bundle on $X$. We can always write $L= A_1\otimes A_2^{\otimes(-1)}$ for two ample line bundles $A_1$ and $A_2$ on $X$. For $k=1,2$, let $i_k: X\to {\mathbb P}^{n_k}$ be {morphisms} such that $i_k^*\mathcal O(1)\cong A_k^{\otimes e_k}$ for some positive integer $e_k$. We obtain a height function 
$$h_{L,i_1,i_2}= h_{A_1,i_1}-h_{A_2,i_2}: X(\overline K)\longrightarrow {\mathbb{R}}.$$ 
It depends on the choices of $(A_1, A_2, i_1,i_2)$.
However, the following result asserts that it is unique up to bounded functions.

\begin{thm}[Weil's height machine] \label{Weil}
The above construction $L\mapsto h_{L,i_1,i_2}$ gives a group homomorphism
$$
{\mathcal H}: \mathrm{Pic}(X) \longrightarrow  
\frac{\{ \mathrm{functions\ }\phi: X(\overline K)\to {\mathbb{R}} \}}
{ \{ \mathrm{bounded\ functions\ }\phi: X(\overline K)\to {\mathbb{R}} \}}.
$$
\end{thm}
A function $h_L: X(\overline K)\to {\mathbb{R}} $ in the class ${\mathcal H}(L)$ in the theorem is called a \emph{Weil height function} associated to $L$. 

The most important property of Weil heights is the following Northcott property. 

\begin{thm}[Northcott property] \label{Northcott}
Let $X$ be a projective variety over a number field $K$. Let  
$h_L: X(\overline K)\to {\mathbb{R}} $ be a Weil height function 
associated to an ample line bundle $L$ on $X$. 
Then for any real {numbers} $C_1$ and $C_2$, the set
$$
\{x\in X(\overline K): \deg(x)<C_1, \ h_L(x)<C_2\}
$$
is finite. 
\end{thm}

Here $\deg(x)$ denotes the degree of the residue field (or coefficient field) $K(x)$ of $x$ over $K$. 
The theorem is easily reduced to the standard height function on projective spaces, for which the property can be obtained explicitly.

\subsubsection{Canonical heights in algebraic dynamics}
\label{sec canonical height}

Let $S$ be a quasi-projective variety over a number field
$K$. Let $(X,f,L,q)$ be a \emph{polarized algebraic dynamical system over} 
 $S$  as in \S\ref{sec adelic dynamics}. 

Let us first assume $S=\Spec K$.  
Let $h_L: X(\overline K)\to \RR $ be any Weil height corresponding to $L$.
The \emph{canonical height function} (or the \emph{Call--Silverman height function})
$$\hat h_{L}: X(\overline K)\lra \RR $$
 is defined by Tate's limit argument
$$\hat h_{L}(x) =\lim_{n\rightarrow \infty} \frac{1}{q^n}h_{L}(f^n(x)).$$
The limit $\hat h_{L}(x)$ always exists and is independent of the choice of the Weil height $h_L$. 
It has the following basic properties.
\begin{enumerate}
\item [(a)] $\hat h_{L}(f(x))=q\hat h_{L}(x)$ for any $x\in X(\overline K)$,
\item [(b)] $\hat h_{L}(x)\geq 0$ for any $x\in X(\overline K)$, and the equality holds if and only if $x$ is preperiodic {in the sense that} $f^m(x)=f^n(x)$ for some $m>n\geq 0$.
\end{enumerate}
The function $\hat h_{L}: X(\overline K)\to \RR$ is the unique Weil height corresponding to $L$ and satisfying (a), and thus is independent of the choice of $h_L$.
We refer to \cite{Yua12, YG} for more details. 

Now let $S$ be a general quasi-projective variety over $K$.
For any closed point $s\in S$, the fiber $(X_s, f_s, L_s,q)$ is a dynamical system over $s=\Spec K(s)$. 
This gives a canonical height function
$\hat h_{L_s}: X_s(\overline{K(s)})\to \RR$. 
Varying $s$, we obtain a canonical height function
$$\hat h_{L}: X(\overline{K})\lra \RR,$$
which still satisfies  properties (a) and (b). 

Alternatively, we can start with a projective model $(X',L')$ of $(X,L)$ over $K$, choose a Weil height function $h_{L'}:X'(\ol K)\to \RR$ associated to $L'$, and consider the restriction $h_{L}':X(\ol K)\to \RR$ of $h_{L'}$ to $X(\ol K)$. Then Tate's limit argument transfers 
$h_{L}':X(\ol K)\to \RR$ to the canonical height $\hat h_{L}:X(\ol K)\to \RR$.

\subsubsection{N\'eron--Tate heights on abelian varieties}

Let $A$ be an abelian variety over a number field $K$. Denote by $[m]:A\to A$ the multiplication by an integer $m$. 
Let $L$ be a symmetric and ample line bundle on $A$. 
Here $L$ is called \emph{symmetric} if $[-1]^*L\cong L$, which implies $[m]^*L\cong L^{\otimes (m^2)}$ for all integers $m$. 
Then $(A,[2], L,4)$ forms a dynamical system over $K$. 
The  \emph{canonical height function}
$$\hat h_{L}: A(\overline K)\longrightarrow {\mathbb{R}}$$
is also called the \emph{N\'eron--Tate height function}. 
We further have the following quadraticity and positivity properties.
\begin{thm}[quadraticity]
\begin{enumerate}[(1)]
\item The height $\hat h_L(x)\geq 0$ for any $x\in A(\overline K)$, and the equality holds if and only if $x$ is torsion.
\item
The function $\hat{h}_L: A(\overline{K})\to{\mathbb{R}}$ is {quadratic} in the sense that
it satisfies the parallelogram rule 
$$
\hat{h}_L(x+y)+\hat{h}_L(x-y)=2\left(\hat{h}_L(x)+\hat{h}_L(y)\right),\quad \forall x,y\in A(\overline{K}).
$$
Moreover, we have
$$
\hat{h}_L(mx)=m^2\hat{h}_L(x),\quad \forall m\in{\mathbb Z},\quad \forall x\in A(\overline{K}).
$$
\item
The quadratic form 
$$\hat{h}_L: A(\overline{K})_{\mathbb{R}} \lra {\mathbb{R}}$$
on the real vector space $A(\overline{K})_{\mathbb{R}}=A(\overline{K})\otimes_\ZZ{\mathbb{R}}$, induced by the height function $\hat{h}_L: A(\overline{K})\to{\mathbb{R}}$ via the quadraticity, is positive definite.
\end{enumerate}
\end{thm}
Similar results hold for abelian schemes over quasi-projective varieties over $K$, but we omit the statements here.

Define the N\'eron--Tate height pairing
$$\langle{x,y}\rangle: A(\overline{K})_{\mathbb{R}}\times A(\overline{K})_{\mathbb{R}}\longrightarrow {\mathbb{R}}$$
by 
$$ \langle{x,y}\rangle= \frac12 \left(\hat{h}(x+y)-\hat{h}(x)-\hat{h}(y)\right), \quad x, y\in A({\overline K})_{\mathbb{R}}.$$
The associated norm of the pairing is given by 
$$|x| := \sqrt{\hat{h}(x)}, \quad x\in A({\overline K})_{\mathbb{R}}.$$
The angle $\angle(x,y)$ between $x,y\in A({\overline K})_{\mathbb{R}}$ is defined as
$$\angle(x,y):=\arccos\left(\frac{\langle{x,y}\rangle}{|x|\cdot|y|}\right).$$
The left-hand side is set to be 0 if $x=0$ or $y=0$.

\subsubsection{N\'eron--Tate heights on curves}
\label{sec NT curve}

Let $C$ be a geometrically integral smooth projective {curve} of genus $g\geq 2$ over a number field $K$. Denote by $J$ the Jacobian variety of $C$ over $K$. 
As the multiplication
$$[2g-2]\colon J({{\overline K}})\longrightarrow J({{\overline K}})$$
is surjective, there is a line bundle $\alpha_0$ on $C_{{\overline K}}$ such that $(2g-2)\alpha_0$ is isomorphic to the canonical sheaf $\omega$. 
Replacing $K$ by a finite extension if necessary, we can assume that $\alpha_0$ is actually a line bundle on $C$. 
This process does not change our definition of heights, {and neither does} the choice of $\alpha_0$.

Consider the Abel--Jacobi embedding
$$
i_{\alpha_0}: C\longrightarrow J, \quad x\longmapsto x-\alpha_0.
$$
Recall that the theta divisor on $J$ is given by
$$
\theta_{\alpha_0}= \underbrace{i_{\alpha_0}(C)+\dots + i_{\alpha_0}(C)}_{g-1 \text{ copies}}.$$
It is well-known that $\theta_{\alpha_0}$ is ample and gives a principal polarization of $J$. 
By \cite[p. 74, eq. (1)]{Ser}, $\theta_{\alpha_0}$ is symmetric in the sense that $[-1]^*\theta_{\alpha_0}$ is linearly equivalent to $\theta_{\alpha_0}$. It defines a 
\emph{N\'eron--Tate height function} 
$$\hat{h}=\hat h_{\theta_{\alpha_0}}:
J({\overline K})_{\mathbb{R}}\lra\RR.$$

We apply $\hat h(\cdot)$, $|\cdot|$, $\langle{\cdot,\cdot}\rangle$ and $\angle(\cdot,\cdot)$ to $C({\overline K})$ via the embedding $i_{\alpha_0}: C\to J$. 
For example, for $x\in C({\overline K})$, we have a 
\emph{N\'eron--Tate height function} 
$$\hat{h}:
C({\overline K})\lra\RR$$
given by 
$$\hat h(x)=\hat h(i_{\alpha_0}(x))=\hat h(x-\alpha_0)=\hat h_{\theta_{\alpha_0}}(x-\alpha_0).$$

\subsection{Heights by adelic line bundles}

In this subsection, we introduce heights defined by adelic line bundles, compare them with the classical heights, and then introduce fundamental inequalities of heights.

\subsubsection{Heights by adelic line bundles}

Let $X$ be a quasi-projective variety over a number field $K$.
Let $\overline L$ be an adelic line bundle on $X$.
Define \emph{the height function $h_{\overline L}: X(\overline K)\to \RR$ associated to} $\overline L$ by
$$
h_{\overline L}(x)=h_{\overline L}(x')= \frac{1}{\deg(x)} \overline L \cdot x', \quad x\in X(\overline K).
$$ 
Here $x'$ is the closed point of $X$ corresponding to the algebraic point $x$. 

If $\overline L$ is integrable, \emph{the height of any closed subvariety $Y$ of $X_{\overline K}$ associated to $\overline L$} is defined by
$$
h_{\overline L}(Y)= h_{\overline L}(Y')
= \frac{\overline L^{\dim Y+1} \cdot Y'}{(\dim Y+1)(\wt L^{\dim Y}\cdot Y')}.
$$
Here $Y'$ is the closed $K$-subvariety of $X$ corresponding to $Y$, i.e. the image of the composition $Y\to X_{\overline K}\to X$.
The geometric part $\wt L$ of $\OL$ is the image of $\OL$ under the functorial map $\wh\Pic(X/\ZZ)\to \wt\Pic(X/K)$.   
The definition is meaningful only if the denominator is nonzero.

In particular, the height of the ambient variety $X$ is 
$$
h_{\overline L}(X)= \frac{\overline L^{\dim X+1} }{(\dim X+1)\, \wt L^{\dim X}}.
$$
Roughly speaking, heights are just the arithmetic degrees divided by the geometric degrees.

In terms of rational sections of $L$, we can write the height of a point 
$x\in X(\ol K)$ as a sum of ``local heights''. 
For this purpose, we introduce some extra notation. 
Denote by $x'\in X$ the closed point corresponding to $x$ as above. 
For any place $v$ of $K$, the base change gives an injection 
$x'\times_{K}K_v \to X_{K_v}$, whose image is a finite set of closed points of $X_{K_v}$. Via analytification, this finite set corresponds to a finite set $O_v(x)$ of classical points of $X_v^\an$, which is called the 
\emph{Galois orbit} of $x$ in $X_v^\an$. 
For each $z\in O_v(x)$, the \emph{degree} $\deg(z)$ denotes the degree of 
$z$ as a closed point of $X_{K_v}$ over $K_v$. 

\begin{lemma} \label{local sum}
For any $x\in X(\overline K)$, we have
$$
h_{\overline L}(x)=-\frac 1{\deg(x)}\sum_{v\in M_K} \sum_{z\in O_v(x)}
\epsilon_v  \log
\|s(z)\|_{v}^{\deg(z)}.
$$
Here $s$ is any rational section of $L$ regular and non-vanishing at $x$.
\end{lemma}

\begin{proof}
If $x\in X(K)$, this follows from Lemma \ref{iccm2010:degree2}. 
The general case is obtained by an argument of base change. 
\end{proof}

The relation between height functions associated with adelic line bundles and classical Weil heights is given by the following nice theorem.
We refer to \cite[Thm. 9.7]{Yua12} or \cite[Thm. 8.4.3]{YG} for a proof.  

\begin{theorem} \label{iccm2010:adelic weil height}
Let $X$ be a projective variety over a number field $K$. 
For any adelic line bundle $\overline L$ on $X$, the height function 
$[K:\QQ]^{-1}h_{\overline L}: X(\overline K)\to \RR$ is a Weil height function corresponding to the line bundle $L$ on $X$. 
\end{theorem}

\subsubsection{Canonical heights}

Resume the setting of \S \ref{sec canonical height}. 
Namely, let $S$ be a quasi-projective variety over a number field
$K$. Let $(X,f,L,q)$ be a \emph{polarized algebraic dynamical system over} 
$S$. 
Then we have a {canonical height function} 
$$\hat h_{L}: X(\overline K)\lra \RR$$
defined by Tate's limit argument.

On the other hand, as in \S\ref{sec adelic dynamics}, we have an {$f$-invariant} adelic line bundle $\OL_f$ over $X/\ZZ$ extending $L$. 
This gives a height function 
$$h_{\OL_f}: X(\overline K)\lra \RR.$$
Now we have the following theorem. 

\begin{theorem} \label{bridge}
$\displaystyle \hat h_{L}= [K:\QQ]^{-1}h_{\overline L_f}$. 
\end{theorem}
\begin{proof}
By taking fibers over $S$, it suffices to prove the case $S=\Spec K$. 
Then the result follows from Theorem \ref{iccm2010:adelic weil height} and the invariant properties of both height functions. 
\end{proof}

Note that we can also apply $\OL_f$ to define canonical heights of closed subvarieties of $X$, which will be important for applications in arithmetic dynamics.

\subsubsection{The fundamental inequality}

Let $X$ be a quasi-projective variety over a number field $K$. 
Let $\overline L$ be an adelic line bundle on $X/\ZZ$. 
This gives a height function
$h_{\OL}: X(\overline K)\to \RR.$
Zhang's \emph{essential minimum} of ${\ol L}$ on $X$ is given by 
$${e_1(X,\ol L)}=\sup_{U}\inf _{x\in U(\ol K)}h_{\ol L}(x),$$
where $U$ runs through open subvarieties of $X$. 

The following theorem from \cite[Thm. 5.3.3]{YZ} generalizes a part of Zhang's theorem of successive minima in \cite{Zha2} from the projective case to the quasi-projective case. 

\begin{thm}[fundamental inequality I] \label{minima2}
Let $X$ be a quasi-projective variety of dimension $d$ over a number field $K$.
Let $\overline L$  be a nef element in $\wh \Pic (X/\ZZ)$ such that its image $\wt L$ in $\wt\Pic(X/K)$ is big.  Then
$$e_1(X,\overline L) \geq h_{\overline L}(X) 
\geq \frac{1}{d+1} e_1(X,\overline L).$$
\end{thm}

We will only use the first inequality of the theorem, which is a consequence of the arithmetic Hilbert--Samuel formula in Theorem \ref{HS} by the following result from \cite[Lem. 5.3.4]{YZ}.

\begin{lem}[fundamental inequality II]  \label{minima3}
Let $X$ be a quasi-projective variety of dimension $d$ over a number field $K$.
Let $\overline L$  be an element of $\wh \Pic (X/\ZZ)$, and let
$\wt L$  be its image in $\wt\Pic(X/K)$. 
\begin{enumerate}[(1)]
\item 
Let $s\in H^0(X,nL)$ be a nonzero element for some positive integer $n$. 
Then for any $x\in (X\setminus |\div(s)|)(\ol K)$, we have 
$$
h_{\OL}(x) \geq -\frac1n\sum_{v\in M_K} \log\|s\|_{v,\sup}.
$$

\item
For any positive integer $n$ such that $\hat h^0(X,n\OL)>0$,
$$
e_1(X,\overline L) \geq 
\frac{\hat h^0(X,n\OL)}{n\, \hat h^0(X,n\TL)}-\frac2n[K:\QQ]
$$
if the right-hand side is strictly positive.

\item
If $\wh\vol(\OL)>0$, then $\wh\vol(\TL)>0$ and 
$$
e_1(X,\overline L) \geq 
\frac{\wh\vol(\OL)}{(d+1) \wh\vol(\TL)}.
$$
\end{enumerate}
\end{lem}

\begin{proof}
Part (1) is a consequence of Lemma \ref{local sum}. 

For part (2), note that $\wh H^0(X,n\OL)$ generates a $K$-subspace of $H^0(X,nL)$ of dimension at most $\wh H^0(X,n\TL)$. Then
a Minkowski type of argument implies that there is a nonzero element $s\in \wh H^0(X,n\OL)$ such that 
$$
-\log \|s\|_{v,\sup} \geq \frac{\hat h^0(X,n\OL)}{[K:\QQ]\hat h^0(X,n\TL)}- 2
$$
for every archimedean place $v$ of $K$. 

The inequality in (3) is the limit of (2) as $n\to\infty$.
{To prove $\wh\vol(\TL)>0$, take} an adelic line bundle $\OM$ on $X/\ZZ$ with 
$\wh\vol(\TM)>0$, which can be chosen as a nef model adelic line bundle for example. 
By the continuity of the volume function in Theorem \ref{arith volume}, 
 there still exists a rational number $\epsilon>0$ such that 
$\wh\vol(\OL-\epsilon \OM)>0$. 
Then (some positive multiple of) $\OL-\epsilon \OM$ is effective, and thus $\TL-\epsilon \TM$ is also effective.
As a consequence, $\wh\vol(\TL)\geq \wh\vol(\epsilon \TM)>0$.

\end{proof}

A key result in the proof of the uniform Mordell conjecture by Dimitrov--Gao--Habegger \cite{DGH} is a height inequality over abelian schemes, which also plays a fundamental role in the further work of K\"uhne \cite{Kuh}. The following height inequality from \cite[Thm. 5.3.7]{YZ} can be viewed as a theoretical version of that of \cite{DGH}.

\begin{prop}[height inequality] \label{height inequality}
Let $\pi: X\to S$ be a morphism of quasi-projective varieties over a number field $K$. 
Let $\OL\in \wh\Pic(X/\ZZ)$ and $\OM\in \wh\Pic(S/\ZZ)$ be adelic line bundles. 
 If $\OL$ is big on $X/\ZZ$, then there exist $\epsilon>0$ and a non-empty open subvariety $U$ of $X$ such that 
$$
h_{\OL}(x) \geq \epsilon\, h_{\OM}(\pi(x)), \quad\ \forall\, x\in U(\overline K). 
$$
\end{prop}
\begin{proof}
As in the proof of Lemma \ref{minima3}, there exists a rational number $\epsilon>0$ such that $\wh\vol(\OL-\epsilon \pi^*\OM)>0$. 
In fact, if $\OL$ and $\OM$ are nef, this is a consequence of Theorem \ref{bigness}, which asserts
$$
\wh\vol(\OL-\epsilon \pi^*\OM) \geq 
\OL^d-d\epsilon\, \OL^{d-1}\cdot \pi^*\OM.
$$
In general, the result follows from the continuity of the volume function in Theorem \ref{arith volume}.

By Lemma \ref{minima3}(3), there is a non-empty open subvariety $U$ of $X$ such that 
$$
h_{\OL-\epsilon \pi^*\OM}(x) \geq 0, \quad\ \forall\, x\in U(\overline K). 
$$
Then the result follows from the simple relation
$$
h_{\OL-\epsilon \pi^*\OM}(x)
={h_{\OL}(x)}
-\epsilon h_{\OM}(\pi(x)).
$$
\end{proof}

\section{Equidistribution of small points}
\label{sec equi}

The first equidistribution theorem {for} small points was proved by the pioneering work of Szpiro--Ullmo--Zhang \cite{SUZ}. It was extended to non-archimedean places by Chambert-Loir \cite{CL}, and extended to semipositive metrics by Yuan \cite{Yua}. 
With their theory of adelic line bundles, the equidistribution theorem was further extended to quasi-projective varieties over number fields by Yuan--Zhang \cite{YZ}. The goal of this section is to introduce and prove the {equidistribution theorem} in the quasi-projective case.

\subsection{Statement of the theorem}

Let $X$ be a quasi-projective variety of dimension $d$ over a number field $K$, and let $\OL$ be a nef adelic line bundle on $X$. 

\subsubsection{Preliminary definitions} \label{sec prelim}

Recall that the geometric part $\wt L$ of $\OL$ is the image of $\OL$ under the functorial map $\wh\Pic(X/\ZZ)\to \wt\Pic(X/K)$.   
The nefness of $\OL$ implies that of $\TL$. 
The arithmetic Hilbert--Samuel formula in Theorem \ref{HS} gives 
$$
\wh\vol(\OL) = \OL^{d+1}, \quad \wh\vol(\TL) = \TL^d.
$$
We further assume that $\wt L$ is big, or, equivalently, $\deg_{\TL}(X)=\TL^d>0$.
Then the height 
$$
h_{\overline L}(X)= \frac{\overline L^{\dim X+1} }{(\dim X+1)\, \wt L^{\dim X}}
$$
of $X$ is well-defined and non-negative. 

A sequence $\{x_m\}$ in $X(\overline K)$ is called \emph{generic} if every infinite subsequence is Zariski dense in $X$. 
By the fundamental inequality in Theorem \ref{minima2}, this gives 
\[
\liminf_{m\to\infty} h_{\OL}(x_m)\geq h_{\OL}(X).
\]
The sequence is called \emph{$h_\OL$-small} if
\[
\lim_{m\to\infty} h_{\OL}(x_m)= h_{\OL}(X).
\]
The existence of a generic and $h_\OL$-small sequence is equivalent to the equality 
$$e_1(X,\overline L) = h_{\overline L}(X).$$ 

The equidistribution theorem asserts that for every place $v$ of $K$, the probability measure $\mu_{x_m,v}$ (for the generic and small sequence) converges weakly to the equilibrium measure $\mu_{\OL,v}$ on the analytic space $X_v^{\an}$. In rough terms, we just say that {the Galois orbit} of $x_m$ is equidistributed on $X_v^{\an}$ with respect to the measure $\mu_{\OL,v}$. 
We will explain these two measures before stating the theorem. 

For every place $v$ of $K$, we have the Berkovich analytic space $X_v^\an$ of the quasi-projective scheme $X_{K_v}$ over the local field $K_v$. 
For any point $x\in X(\ol K)$, 
we have a \emph{Galois orbit} $O_v(x)$ defined right before Lemma \ref{local sum}.  
Namely, denote by $x'\in X$ the closed point corresponding to $x$.  
Then $O_v(x)$ is the image of the functorial map 
$$(x'\times_{K}K_v)^\an \lra X_{K_v}^\an.$$ 
Then we have a probability measure 
$$
 \mu_{x,v} =\frac{1}{\deg(x)}\sum_{z\in O_v(x)} \deg(z)\delta_z.
$$
Here $\delta_z$ denotes the Dirac measure of $z\in X_v^\an$, and 
 $\deg(z)$ denotes the degree of 
$z$ over $K_v$, viewed as a closed point of $x'\times_{K}K_v$. 

For every place $v$ of $K$, we have an \emph{equilibrium measure} 
\[
\mu_{\OL,v}=\frac{1}{\deg_{\TL}(X)} c_1(\OL)_v^d.
\]
We explain the measure $c_1(\OL)_v^d$ on $X_v^\an$ briefly. 
The general case {where} $X$ is quasi-projective can be obtained from the projective case by a limit process.
Assume that $X$ is projective for the moment. 
{If} $v$ is complex (and $X$ is projective), the measure is the classical Monge--Amp\`ere measure on $X_v^\an=X_v(\CC)$ in complex analysis by Bedford--Taylor \cite{BT82}. 
If $v$ is real (and $X$ is projective), the measure is the push-forward of the Monge--Amp\`ere measure via the quotient map $X_v(\CC)\to X_v^\an$. 
If $v$ is non-archimedean (and $X$ is projective), the measure is 
the Chambert-Loir measure originally constructed by Chambert-Loir \cite{CL}. See also \cite[\S3]{Yua12} and \cite[\S9.2]{YG} for more details on the Chambert-Loir measure in the projective case.

\subsubsection{The equidistribution theorem} 

We have the following equidistribution theorem from \cite[Thm. 5.4.3]{YZ}. 

\begin{thm}[equidistribution] \label{equi3}
Let $X$ be a quasi-projective variety over a number field $K$.
Let $\overline L$ be a nef adelic line bundle on $X/\ZZ$. 
Assume that the geometric part $\wt L$ of $\OL$ is big on $X/K$. 
Let $\{x_m\}_{m\geq1}$ be a generic and $h_\OL$-small sequence in $X(\overline K)$.
Then for every place $v$ of $K$, 
the measure {$\mu_{x_m,v}$} converges weakly to $\mu_{\overline L,v}$ on  $X_v^\an$.
\end{thm}

The theorem was previously proved in many cases where $X$ is projective by different works. The following is an incomplete list (for $X$ projective), and we refer to \cite[\S6.3]{Yua12} for more details.  
\begin{enumerate}[(1)]
\item The pioneering work of Szpiro--Ullmo--Zhang \cite{SUZ} proved the theorem for archimedean $v$ {under the assumption that} the metric of $\OL$ is strictly positive at $v$. The work immediately {led} to the solution of the Bogomolov conjecture by Ullmo \cite{Ull} and Zhang \cite{Zh3}. 

\item For $\dim X=1$ and archimedean $v$, the theorem was proved {by} Bilu \cite{Bi97} and Autissier \cite{Au01}.

\item For $\dim X=1$ and  non-archimedean $v$, the theorem was proved by Baker--Hsia \cite{BH05}, Baker--Rumely \cite{BR06}, Favre--Rivera-Letelier \cite{FR06}, and Chambert-Loir \cite{CL}.

\item For non-archimedean $v$, Chambert-Loir \cite{CL} actually introduced the equilibrium measure $\mu_{\OL,v}$ in any dimension, and his proof worked for general dimensions under a positivity assumption at $v$.

\item The works \cite{SUZ, CL} are based on the arithmetic Hilbert--Samuel formula for perturbations of $\OL$, and thus {require} a positivity assumption of $\OL$ at $v$. Yuan \cite{Yua} proved the arithmetic Siu inequality, and thus proved the theorem for all cases of projective $X$. 
\end{enumerate}

If $X$ is quasi-projective, K\"uhne \cite{Kuh} proved {equidistribution for} canonical heights on abelian schemes, which is independent of the work of \cite{YZ}, and {applied it to prove} the uniform Bogomolov conjecture and the uniform Mordell conjecture. 
We will return to these uniformity theorems in \S\ref{sec uniform}. 

In {applications}, it might be convenient to have an equidistribution theorem on the Berkovich space $X_{\CC_v}^\an$, where $\CC_v$ is the completion of the algebraic closure of $K_v$.
This version of equidistribution is actually equivalent to the version for $X_{K_v}^\an$ in Theorem \ref{equi3}. 
We refer to \cite[\S3.1]{Yua} for the equivalence in the projective case, which can be extended to the quasi-projective case with little extra effort.

\subsection{Application to algebraic dynamics} 

The projective case of the equidistribution theorem has been widely used in algebraic dynamical systems. We refer to \cite{Yua12} for a detailed introduction of this aspect. 
The quasi-projective case has two natural applications to algebraic dynamical systems. We describe them here briefly. 

\subsubsection{Application to non-degenerate subvarieties} 
Let $S$ be a quasi-projective variety over  a number field $K$. 
Let $(X,f,L,q)$ be a \emph{polarized algebraic dynamical system over} 
$S$. 
Then we are in the setting of \S\ref{sec adelic dynamics} and \S \ref{sec canonical height}. 
In particular, we have an adelic line bundle $\OL_f$ over $X/\ZZ$ extending $L$ and satisfying {$f^*\OL_f\simeq q\OL_f$}. 
This gives a canonical height function 
$$h_{\OL_f}: X(\overline K)\lra \RR,$$
which is equal to $[K:\QQ]\cdot \hat h_{L}$ by Theorem \ref{bridge}. 
Denote by $\TL_f$ the geometric part of $\OL_f$ over $X/K$.
It is also nef and satisfies {$f^*\TL_f\simeq q\TL_f$}. 

Denote $n=\dim X$ and $e=\dim S$. 
Note that $\deg(f)=q^{n-e}$, which can be computed on the generic fiber of $X$ over $S$. 
By the projection formulae, 
$$(q\OL_f)^{n+1}=(f^*\OL_f)^{n+1}=\deg(f)\, \OL_f^{n+1}=q^{n-e}\, \OL_f^{n+1},$$
$$(q\TL_f)^{n}=(f^*\TL_f)^{n}=\deg(f)\, \TL_f^{n}=q^{n-e}\, \TL_f^{n}. 
$$
Thus we always have $\OL_f^{n+1}=0$. 
{Assume} $e=\dim S>0$ in the following. Then we also have $\TL_f^{n}=0$. 
In other words, $\TL_f$ is not big on $X/K$.
Then the pair $(X, \OL_f)$ does not satisfy the conditions of Theorem \ref{equi3}. However, the condition might be satisfied if we pass to subvarieties of $X$.  

Let $Y$ be a {closed} subvariety of $X$. We say that $Y$ is \emph{non-degenerate} if the pull-back $\TL_f|_Y$ is big on $Y$.
In this case, the pair $(Y, \OL_f|_Y)$ {satisfies} the conditions of Theorem \ref{equi3}, and we obtain the following equidistribution of generic and small points of $Y$ from \cite[Thm. 6.2.3]{YZ}.

\begin{thm}[equidistribution over non-degenerate subvarieties] \label{equi4}
Let $S$ be a quasi-projective variety over a number field $K$. 
Let $(X,f,L,q)$ be a polarized algebraic dynamical system over $S$.
Let $Y$ be a non-degenerate closed subvariety of $X$ over $K$. 
Let $\{y_m\}_{m\geq 1}$ be a generic sequence of $Y(\overline K)$ such that 
$h_{\OL_{f}}(y_m) \to 0$.
Then for every place $v$ of $K$, 
the measure {$\mu_{y_m,v}$} converges weakly to $\mu_{\overline L_f|_Y,v}$ on  $Y_v^\an$.
\end{thm}

We refer to \cite[\S6.2.3]{YZ} for a historical account of the equidistribution theorem in this dynamical setting.
In particular, if $X\to S$ is an abelian scheme, the notion of non-degeneracy of $Y$ was studied by Dimitrov--Gao--Habegger \cite{DGH} in their proof of the uniform Bogomolov conjecture. A version of the above equidistribution theorem in this case was proved by K\"uhne \cite{Kuh} and applied to refine the result of \cite{DGH}.

\subsubsection{Application to PCF points} 
The second application of our equidistribution theorem is to derive equidistribution of post-critically finite points (or just PCF points). 

Let $S$ be a smooth and quasi-projective variety over a number field $K$. Let $X=\BP^1_S$ be the {projective line} over $S$, and let $f:X\to X$ be a finite morphism over $S$ of degree $q>1$. 
Denote by $\pi:X\to S$ the structure morphism.
There is a line bundle $L$ on $X$ of the form $\CO(q-1)\otimes \pi^*N$ for some line bundle $N$ on $S$ such that $f^*L\simeq qL$. 
Then we have a dynamical system $(X,f,L,q)$ over $S$. 

In practice, we can take $S$ to be the moduli space $\CM_q^1$ of dynamical systems on $\PP^1$ over $K$ or a subvariety of some moduli space. 

The canonical morphism $f^*\omega_{X/S}\to \omega_{X/S}$ induces
a global section $\delta f$ of $\omega_f=\omega_{X/S}\otimes f^*\omega_{X/S}^\vee$ on $X$.
The \emph{ramification divisor}  $R=R(f)$ of the finite morphism $f:X\to X$
is defined to be the divisor of the section $\delta f$. It is also viewed as a (possibly non-reduced) closed subscheme in $X$.

For any point $y\in S(\overline K)$, the fiber of $(X,f,L,q)$ above $y$ gives a dynamical system {$(X_y, f_y, L_y, q)$} over $\overline K$. The ramification divisor {$R(f_y)$ on $X_y=\PP^1_{\ol K}$} is equal to the fiber of $R(f)$ above $y$. 
The point $y\in S(\overline K)$ is called \emph{post-critically finite} (PCF) if 
every irreducible component of $R(f_y)$ (with reduced structure)  is preperiodic under $f_y$.

Start from $f$-invariant extension $\OL_f$ of $L$ in $\wh\Pic(X)$. 
Define a nef adelic line bundle on $S/\ZZ$ by
$$
\OM= N_{R/S}(\OL_f|_R).
$$
Here the norm map {on the right-hand side} is also the $1$-fold Deligne pairing
recalled in \S\ref{sec Deligne pairing}, which can be extended to $R$ (if it is non-integral) by linearity via writing $R$ as a Weil divisor. 
Now we have a non-negative height function 
$$
h_\OM:S(\overline K)\lra \RR.
$$
It detects PCF points in the  sense that $y\in S(\ol K)$ is PCF if and only if $h_\OM(y)=0$. 

Now we are in the setting to apply Theorem \ref{equi3} to {$(S, \OM)$} over $K$. An extra condition to check is whether the geometric part $\TM$ of $\OM$ is big on {$S/K$}.  As a deep theorem combining the works of 
DeMarco \cite{DeM1, DeM2}, Bassanelli--Berteloot \cite{BB07} and Gauthier--Okuyama--Vigny  \cite{GOV},  $\TM$ is big on {$S/K$} if the morphism {$S\to \CM_q^1$} is generically finite and its image is not contained in the flexible Latt\`es locus.
Moreover, in this case, the equilibrium measure $\mu_{\OM, v}$ at archimedean $v$ is exactly equal to the normalized bifurcation measure introduced by DeMarco \cite{DeM1, DeM2}. 
Hence, we obtain the following equidistribution theorem from \cite[Thm. 6.3.5]{YZ}.

\begin{thm} [equidistribution: PCF maps on projective line] \label{equi PCF1}
Let $S$ be a smooth and quasi-projective variety over a number field $K$. Let $X=\BP^1_S$ be the projective line over $S$, and let $f:X\to X$ be a finite morphism over $S$ of degree $q>1$.  
Assume that the morphism $S\to \CM_q^1$ is generically finite and its image is not contained in the flexible Latt\`es locus. 
Let $\{y_m\}_{m\geq 1}$ be a generic sequence of PCF points of $S(\overline K)$.
Then for every place $v$ of $K$, 
the measure {$\mu_{y_m,v}$} converges weakly to the normalized bifurcation measure on {$S_v^\an$}.
\end{thm}

We refer to  \cite[\S 6.3]{YZ} for {a more detailed discussion} of this topic. 
For example, if $S$ is a family of polynomial maps on $\PP^1$, the theorem was previously proved by Favre--Gauthier \cite{FG}.
Their strategy is to reduce the problem to the equidistribution of Yuan \cite{Yua} in the projective case, which works for polynomial maps but not for rational maps.

The equidistribution of PCF points fits perfectly {into} the setting of the dynamical {Andr\'e--Oort} conjecture of Baker--DeMarco \cite[Conj. 1.10]{BD}. Our equidistribution theorem plays a crucial role in the recent solution of the dynamical {Andr\'e--Oort} conjecture for 1-dimensional families by Ji--Xie \cite{JX}.  Previously, the conjecture was proved for 1-dimensional families of polynomial maps by Favre--Gauthier \cite{FG}.

\subsection{Variational principle}

The proof of Theorem \ref{equi3} still follows the strategy of \cite{Yua} by applying the variational principle of \cite{SUZ}. 
The following is the key positivity result in the proof, which is a part of 
Theorem \ref{arith volume}(5), but we state here for its importance and to clarify its origin from the arithmetic Siu inequality in Theorem \ref{bigness}. 

\begin{lem} \label{volume asymptotic}
Let $X$ be a quasi-projective variety of dimension $d$ over a number field $K$.
Let $\overline L$ and $\ol M$ be integrable adelic line bundles on $X/\ZZ$. 
Assume that $\OL$ is nef on $X/\ZZ$. 
Then 
$$   \widehat{\vol}(\OL+t\OM)\geq 
(\OL+t\OM)^{d+1}+O(t^2), \quad t\in \QQ.$$
\end{lem}
\begin{proof}

Writing $\OM=\OM_1-\OM_2$ with $\OM_1$ and $\OM_2$ nef, we have
\[
   \OL+t\OM=(\OL+t\OM_1)-t\OM_2.
\]
The arithmetic Siu inequality in Theorem \ref{bigness} gives
$$  \widehat{\vol}(\OL+t\OM)
   \geq (\OL+t\OM_1)^{d+1}
       -(d+1)(\OL+t\OM_1)^d\cdot t\OM_2 
   =(\OL+t\OM)^{d+1}+O(t^2).
   $$
\end{proof}

Now we prove the equidistribution theorem.

\begin{proof}[Proof of Theorem \ref{equi3}]
The conditions and the result do not change if {we replace} $\OL$ by $\OL+\pi^*\ON$ for an element $\ON\in \Pichat(\Spec K)_\intb$ with $\wh\deg(\ON)>0$. 
Here 
$$\pi^*: \Pichat(\Spec K)_\intb\lra \Pichat(X)_\intb$$
 is the pull-back map. 
As a consequence, we can assume $\OL^{d+1}>0$.
Here we denote $d=\dim X$. 

Let $\OM$ be an element in the kernel of the map
$\wh\Pic(X)_{\intb}\to \wt\Pic(X/K)_{\intb}$. 
Let $t$ be a nonzero rational number.
By Lemma \ref{volume asymptotic}, 
$$   \widehat{\vol}(\OL+t\OM)\geq 
(\OL+t\OM)^{d+1}+O(t^2)$$
is strictly positive if $|t|$ is small. 
By the fundamental inequality in Lemma \ref{minima3}, 
$$
e_1(X,\overline L+t\OM) \geq 
\frac{\wh\vol(\OL+t \OM)}{(d+1) \wh\vol(\TL)}
 \geq 
\frac{\OL^{d+1}+t (d+1)\OL^d\OM}{(d+1) \deg_\TL(X)}+O(t^2).
$$

Apply the inequality to the generic sequence $\{x_m\}_m$.
We have
$$
\liminf_{m\to \infty} h_{\OL+t \OM}(x_m) \geq 
\frac{\OL^{d+1}+t (d+1)\OL^d\OM}{(d+1) \deg_\TL(X)}+O(t^2).
$$
By assumption, 
$$
\lim_{m\to \infty} h_{\OL}(x_m) = h_\OL(X)=
\frac{\OL^{d+1}}{(d+1) \deg_\TL(X)}.
$$
Then the inequality implies 
$$
\liminf_{m\to \infty} t h_{\OM}(x_m) \geq 
t \frac{\OL^d\OM}{\deg_\TL(X)}+O(t^2).
$$

If $t>0$, the above implies 
$$
\liminf_{m\to \infty} h_{\OM}(x_m) \geq 
 \frac{\OL^d\OM}{\deg_\TL(X)}+O(t).
$$
If $t<0$, the above implies 
$$
\limsup_{m\to \infty} h_{\OM}(x_m) \leq 
 \frac{\OL^d\OM}{\deg_\TL(X)}+O(|t|).
$$
Set $t\to 0$ in each case. We obtain
$$
\lim_{m\to \infty} h_{\OM}(x_m) =  \frac{\OL^d\OM}{\deg_\TL(X)}.
$$

We are going to deduce the equidistribution theorem on $X_v^\an$ from the above limit identity.
Assume that $\OL\in \wh\Pic(\CU)_{\nef}$ for a quasi-projective model $\CU$ of $X$ over $\ZZ$, and assume that $\OL$ is represented by a Cauchy sequence 
$(\CL,(\CX_i,\overline \CL_i, \ell_{i})_{i\geq1})$ in $\wh\Picc(\CU)_{\rmod}$.
Here $\CX_i$ is a projective model of $\CU$, and $\CLL_i$ is a hermitian $\QQ$-line bundle on $\CX_i$. 
Assume that there is a morphism $\psi_i:\CX_i\to \CX_1$ extending the identity morphism of $\CU$.
Denote $X_i=\CX_{i,\QQ}$, which contains $X$ as an open subvariety.
 
Let $\CX_1'$ be another projective model of $X_1$ over $\ZZ$.
Let $\CMM$ be a hermitian $\QQ$-line bundle on $\CX_1'$, with a fixed isomorphism $\CM_\QQ\to \CO_{X_1}$. Then it induces a metric $\|\cdot \|_w$ of $\CO_{X_1}$ on 
$X_{1,w}^\an$ for any place $w$ of $K$. 
Assume that the metric $\| 1\|_w=1$ for any place $w\neq v$ of $K$.
Denote $f=-\log\|1\|_v$, which is continuous on $X_{1,v}^\an$. 
By definition, 
$$
h_{\CMM}(x_m) = \int_{X_v^\an} f \mu_{x_m,v},
$$
and
$$
\OL^d\cdot \CMM 
= \lim_{i\to\infty}  \CLL_i^d\cdot \CMM
= \lim_{i\to\infty} 
\int_{X_{i,v}^\an} f c_1(\CLL_i)_v^d.
$$
Then the above result gives a limit identity
$$
\lim_{m\to\infty}
\int_{X_v^\an} f \mu_{x_m,v}
= \frac{1}{\deg_\TL(X)} \lim_{i\to\infty} 
\int_{X_{i,v}^\an} f\, c_1(\CLL_i)_v^d.
$$
Here $f$ is viewed as a function on $X_{i,v}^\an$
by the pull-back induced by $\psi_{i,\QQ}:X_i\to X_1$.

Now we {are going} to vary $f=-\log\|1\|_v$, which is a function on $X_{1,v}^\an$ {induced by} $(\CX_1',\CMM)$.
We claim that the space of all such model functions $f$ is dense in 
$C(X_{1,v}^\an)$ under the topology of uniform convergence. 
If $v$ is complex, $f$ can be any smooth function, and the result is classical. If $v$ is real, the result can be derived from the complex case. If $v$ is non-archimedean, then $f$ is a model function, and the density theorem is due to Gubler (cf. \cite[Thm. 7.12]{Gub3} and \cite[Lem. 3.5]{Yua}).

Note that  
$$
\lim_{i\to\infty} \int_{X_{i,v}^\an}  c_1(\CLL_i)_v^d
= \lim_{i\to\infty} (\CL_{i,\QQ})^d
= \TL^d= \deg_\TL(X).
$$
Therefore, the limit identity also holds for any $f\in C(X_{1,v}^\an)$.

Finally, assume $f\in C_c(X_{v}^\an)$, viewed as an element of $C(X_{i,v}^\an)$ by the open immersion $X\to X_i$. Then
$$
\lim_{i\to\infty} 
\int_{X_{i,v}^\an} f c_1(\CLL_i)_v^d
=\lim_{i\to\infty} 
\int_{X_{v}^\an} f\, c_1(\CLL_i)_v^d|_{X_v^\an}
=\int_{X_{v}^\an} f c_1(\OL)_v^d.
$$
Here the last equality follows from the definition of $c_1(\OL)_v^d$. Therefore, the limit identity becomes
$$
\lim_{m\to\infty}
\int_{X_v^\an} f \mu_{x_m,v}
= \frac{1}{\deg_\TL(X)} \int_{X_{v}^\an} f c_1(\OL)_v^d.
$$
This proves the equidistribution theorem. 
\end{proof}

\section{Admissible canonical bundle}
\label{sec admissible}

In the original work, Arakelov \cite{Ara} introduced the admissible metric of  the canonical line bundle of a complex curve to {obtain} a clean adjunction formula. 
Zhang \cite{Zha1} extended the definition to curves over non-archimedean fields, and thus obtained an admissible canonical bundle for a curve over a number field. 
To prove the uniform Bogomolov conjecture, Yuan \cite{Yua21} extended the definition {to obtain the admissible canonical bundle} on a relative curve (or a family of curves), and proved the bigness of the admissible canonical bundle under a natural condition. 

In this section, we will introduce the admissible canonical bundle and prove its bigness. {In the next section}, we will apply the bigness to prove the uniform Bogomolov conjecture.

\subsection{Admissible canonical bundle}
\label{subsec admissible}

The goal of this subsection is to introduce the admissible canonical bundle on a relative curve (or a family of curves).  
The theory for a single curve over a local field or a number field was introduced by Zhang \cite{Zha1} in terms of reduction graphs before the introduction of adelic line bundles in \cite{Zha2}. 
A re-organization of the theory in terms of adelic line bundles is given in the appendix of \cite{Yua21}. 
The theory was extended to relative curves by Yuan \cite{Yua21}. 

\subsubsection{Admissible metrics: local setting}

Let us first recall Zhang's admissible metrics on curves over local fields. 
We refer to \cite[Appendix A]{Yua21} for more details. 

Let $F$ be a complete valued field with a non-trivial valuation. Let $C$ be a geometrically integral smooth projective curve of genus $g\geq 2$ over $F$. 
Let $\omega_{C/F}$ be the canonical sheaf. 
Let $\Delta$ be the diagonal of $C^2=C\times_FC$, viewed as a divisor on $C^2$. 

We start with some terminology of metrics on $C^2$.
A continuous metric $\|\cdot\|_{\Delta}$ of $\CO_{C^2}(\Delta)$ on $C^2$ is called 
\emph{symmetric}
if the Green function 
$$g_\Delta=-\log\|1\|_{\Delta}: (C^2\setminus \Delta)^\an\lra \RR$$ 
is symmetric in the two components of $C^2$.
For any finite extension $F'$ of $F$, by pull-back via the natural map 
$(C^2_{F'})^\an \to (C^2)^\an$, the Green function  $g_\Delta$ induces a Green function of 
$\Delta_{F'}$ on $(C^2_{F'})^\an$, and thus the metric  
 $\|\cdot\|_{\Delta}$ induces a continuous metric of $\CO_{C^2}(\Delta)_{F'}$ on $(C^2_{F'})^\an$. 
 By abuse of notation, we still denote them by $g_\Delta$ and $\|\cdot\|_{\Delta}$.

For any finite extension $F'$ of $F$ and any point $x\in C(F')$, denote
$$
(\CO(x), \|\cdot\|_{x}) := i_{x}^* (\CO(\Delta), \|\cdot\|_{\Delta})
$$
{viewed as} metrized line bundles on $(C_{F'})^\an$.
Here 
$\CO(x)$ is the line bundle on $C_{F'}$ corresponding to $x\in C(F')$, and
$$i_{x}=(x, \mathrm{id}): \Spec {F'}\times_{\Spec F} C \lra C\times_{\Spec F} C$$
is the natural morphism.
It follows that 
$$
g_{x}=-\log\|1\|_{x}: (C_{F'} \setminus \{x\})^\an\lra \RR$$ 
is equal to 
the pull-back of $g_{\Delta}$ via the map
$i_{x}:C_{F'} \to C^2_{F}$.
We may also write $g_{x}=g_{\Delta}(x,\cdot)$. 

We will need the notation of Monge--Amp\`ere measures on $C^\an$, which is actually the Chambert-Loir measure in the non-archimedean case. 
The references for this projective case are {given} at the end of \S\ref{sec prelim}. 
 
By \cite[Thm. A.1]{Yua21}, there are a unique integrable metric 
$\|\cdot\|_{a}$ of $\omega_{C/F}$ on $C^\an$ and a unique symmetric integrable metric 
$\|\cdot\|_{\Delta}$ of $\CO(\Delta)$ on $(C^2)^\an$ 
satisfying the following properties
 for all finite extensions $F'/F$ and all points $x\in C(F')$:
\begin{enumerate}[(1)]
\item the equality
$$
(2g-2)c_1(\CO(x), \|\cdot\|_{x})
=c_1(\omega_{C_{F'}/F'}, \|\cdot\|_{a}),
$$
of Monge--Amp\`ere measures holds on $(C_{F'})^\an$;
\item the integral
$$
\int_{(C_{F'})^\an} g_{\Delta}(x,\cdot)\, c_1(\omega_{C_{F'}/F'}, \|\cdot\|_{a})=0;
$$
\item the residue map
$$
(\omega_{C/F} \otimes_{\CO_C} \CO(x) )|_{x}\lra F'
$$
is an isometry, where $F'$ is endowed with the absolute value extending that of $F$.
\end{enumerate}
We call the metric $\VCV_a$ the \emph{admissible metric} of 
$\omega_{C/F}$, which is actually nef (or semipositive). 
We call the Green function $g_\Delta$ the \emph{admissible Green function} of $\Delta$, which is also written as $g_{\Delta,a}$ to emphasize the situation.

If $F=\CC$, the metrics are just the original Arakelov metrics on compact Riemann surfaces in \cite{Ara}. 
If $F=\RR$, they are induced by the Arakelov metrics on $C(\CC)$. 
If $F$ is non-archimedean, the metrics are essentially introduced by Zhang \cite{Zha2} as counterparts of the complex setting.

\subsubsection{Admissible canonical bundle: single curves}
\label{sec admissible single curves}

Let $K$ be a number field. Let $C$ be a geometrically integral smooth projective curve of genus $g\geq 2$ over $K$. 
Let $\omega_{C/K}$ be the canonical sheaf. 
Let $\Delta$ be the diagonal of $C^2=C\times_KC$, viewed as a divisor on $C^2$. 
Let $D$ be a divisor on $C$. 

Apply the above constructions to the local field $K_v$ for every place $v$ of $K$.
We obtain an adelic line bundle
$\bar\omega_{C/K,a}=(\omega_{C/K}, (\|\cdot\|_{a,v})_v)$ on $C$, and an adelic line bundle 
$\ol{\CO}(\Delta)_a=(\CO(\Delta), (\|\cdot\|_{\Delta, a,v})_v)$ on $C^2$.

Both of the above objects are integrable. 
For convenience, we say that both of them are \emph{admissible}. 
Moreover, $\bar\omega_{C/K,a}$ is called \emph{Zhang's admissible canonical bundle} of $C$ over $K$.

The admissible adelic line bundles $\bar\omega_{C/K,a}$ and 
$\ol{\CO}(\Delta)_a$ {were} originally introduced by Zhang \cite{Zha1}. 
We refer to \cite[Appendix A]{Yua21} for more details in terms of the current terminology. 

The arithmetic self-intersection number $\bar\omega_{C/K,a}^2\in \RR$ is called the \emph{admissible volume} of $C/K$. It is an important arithmetic invariant of $C$. 

It is known that $\bar\omega_{C/K,a}$ is always nef, so the admissible volume $\bar\omega_{C/K,a}^2\geq 0$. 
Moreover,  Zhang \cite{Zha1}
has reduced the Bogomolov conjecture for $C$
to the strict positivity $\bar\omega_{C/K,a}^2> 0$ (or equivalently the bigness of $\bar\omega_{C/K,a}$). 
We will come back to this topic in \S\ref{sec Bog single curves}. 

The positivity of the admissible volume $\bar\omega_{C/K,a}^2$ is built upon Zhang's $\varphi$-invariant introduced in \cite{Zha3}. Recall that the global $\varphi$-invariant
$$
\varphi(C)= \sum_{v\in M_K} \epsilon_{v} \varphi_v(C),
$$
where {the} local $\varphi$-invariant  is defined as the  integral
$$
\varphi_v(C)
=-\int_{(C_{K_v}\times C_{K_v})^\an}g_{\Delta,a} \, c_1(\ol\CO(\Delta)_a)^2.
$$
The summation has only finitely many nonzero terms, since $\varphi_v(C)=0$ if $v$ is non-archimedean and $C$ has potentially good reduction at $v$. 
If $v$ is archimedean, \cite[Prop. 2.5.3]{Zha3} proves that $\varphi_v(C)>0$ by completing it as a sum of squares in terms of the heat kernel. 
If $v$ is non-archimedean and $C$ does not have potentially good reduction at $v$, then $\varphi_v(C)$ is an invariant of the reduction graph of $C$ at $v$, and 
 Cinkir \cite[Thm. 2.11]{Cin1} proves  $\varphi_v(C)> 0$ in this case.
As a consequence, we have $\varphi(C) >0$. 

Finally, the following theorem finishes the proof of the positivity and thus the second proof of the Bogomolov conjecture.
 
\begin{theorem}[de Jong \cite{dJo3}]
\label{dejong}
Let $C$ be a geometrically integral smooth projective curve of genus $g\geq 2$ over a number field $K$. Then
$$
\ol\omega_{C/K,a}^2 \geq \frac{2}{3g-1}\sum_{v\in M_K} \epsilon_v\varphi_v(C).
$$
\end{theorem}

A stronger constant of the theorem was later obtained by Wilms \cite[Thm. 1.2]{Wil3} by applying the Hodge index theorem of adelic line bundles of Yuan--Zhang \cite{YZ17}.

\subsubsection{Admissible canonical bundle: relative curves} 

A projective curve $C$ over a field $k$ is called \emph{semistable} if $C_{\bar k}$ is reduced and all singular points of $C_{\bar k}$ are ordinary double points.  
It is called \emph{stable} if it is semistable with arithmetic genus $g\geq 2$, and any rational irreducible component of $C_{\bar k}$ intersects other irreducible components at three or more points. 

By a \emph{relative curve over a noetherian scheme $S$}, we mean a projective and flat morphism $\pi:X\to S$ {purely of relative dimension 1} with geometrically connected fibers. 
The relative curve is called \emph{smooth (resp. stable, semistable)}
{if} every fiber of  
$\pi:X\to S$ is smooth (resp. stable, semistable).
It is said to be \emph{of genus $g$} {if} every fiber of $\pi:X\to S$ has arithmetic genus $g$.

By Yuan \cite[Thm. 2.3]{Yua21}, we have the following construction of admissible canonical bundles {for} quasi-projective families of curves.

\begin{thm} \label{bb22:admissible2}
Let $k$ be either $\ZZ$ or a field. 
Let $S$ be a quasi-projective and flat normal integral scheme over $k$. 
Let $\pi:X\to S$ be a smooth relative curve of genus $g\geq 2$.
Denote by $\Delta:X\to X\times_S X$ the diagonal morphism. Then the following are true:
\begin{enumerate}[(a)]
\item
There is an adelic line bundle $\overline\omega_{X/S,a}$ in $\wh\Picc(X/k)$ with underlying line bundle $\omega_{X/S}$, 
such that for any $v\in S^\an$, the metric of $\omega_{X_{\CH_v}/\CH_v}$ on $X_{\CH_v}^\an$ induced by 
$\overline\omega_{X/S,a}$
is equal to the canonical admissible metric $\|\cdot\|_a$. 
Moreover, $\overline\omega_{X/S,a}$ is nef and unique up to isomorphism. 

\item
There is an adelic line bundle $\overline\CO(\Delta)_a$ in $\wh\Picc(X\times_S X/k)$ with underlying line bundle $\CO(\Delta)$, 
such that for any $v\in S^\an$, the metric of $\CO(\Delta_{\CH_v})$ on $(X_{\CH_v}^2)^\an$ induced by $\overline\CO(\Delta)_a$
is equal to the canonical admissible metric $\|\cdot\|_{\Delta,a}$. 
Moreover, $\overline\CO(\Delta)_a$ is integrable and unique up to isomorphism. 
\end{enumerate}
\medskip\noindent 
Moreover, the adelic line bundles satisfy the following extra properties:
\begin{enumerate}[(1)]
\item The canonical isomorphism
$$
\omega_{X/S} \lra \Delta^*\CO(-\Delta) 
$$
induces an isomorphism 
$$
\ol\omega_{X/S,a} \lra \Delta^*\ol\CO(-\Delta)_a.
$$
\item The canonical isomorphisms
$$
p_{1*}\pair{\CO(\Delta), p_2^* \omega_{X/S}} \lra \Delta^*p_2^* \omega_{X/S}\lra \omega_{X/S}
$$
induce an isomorphism
$$
p_{1*}\pair{\ol\CO(\Delta)_a, p_2^* \ol\omega_{X/S,a}} \lra \ol\omega_{X/S,a}.
$$
Here $p_1,p_2: X\times_S X\to X$ denote the two projections.
\end{enumerate}
\end{thm}

If $k=\ZZ$, at any point 
$s\in S$ whose residue field is a number field, the pull-backs 
$\overline\omega_{X/S,a}|_{X_s}$ 
and 
$\overline\CO(\Delta)_a|_{X_s^2}$
are canonically isomorphic to the admissible adelic line bundles 
$\overline\omega_{X_s/s,a}$ on $X_s$
and 
$\overline\CO(\Delta_s)_a$ on $X_s^2$
introduced in the previous {subsection}. 

Note that the uniqueness in {(a) and (b)} is a consequence of the essential injectivity of the analytification functor
$$
\wh\Picc(X) \lra \wh\Picc(X^\an). 
$$
For the existence (or the construction), the adelic line bundles essentially come from natural operations on {invariant} adelic line bundles from the dynamical system given by the relative Jacobian scheme. 
To illustrate the idea, we will introduce a description of 
$\ol\omega_{X/S,a}$ in that setting, which can also {serve as} a definition.

Resume the notation in Theorem \ref{bb22:admissible2}. 
Let $\pi:X\to S$ be a smooth relative curve of genus $g\geq 2$.
Denote by $J=\Pic^0_{X/S}$ the Jacobian scheme of $X$ over $S$. 
For more details on Jacobian schemes, we refer to \cite[Chap. 6]{MFK} and \cite[Chap. 9]{BLR}.

Let $\alpha$ be a line bundle on $X$ of degree $d$ on the fibers of $X\to S$. 
We have a finite $S$-morphism 
$$
i_\alpha: X\lra J, \quad x\longmapsto dx-\alpha.
$$
The right-hand side is understood in terms of the functor of points. 
We will also need the morphism
$$
(i_\alpha,i_\alpha): X\times_SX\lra J\times_SJ.
$$

Now we consider line bundles on $J$.
Without taking a base change of $S$, we have to deal with the case that there is no line bundle on $X$ of degree 1 on fibers over $S$.
However, there is still a line bundle $\Theta$ on $J$ satisfying the following properties:
\begin{enumerate}[(1)]
\item $\Theta$ is symmetric on $J$ in the sense that $[-1]^*\Theta\simeq \Theta$;
\item $\Theta$ is rigidified along the identity section $e:S\to J$ in the sense that $e^*\Theta\simeq \CO_S$;
\item for every geometric point $s\in S$, the fiber 
$\Theta_s$ is isomorphic to $\CO(2\theta_{\alpha_0})$ as a line bundle on $J_s$, where the theta divisor $\theta_{\alpha_0}$ on the Jacobian variety $J_s$ of the curve $X_s$ is as in \S\ref{sec NT curve}. 
\end{enumerate}
Note that (1) and (2) {imply} $[m]^*\Theta\simeq m^2\Theta$ for every integer $m$, and (3) implies that $\Theta$ is relatively ample over $S$. 
This {comes from} a construction of \cite[\S6.1, Prop. 6.9]{MFK}, and we also refer to \cite[Def. 2.4]{Yua21} for more details. 

As a consequence, we obtain a dynamical system $(J, [2], \Theta, 4)$ over $S$. Denote by $\ol\Theta$ the invariant adelic line bundle on $J/k$ extending $\Theta$, as introduced in \S\ref{sec adelic dynamics}. 

With this construction, we finally have the following alternative description of the admissible canonical line bundle. 

\begin{thm} \label{canonical alternative definition}
There is an isomorphism
$$\overline\omega_{X/S,a}
=\frac{1}{4g(g-1)}i_\omega^*\OTheta+ \frac{1}{64g^2(g-1)^4}\pi^*\pi_*\pair{i_\omega^*\OTheta,i_\omega^*\OTheta}$$
in $\wh\Pic(X/k)_\QQ$.
As a consequence, there is an isomorphism 
$$
\pi_*\pair{\overline\omega_{X/S,a},\overline\omega_{X/S,a}}
=\frac{1}{16g(g-1)^3}\pi_*\pair{i_\omega^*\OTheta,i_\omega^*\OTheta}
$$
in $\wh\Pic(S/k)_\QQ$.
\end{thm}

Here the Deligne pairings for the morphism $\pi:X\to S$ are described in \S\ref{sec Deligne pairing}. 
The first equality implies the second one by an easy calculation. 
By the formula, we immediately see that $\overline\omega_{X/S,a}$ is nef on $X$.
We refer to \cite[Thm. 2.10(1)]{Yua21} for more details on the theorem.

\subsection{Faltings heights and Hodge line bundles} 
\label{sec Hodge line bundles}

In this subsection, we recall Faltings heights of curves originally introduced by Faltings \cite{Fal} in his proof of the Mordell conjecture, and then we recall a result of Yuan--Zhang \cite{YZ} realizing the Faltings heights as the height function
of the adelic Hodge line bundle. These results will be used in the remaining parts of this paper. 
 
\subsubsection{Faltings height}
\label{sec Faltings height}

Let $A$ be an abelian variety over $K$. 
Denote by $\pi: \CA\to \Spec O_{K}$ the N\'eron model of $A$ over $K$. 
Denote the identity section by $e:  \Spec O_{K}\to \CA$.
Denote the \emph{Hodge line bundle} 
$$
\underline\omega_{\CA/O_K}=e^*\omega_{\CA/O_K}=\pi_*\omega_{\CA/O_K}. 
$$
Endow $\underline\omega_{\CA/O_K}$ with the 
\emph{Faltings metric} defined by
$$
\|\alpha\|_{\rm Fal}^2=\frac{1}{2^g} \left|\int_{A_\sigma(\CC)} \alpha\wedge\bar \alpha\right|
$$
for any embedding $\sigma:K\to \CC$ and any element $\alpha$ of
$$
(\underline\omega_{\CA/O_K})_\sigma(\CC)
=e^*\Omega_{A_\sigma(\CC)/ \CC}^g
\simeq
\Gamma(A_\sigma(\CC), \Omega_{A_\sigma(\CC)/ \CC}^g).
$$
Define the \emph{Faltings height} of $A$ over $K$ to be 
$$
h_\Fal^*(A)=\frac{1}{[K:\QQ]} \wh\deg(\underline\omega_{\CA/O_K}, \|\cdot\|_{\rm Fal}).
$$
Here the arithmetic degree is from the map
$\wh\deg:\wh\Pic(O_{K})\to \RR$.

Let $K'$ be a finite extension of $K$ such that $A$ has semi-abelian reduction over $O_{K'}$. 
Define the \emph{stable Faltings height} of $A$ to be
$$
h_\Fal(A)= h_\Fal^*(A_{K'}).
$$
The definition is independent of the choice of $K'$. 

Let $C$ be a geometrically integral smooth projective curve of genus $g\geq 2$ over a {number field} $K$. 
Let $J$ be the Jacobian variety of $C$ over $K$. 
Define the \emph{stable Faltings height}
$$
h_\Fal(C)=h_\Fal(J)
$$
to be the stable Faltings height of $J$ over $K$. 
Define the \emph{adjusted Faltings height}
$$h_\Fal^+(C)=\max\{h_\Fal(C),1\},$$
which is always positive. 

Alternatively, we can define $h_\Fal(C)$ without going to the Jacobian variety of $C$. 
In fact, let $\psi:\CCC\to \Spec O_K$ be the minimal regular model of $C$ over $O_K$. The \emph{natural metric} $\|\cdot\|_{\rm nat}$ on the vector bundle
{$\psi_*\omega_{\CCC/O_K}$} is defined by
$$
\|\beta\|_{\rm nat}^2=\frac{1}{2}\left| \int_{C_\sigma(\CC)} \beta\wedge\bar \beta\right|
$$
for any embedding $\sigma:K\to \CC$ and any element $\beta$ of
$$
({\psi_*}\omega_{\CCC/O_K})_\sigma(\CC)
=\Gamma(C_\sigma(\CC), \omega_{C_\sigma(\CC)/ \CC}).
$$
This gives a hermitian vector bundle $({\psi_*}\omega_{\CCC/O_K}, \|\beta\|_{\rm nat})$ over $\Spec O_K$.
Then we have an identity 
$$
h_\Fal^*(C)=\frac{1}{[K:\QQ]} \wh\deg ({\psi_*}\omega_{\CCC/O_K}, \|\beta\|_{\rm nat})
=\frac{1}{[K:\QQ]} \wh\deg (\det {\psi_*}\omega_{\CCC/O_K}, \det \|\beta\|_{\rm nat}). 
$$ 
See \cite[Lem. 3.4]{Yua21} for a history of this identity. 
Then we have 
$$
h_\Fal(C)=h_\Fal^*(C_{K'})
$$
for any finite extension $K'$ of $K$ such that $C$ has semistable reduction over $O_{K'}$.

\subsubsection{Adelic Hodge line bundle}
\label{sec adelic hodge}

Let $S$ be a noetherian scheme. 
Let $\pi:X\to S$ be a stable relative curve of genus $g\geq 2$.
The  \emph{Hodge vector bundle} of $X$ over $S$ is just 
$\pi_* \omega_{X/S}$, which is a vector bundle of rank $g$ on $S$. 
The
\emph{Hodge line bundle} of $X$ over $S$ is defined as 
$$
\lambda_S:=\det \pi_* \omega_{X/S}=\wedge^g \pi_* \omega_{X/S}. 
$$  

Assume that $S$ is a flat and quasi-projective integral scheme over $\ZZ$ or $\QQ$.
Then the \emph{natural metric} on $\pi_* \omega_{X/S}$ is defined  
such that for any point $y\in S(\CC)$ and any section
$$\beta\in  (\pi_*\omega_{X/S})(y)= \Gamma(X_y, \omega_{X_y/y}),$$
the metric gives
$$
\|\beta\|_{\rm nat}^2=\frac{1}{2}\left| \int_{X_y} \beta\wedge\bar \beta\right|.
$$
The \emph{determinant metric} on $\lambda_S$ is just 
 $$\|\cdot\|_{\rm det}=\det \|\cdot\|_{\rm nat}$$ induced by  the process
$\lambda_S=\det(\pi_*\omega_{X/S})$. 
This gives a pair $(\lambda_S, \|\cdot\|_{\rm det})$. 
As in \cite[Lem. 3.4]{Yua21}, we can also interpret this pair in terms of the Hodge vector bundle on the Jacobian scheme of $X$ over $S$. 

By \cite[Thm. 5.5.1]{YZ}, the pair $(\lambda_S, \|\cdot\|_{\rm det})$ extends canonically to an \emph{adelic Hodge line bundle} $\ol\lambda_{S}$ on $S/\ZZ$. It is an adelic line bundle on $S/\ZZ$ such that 
$$
h_{\ol\lambda_{S}}(y)= h_{\Fal}(X_y), \quad \forall y\in S(\ol\QQ). 
$$
In other words, we realize the Faltings heights as a height function given by the adelic Hodge line bundle.

\subsection{Bigness of the admissible canonical bundle}

Finally, we have the following bigness theorem from \cite[Thm. 3.1, Thm. 3.2]{Yua21}. The goal of this subsection is to sketch a proof of the theorem. 

\begin{thm}[bigness: admissible canonical bundle] \label{bigness55}
Let $k$ be either $\ZZ$ or a field. 
Let $S$ be a quasi-projective and flat normal integral scheme over $k$. 
Let $\pi:X\to S$ be a smooth relative curve over $S$ of genus $g\geq 2$ with maximal variation.
Then the adelic line bundle $\overline\omega_{X/S,a}$ is nef and big on $X$, and the adelic line bundle 
$\pi_*\pair{\overline\omega_{X/S,a},\overline\omega_{X/S,a}}$ is nef and big on $S$.
\end{thm}

Here we say that the relative curve $\pi:X\to S$ has \emph{maximal variation} if the moduli morphism $S\to M_{g,k}$ is generically finite to its image, where $M_{g,k}$ denotes the coarse moduli scheme of smooth curves of genus $g$ over $k$.

We note that the two nefness and bigness properties in the theorem are equivalent. 
In fact, {the implication from the nefness and bigness of $\overline\omega_{X/S,a}$ to those of 
$\pi_*\pair{\overline\omega_{X/S,a},\overline\omega_{X/S,a}}$ follows from} a basic property of the Deligne pairing (cf. \cite[Lem. 2.1(1)]{Yua21}). 
Conversely, by Theorem \ref{canonical alternative definition}, 
$$4g(g-1)\overline\omega_{X/S,a}=i_\omega^*\OTheta+\pi^*\pi_*\pair{\overline\omega_{X/S,a},\overline\omega_{X/S,a}}.$$
Then  the self-intersection of the right-hand side on $X$ contains a term
$$
i_\omega^*\OTheta\cdot \big(\pi^*\pi_*\pair{\overline\omega_{X/S,a},\overline\omega_{X/S,a}}\big)^{\dim S}
=a \big(\pi_*\pair{\overline\omega_{X/S,a},\overline\omega_{X/S,a}}\big)^{\dim S},
$$ 
which is strictly positive if $\pi_*\pair{\overline\omega_{X/S,a},\overline\omega_{X/S,a}}$ is big. 
Here $a=4g(g-1)(2g-2)$ is the degree of $i_\omega^*\OTheta$ on the generic fiber of $X\to S$. 
This gives the inverse implication. 

Our proof of Theorem \ref{bigness55} is a relative version of the proof of the Bogomolov conjecture {by} Zhang \cite{Zha1,Zha3}, Cinkir \cite{Cin1} and de Jong \cite{dJo3}. See \S\ref{sec admissible single curves} for an overview of the framework. 
Then in our proof, we need to ``globalize'' many invariants of curves into relative curves, and also ``globalize'' related equalities and inequalities. 
Before introducing our proof, we first introduce some related globalized 
results for the $\varphi$-invariants.

\subsubsection{Globalization $\ol\Phi_S$ of the $\varphi$-invariant}

Recall that Zhang's $\varphi$-invariant {was recalled} in 
\S\ref{sec admissible single curves}, and the goal here is to ``glue'' these local $\varphi$-{invariants} to an adelic divisor on relative curves. 
This process in spirit is similar to Theorem \ref{bb22:admissible2}, which extends Zhang's admissible canonical bundle for curves to relative curves.

\kkk
Let $S$ be a quasi-projective and flat normal integral scheme over $k$.
Let $\pi:X\to S$ be a smooth relative curve of genus $g\geq 2$. 

Consider the morphism $(\pi,\pi):X\times_SX\to S$ and its diagonal divisor $\Delta:X\to X\times_SX$. 
There are canonical isomorphisms
$$
(\pi,\pi)_*\pair{\CO(\Delta),\,\CO(\Delta),\, \CO(\Delta)}
\lra \pi_*\pair{\Delta^*\CO(\Delta),\Delta^*\CO(\Delta)}
\lra \pi_*\pair{\omega_{X/S},\omega_{X/S}}.
$$
This defines a section $s$ of the underlying line bundle of the adelic line bundle 
$$
\pi_*\pair{\ol\omega_{X/S,a},\ol\omega_{X/S,a}}-(\pi,\pi)_*\pair{\ol\CO(\Delta)_a,\,\ol\CO(\Delta)_a,\,\ol\CO(\Delta)_a}.
$$
Via this adelic line bundle, define 
$$\ol\Phi_S=\wh\div(s),$$ 
viewed as an adelic divisor on $S$. 
The underlying divisor of $\ol\Phi_S$ is 0 by definition. 
By the integration formula, 
the total Green's function $\wt g_{\ol\Phi_S}$ on $S^\an$
at any discrete or archimedean valuation $v\in S^\an$
is given by 
$$
\wt g_{\ol\Phi_S}(v)=\int_{(X_v\times X_v)^\an} \log \|1\|_{\Delta, a} c_1(\ol\CO(\Delta)_a)^2.
$$
This recovers the $\varphi$-invariant of {$X_{\CH_v}$}.

As before, let $J=\Pic_{X/S}^0$ be the Jacobian scheme, and $\OTheta$ be the invariant adelic line bundle on $J/k$. 
Let 
$j$ be the morphism 
$$j:X\times_SX \lra J, \quad (x,y)\longmapsto y-x.$$
The following result computes the Deligne pairing $(\pi,\pi)_*\pair{j^*\OTheta,j^*\OTheta,j^*\OTheta}$. 
It is a family version of, and is inspired by, de Jong \cite[Thm. 8.1]{dJo3}, where the latter was in turn based on the main result of Zhang \cite{Zha3}. 
The following result is from \cite[Thm. 3.6]{Yua21}.

\begin{thm} \label{lower bound family}
\kkk
There is an identity in $\wh\Pic(S)_\QQ$ given by 
$$
(\pi,\pi)_*\pair{j^*\OTheta,j^*\OTheta,j^*\OTheta}
=(12g-4)\pi_*\pair{\overline\omega_{X/S,a},\overline\omega_{X/S,a}}
- 8\CO(\ol \Phi_S).
$$
\end{thm}

The proof is still a family version of that of 
de Jong \cite[Thm. 8.1]{dJo3}. 
By the nefness of $\OTheta$, the Deligne pairing $(\pi,\pi)_*\pair{j^*\OTheta,j^*\OTheta,j^*\OTheta}$  is nef. 
Then the theorem gives
$$
\pi_*\pair{\overline\omega_{X/S,a},\overline\omega_{X/S,a}}
= \frac{2}{3g-1}\CO(\ol \Phi_S)+\nef. 
$$ 
Here ``$\nef$'' means a nef adelic line bundle, and later we will write
``$\eff$'' for an effective adelic divisor. 
In this form, we see that this is a family version of Theorem \ref{dejong}.

\subsubsection{Bigness in the geometric case}

Now we sketch our proof of Theorem \ref{bigness55}.
We already know that it suffices to prove the bigness of $\pi_*\pair{\overline\omega_{X/S,a},\overline\omega_{X/S,a}}$. 
Note that the case $k=\ZZ$ and $S=\Spec O_K$ for a number field $K$ is just the positivity equivalent to the original Bogomolov conjecture.  
As mentioned above, the key is to prove that 
$\pi_*\pair{\overline\omega_{X/S,a},\overline\omega_{X/S,a}}$ is nef and big on $S$. 
The process is to align the relevant results of \cite{Zha1, Zha3, Cin1, dJo3} 
recalled in \S\ref{sec admissible single curves}
into a family.

Let us first sketch the geometric case that $k$ is a field.
{After replacing $k$ by a finite extension} if necessary, we can assume that $\pi:X\to S$ has a {stable compactification} $\ol\pi:\ol X\to \ol S$, i.e. a projective variety 
$\ol S$ over $k$ with an open immersion $S\to \ol S$, a stable relative curve 
$\ol\pi:\ol X\to \ol S$ of genus $g$, and an open immersion $X\to \ol X$ compatible with the previous morphisms. 
Our proof takes the following steps.

\medskip\noindent\emph{Step 1}.
There is an effective adelic divisor $\ol E_S$ on $S$ such that
$$\pi_*\pair{\overline\omega_{X/S,a},\overline\omega_{X/S,a}}
= \overline\pi_*\pair{\omega_{\overline X/\overline S},\omega_{\overline X/\overline S}}
- \CO(\ol E_S)
$$
in $\wh\Pic(S)_\QQ$.
Here $\overline\pi_*\pair{\omega_{\overline X/\overline S},\omega_{\overline X/\overline S}}$ is viewed as an element of $\wh\Pic(S)_\QQ$ by the natural map $\Pic(\ol S)_\QQ\to \wh\Pic(S)_\QQ$.
The divisor $\ol E_S$ is in fact defined by the identity.
Then it has underlying divisor 
$0\in \Div(S)$, and thus is totally determined by its Green's function $g_{\ol E_S}$ on the Berkovich analytic space $S^\an$.
We have an ``explicit'' description of $g_{\ol E_S}$. Namely, its value at any discrete valuation $v$ of $k(S)/k$ is given by the $\epsilon$-invariant of the curve $X_{\CH_v}$ over the valuation field $\CH_v$ of $v$ defined by Zhang \cite{Zha1} in terms of graph theory. This determines $g_{\ol E_S}$ by continuity, and we say that $\ol E_S$ is the globalization of the $\epsilon$-invariant.

\medskip\noindent\emph{Step 2}.
The Noether formula gives
$$\overline\pi_*\pair{\omega_{\overline X/\overline S},\omega_{\overline X/\overline S}}
=12\lambda_{\overline S}- \CO(\Delta_{\overline S})
$$
in $\Pic(\ol S)_\QQ$.
Here $\lambda_{\overline S}=\det \ol\pi_*\omega_{\ol X/\ol S}$ is the Hodge line bundle of $\ol X$ over $\ol S$, and $\Delta_{\overline S}$ is the divisor of $\ol S$ with support equal to
$\ol S\setminus S$ measuring the {singularities} of $\ol X$ over $\ol S$.
Both $\lambda_{\overline S}$ and $\Delta_{\overline S}$ are the well-known tautological divisors in the theory of moduli spaces of curves.

\medskip\noindent\emph{Step 3}.
The difference $(2g-2)\Delta_{\ol S}-\ol E_S$ is an effective adelic divisor in $\wh\Div(S)$.
As both $\Delta_{\ol S}$ and $\ol E_S$  have underlying divisor 
$0\in \Div(S)$, it suffices to check $(2g-2)g_{\Delta_{\ol S}}\geq g_{\ol E_S}$ on the Berkovich analytic space $S^\an$.
By continuity, we only need to check it at any discrete valuation $v$ of $k(S)/k$, or equivalently compare the $\epsilon$-invariant of the curve $X_{\CH_v}$ defined by Zhang \cite{Zha1} 
and the classical $\delta$-invariant of $X_{\CH_v}$ counting the number of nodes of reduction. 
The comparison is done by graph theory. 

Combining these steps, we have
$$\pi_*\pair{\overline\omega_{X/S,a},\overline\omega_{X/S,a}}
= 12\lambda_{\overline S}- (2g-1)\CO(\Delta_{\overline S})+\eff.
$$

\medskip\noindent\emph{Step 4}.
There is an effective adelic divisor $\ol \Phi_S$ on $S$ such that
$$
\pi_*\pair{\overline\omega_{X/S,a},\overline\omega_{X/S,a}}
= \frac{2}{3g-1}\CO(\ol \Phi_S)+\nef.
$$
This {follows from} Theorem \ref{lower bound family}, as a family version of de Jong's inequality in Theorem \ref{dejong}. 

\medskip\noindent\emph{Step 5}.
The difference $\ol\Phi_S-\frac{1}{39}\Delta_{\ol S}$ is an effective adelic divisor in $\wh\Div(S)$.
Similar to Step 3, it suffices to compare the values of the Green's function at all discrete valuations $v$ of $k(S)/k$. 
Then it follows from the result of Cinkir \cite[Thm. 2.11]{Cin1} in graph theory.

As a consequence, we have
$$
\pi_*\pair{\overline\omega_{X/S,a},\overline\omega_{X/S,a}}
= \frac{2}{39(3g-1)}\CO(\Delta_{\ol S})+\nef+\eff.
$$

\medskip\noindent\emph{Step 6}.
Take a linear combination of the equalities respectively at the {ends} of Step 3 and Step 5 to cancel the term $\CO(\Delta_{\ol S})$. 
We obtain
$$
\left(1+ \frac{39(3g-1)}{2} (2g-1) \right)
\pi_*\pair{\overline\omega_{X/S,a},\overline\omega_{X/S,a}}
= 12\lambda_{\overline S}+\nef+\eff.
$$
Then the bigness of 
$\pi_*\pair{\overline\omega_{X/S,a},\overline\omega_{X/S,a}}$ follows from the classical result that the Hodge line bundle 
$\lambda_{\overline S}$
is nef and big on $\overline S$.

\subsubsection{Bigness in the arithmetic case}

Now we prove the bigness 
of $\pi_*\pair{\overline\omega_{X/S,a},\overline\omega_{X/S,a}}$
in the arithmetic case $k=\ZZ$. 
The above proof is not valid in this case as Step 2 becomes rather subtle in the arithmetic case. 
However, we already know the bigness of the image {$\TM$} of 
$\OM=\pi_*\pair{\overline\omega_{X/S,a},\overline\omega_{X/S,a}}$
in $\wt\Pic(S_\QQ/\QQ)$ by the above proof.
We claim that it suffices to prove that 
$$\OM= \ol\CO(c)+\eff+\nef$$
for some rational number $c>0$. 
Here  $\ol\CO(c)$  denotes a hermitian line bundle on $\Spec\ZZ$ of arithmetic degree $c$, viewed as an adelic line bundle on $S/\ZZ$ by pull-back.
In fact, the claim implies (for $d=\dim S$)
$$
\OM^{d}\geq \OM^{d-1}\cdot \ol\CO(c)= c\cdot \wt M^{d-1} >0. 
$$

Theorem \ref{lower bound family} (or Step 4) still works in the arithmetic case, so it suffices to prove that there exists a constant $c_0>0$ depending only on $g$ such that
$\ol\Phi_S-c_0$ is an effective adelic divisor on $S$.
This is implied by the following theorem from \cite[Thm. 3.10]{Yua21}. 

\begin{thm} \label{positive lower bound}
For any integer $g\geq 2$, there is a constant $c_0(g)>0$ depending only on $g$ such that $\varphi(C)\geq c_0(g)$ for any connected compact Riemann surface $C$ of genus $g$.
\end{thm}

Let $M_g$ be a fine moduli space of curves of genus $g$ over $\CC$ with a suitable level structure, and $\overline M_g$ be a suitable compactification of $M_g$. 
By Zhang \cite{Zha3}, we know that $\varphi>0$ on $M_g$. 
We will apply an adelic method to prove that $\varphi$ tends to infinity along the boundary $\overline M_g\setminus M_g$.
This will be {sufficient to prove the theorem by compactness}. 

Before introducing the adelic method, let us mention some results on asymptotic behavior of degeneration of the 
$\varphi$-invariant along the boundary of $\OM_g$. This topic is widely studied in the literature due to its arithmetic importance. See \cite[Thm. 1.1]{dJo2} for a precise asymptotic formula for $g=2$, which also gives a conjectural formula for $g>2$. 
For $g\geq2$ and degeneration to isolated singularities, the asymptotic formula was later proved by \cite[Thm. 7.1]{JS2} and \cite[Cor. 1.2]{Wil2}. 
These imply Theorem \ref{positive lower bound} in the case $g=2$ and in the case of 1-parameter families of Riemann surfaces of genus $g\geq2$.
As mentioned above, the proof of the theorem by Yuan \cite{Yua21} {uses an adelic method}, which asserts that the $\varphi$-invariant goes to infinity under degeneration. 
By more detailed analysis of the adelic method, Song \cite{Son} proved an asymptotic formula for the degeneration of $\varphi$ for all $g\geq2$ in terms of graph-theoretic data encoded from $\wt\Phi$, which confirms a variant of a conjecture of de Jong \cite[Conj. 1.2]{dJo2}.

\subsubsection{Adelic method for degeneration}

Now we introduce an adelic method to prove Theorem \ref{positive lower bound}. We will first introduce the adelic method in a general framework, since it is a very effective method to study degeneration of analytic terms coming from adelic divisors. 
The philosophy is that the geometric part of an adelic divisor governs the growth of its Green function over complex analytic spaces.

To {set up} the general framework, we will use the theory of adelic divisors on the pair $\OB=(\CC,|\cdot|)$ in \cite[\S2.7]{YZ}, which is a local theory parallel to the one reviewed in \S\ref{subsec adelic}. 
It can be avoided by still using the theory of adelic divisors on $\ZZ$ in many situations, but its use is more natural and makes the situation clear. 
Let us first recall the theory of adelic divisors on $\OB=(\CC,|\cdot|)$. 
For simplicity, we only describe the notion of adelic $\QQ$-divisors. 

Let $U$ be a quasi-projective variety over $\CC$. 
First, the group of model $\QQ$-divisors of $U/\OB$ is defined by 
$$\wh\Div(U/\OB)_{\rm mod,\QQ}=\varinjlim_X \wh\Div(X/\OB)_\QQ,$$
where the limit is over the system of projective models $X$ of $U$ over $\CC$. 
Here $\wh\Div(X/\OB)$ is the group of Green divisors on $X$, i.e. pairs
 $\OD=(D,g_D)$, where $D$ is a Cartier divisor on $X$ and 
 $g_D: X\setminus |D|\to \RR$ is a Green function of $D$ on $X$.
The Green divisor  $\OD=(D,g_D)$ is called \emph{effective} if $D$ is {an effective divisor on $X$}  and $g_D\geq 0$. 
A model {$\QQ$-divisor} of $U/\OB$ is \emph{effective} if it is the image of some effective Green divisor coming from some projective model.

Second, take a \emph{boundary divisor} $(X_0, \OE_0)$ with $\OE_0=(E_0,g_0)$ of $U$ over $\OB$. That is, $X_0$ is a projective model of $U$ over $\CC$, 
 $E_0$ is an effective Cartier divisor on $X_0$ with support equal to $X_0\setminus U$, and $g_0$ is a strictly positive Green function of $E_0$ on $X_0$. 
This gives a \emph{boundary norm} 
$$\|\cdot\|_{\OE_0}:\wh\Div (U/\OB)_{\rmod,\QQ}
\lra [0,\infty]$$
by 
$$
\|\OD\|_{\OE_0}:=\inf\{\epsilon\in \BQ_{>0}: \ 
 -\epsilon \OE_0 \leq
\OD \leq  \epsilon \OE_0\}.
$$
Here the inequalities are defined in terms of effectivity. 
It further induces a \emph{boundary topology} on $\wh\Div (U/\OB)_{\rmod,\QQ}$, which does not depend on the choice of $(X_0,\OE_0)$.

Third, the group $\wh\Div(U/\OB)_{\QQ}$ of \emph{adelic $\QQ$-divisors} 
on $U/\OB$
is defined to be the completion of 
$\wh\Div(U/\OB)_{\rmod,\QQ}$ with respect to the 
\emph{boundary topology} induced by $(X_0, \OE_0)$. 
By definition, an adelic $\QQ$-divisor on $U/\OB$ is represented by a Cauchy sequence in $\wh \Div  (U/\OB)_{\rmod,\QQ}$, i.e., a sequence $\{\OD_i\}_{i\geq 1}$ in $\wh \Div  (U/\OB)_{\rmod,\QQ}$ satisfying the property that there is a sequence $\{\epsilon_i\}_{i\geq 1}$ of positive rational numbers converging to $0$ such that 
$$
 -\epsilon_i \OE_0 \leq
\OD_{i'}-\OD_{i} \leq  \epsilon_i \OE_0,\quad\ i'\geq i\geq 1.
$$

There is a canonical map
$$
\wh\Div(U/\OB)_\QQ \lra \Div(U)_\QQ,\quad
\{\OD_i\}_{i\geq 1}\longmapsto D_1|_{U}.
$$ 
For  $\OD=\{\OD_i\}_{i\geq 1}$, we usually call $D=D_1|_{U}$ the \emph{underlying $\QQ$-divisor} of $\OD$.  

There is an injective analytification map 
$$
\wh\Div(U/\OB)_\QQ \lra \Div(U^\an)_\QQ,\quad
\{\OD_i\}_{i\geq 1}\longmapsto (D, g_{\OD}).
$$ 
Here $D=D_1|_{U}$ is the {underlying $\QQ$-divisor} of $\OD$ as above, and
$g_\OD$ is defined to be $\lim_i g_{D_i}$, which converges uniformly on compact subsets of $U\setminus |D|$, and thus defines a Green function of $D$ on $U$. 
We call $g_\OD$ the \emph{Green function} of $\OD$ on $U$. 
By abuse of notation, we will also write $\OD=\{\OD_i\}_{i\geq 1}$ {as}  
$\OD=(D, g_\OD).$

There is also a canonical forgetful map
$$
\wh\Div(U/\OB)_\QQ \lra \wh\Div(U/\CC)_\QQ,\quad
\{\OD_i\}_{i\geq 1}\longmapsto \{D_i\}_{i\geq 1}.
$$ 
We usually write the image of $\OD$ by $\wt D$, and call $\wt D$ the \emph{geometric part} of $\OD$.  
We may also write $\wt\Div(U/\CC)_\QQ$ for $\wh\Div(U/\CC)_\QQ$ to emphasize the geometric situation. 

Finally, our adelic method for degeneration is summarized in the following basic result. 

\begin{lem}[adelic method]\label{adelic method}
Let $U$ be a quasi-projective variety over $\CC$. 
Let  
$(X_0, \OE_0)$ be a {boundary divisor} of $U$ over $\OB=(\CC,|\cdot|)$
with $\OE_0=(E_0, g_{\OE_0})$ on $X_0$. 
Let $\OD=(D, g_\OD)$ be an adelic {$\QQ$-divisor} in $\wh\Div(U/\OB)_\QQ$.
Denote by $\wt D$ and $\wt E_0$ the geometric parts of $\OD$ and {$\OE_0$} in $\wt\Div(U/\CC)_\QQ$.  
Assume that $\wt D \geq a\wt E_0$ for some constant $a\in \QQ$. Then for any real number $\epsilon>0$, there exists a real number $\delta$ such that  
$$
g_{\OD} \geq (a-\epsilon) g_{\OE_0}-\delta 
$$
on $U$. 
\end{lem}
\begin{proof}
By definition, there is a sequence $\OD_i=(D_i,g_i)$ for $i\geq1$ in $\wh\Div(U/\OB)_{\rmod,\QQ}$ such that 
$$
-\epsilon_i \OE_0\leq \OD-\OD_i\leq  \epsilon_i \OE_0
$$
for a sequence of rational numbers $\epsilon_i$ converging to 0.

The condition $\wt D \geq a\wt E_0$ gives 
$$
D_i\geq \wt D-\epsilon_i E_0 \geq (a-\epsilon_i) E_0. 
$$
Then the model divisor $D_i-(a-\epsilon_i) E_0$ is effective. As a consequence, its Green function 
$$g_i-(a-\epsilon_i) g_{\OE_0}>\delta_i$$ 
on $U$
for some real number $\delta_i$.  

Combining with  
$$g_\OD- g_i\geq-\epsilon_i g_{\OE_0},$$
the above inequality gives 
$$g_\OD-(a-2\epsilon_i) g_{\OE_0}>\delta_i$$ 
on $U$. 
This finishes the proof.
\end{proof}

Now it is easy to prove Theorem \ref{positive lower bound}. 

\begin{proof}[Proof of Theorem \ref{positive lower bound}]
Fix an integer $N\geq 3$, and denote by $S=M_{g,N}$ the (fine) moduli scheme of smooth curves of genus $g$ over $\CC$ with a full level-$N$ structure.
Then $S$ is a smooth quasi-projective variety over $\CC$, which follows from the GIT construction in \cite[\S7.4]{MFK}.  
By \cite[Thm. 2.1]{GO}, there is a projective compactification 
$S^*= M_{g,N}^*$ of $S$ together with a tautological stable relative curve 
$X^*\to S^*$. 

Denote $\OB=(\CC,|\cdot|)$. 
By the globalization of the $\varphi$-invariant over $\ZZ$, there is an adelic divisor $\ol\Phi$ on $S/\OB$ with underlying divisor $0$ and with Green function $\varphi:S\to \RR$ on $S$. 

The existence of $\ol\Phi$ implies that $\varphi:S\to \RR$ is continuous on $S$. 
As mentioned above, since $\varphi>0$ on $S$ 
by Zhang \cite{Zha3}, it suffices to prove that $\varphi$ tends to infinity along the boundary $S^*\setminus S$.

The boundary $\Delta=S^*\setminus S$ with the reduced structure is a Cartier divisor on $S^*$. Complete $\Delta$ to a Green divisor $\ol\Delta=(\Delta, g_{\ol\Delta})$ on $S^*$ with a strictly positive Green function $g_{\ol\Delta}$.
Then we have a boundary divisor $(S^*, \ol\Delta)$ of $S$ over {$\OB$}.  

From Step 5 of the proof of the {bigness in the geometric case}, 
the difference $\wt\Phi-\frac{1}{39}\Delta$ is an effective adelic $\QQ$-divisor in $\wt\Div(S/\CC)$.
In fact, it suffices to compare the values of the Green functions at all discrete valuations $v\in (S/\CC)^\an$ (where $\CC$ is trivially valued), since the set of discrete valuations is dense in the whole Berkovich space. 
Then the result follows from the results of Cinkir \cite[Thm. 2.11]{Cin1} in graph theory.

Now we apply Lemma \ref{adelic method}. This gives a relation
$$
\varphi> \frac{1}{40} g_{\ol \Delta}-c
$$
for some constant $c$. 
As a consequence, $\varphi$ goes to infinity along $\Delta$.
This finishes the proof.
\end{proof}

\subsection{Uniform fiberwise bigness} 
\label{sec fiberwise bigness}

Let $C$ be a geometrically integral smooth projective curve of genus $g\geq 2$ over a number field $K$. Recall that in \S\ref{sec Faltings height}, we have introduced  the {adjusted Faltings height}
$$h_\Fal^+(C)=\max\{h_\Fal(C),1\}.$$ 
Recall that in \S\ref{sec admissible single curves}, we have introduced 
the admissible volume $\ol\omega_{C/K,a}^2$, and now we introduce the 
\emph{normalized admissible volume}
$$
[\ol\omega_{C,a}^2]_\QQ=[\ol\omega_{C/K,a}^2]_\QQ= \frac{1}{[K:\QQ]} \ol\omega_{C/K,a}^2.
$$
The definitions of  $h_\Fal^+(C)$ and $[\ol\omega_{C,a}^2]_\QQ$ also {apply} to smooth projective curves $C$ of genus $g\geq 2$ over $\ol\QQ$, since any such a curve can be descended to a number field and the definitions {do} not depend on the choice of the number field. 

We have already seen the positivity $\ol\omega_{C/K,a}^2>0$. The following theorem from \cite[Thm. 1.4]{Yua21} is a uniform version of this positivity and also a fiberwise version of Theorem \ref{bigness55}.

\begin{thm} [uniform fiberwise bigness] \label{fiberwise11}
Let $g\geq 2$ be an integer. 
Then there are constants $c_3>0$ and $c_4>0$ depending only on $g$ satisfying the following properties. 
For any geometrically integral smooth projective curve $C$  of genus $g$ over a number field $K$, 
we have
$$
c_3 \cdot h_\Fal^+(C) \leq  [\ol\omega_{C/K,a}^2]_\QQ \leq  c_4 \cdot h_\Fal^+(C).
$$ 
\end{thm}

As we will see, the theorem implies that $h_\Fal^+(C)$ and {$[\ol\omega_{C/K,a}^2]_\QQ$} are equivalent invariants in uniform Bogomolov-type problems.  

The second inequality of the theorem {was previously known to experts}.
In fact, as in Theorem \ref{bigness55}, a combination of the works of Faltings \cite{Fal2}, Zhang \cite[Thm. 4.4]{Zha1} and Wilms \cite{Wil1} gives
$$
 [\ol\omega_{C/K,a}^2]_\QQ
\leq 12\, h_\Fal(C)+ 6g\log (2\pi^2). 
$$
The new piece here is the first inequality.

Let us first sketch an idea to prove Theorem \ref{fiberwise11} by Theorem \ref{bigness55}. The idea is outlined at the beginning of \cite[\S4.5]{Yua21}, and is different from the more explicit method in the proof of \cite[Thm. 4.14]{Yua21}. 

\begin{proof}[Proof of Theorem \ref{fiberwise11}]
Take $S$ to be a fine moduli space of smooth curves of genus $g$ over 
$\QQ$ with a suitable full level-$N$ structure, and take $\pi:X\to S$ to be the universal curve. We will introduce two adelic line bundles $\OM$ and $\ON$ on $S/\ZZ$. 

The first adelic line bundle on $S/\ZZ$ is
$$\ON=\pi_*\pair{\ol\omega_{X/S,a}\, ,\ol\omega_{X/S,a}}.$$
It is big on $S/\ZZ$ by Theorem \ref{bigness55}. 

The second adelic line bundle on $S/\ZZ$ is
$$
\OM= \ol\lambda_{S} +\ol\CO(c),
$$
where $\ol\lambda_{S}$ is the adelic Hodge line bundle on $S$ recalled in \S\ref{sec adelic hodge}, and  $\ol\CO(c)$  denotes a hermitian line bundle on $\Spec\ZZ$ of arithmetic degree $c$, viewed as an adelic line bundle on $S/\ZZ$ by pull-back.

We claim that $\OM$ is big on $S/\ZZ$ for sufficiently large constant $c>0$. 
This is a consequence of \cite[Lem. 5.2.10]{YZ}, by the fact that the Hodge line bundle $\lambda_{\ol S}$ on $\ol S/\QQ$ for a stable compactification 
$\ol\pi:\ol X\to\ol S$ is big on $\ol S/\QQ$. 

Once the adelic line bundles are big, the height inequality in 
Theorem \ref{height inequality} implies that there are a non-empty open subvariety $U\subsetneq S$ and constants $a_1, a_2>0$ such that 
$$
a_1 h_{\OM}(y) \leq h_{\ON}(y)\leq a_2 h_{\OM}(y),\quad \forall y\in U(\ol \QQ).
$$ 
Replacing $S$ by irreducible components of $S\setminus U$,  replacing $X\to S$ by the corresponding base change, and performing the above argument successively, we conclude that the above inequality holds for all points $y\in S(\ol \QQ)$ (instead of $y\in U(\ol \QQ)$). 

We can combine the inequalities with the equalities
$$
h_{\ON}(y)= [\ol\omega_{X_y,a}^2]_\QQ,\quad
h_{\OM}(y)
= h_{\rm Fal}(X_y) +c, \quad \forall y\in S(\ol\QQ).
$$
Here the first equality follows from the canonical isomorphism 
$$
\pi_*\pair{\ol\omega_{X/S,a}\, ,\ol\omega_{X/S,a}}|_{y'}
=(\pi_{y'})_*\pair{\ol\omega_{X_{y'}/y',a}\, ,\ol\omega_{X_{y'}/y',a}},
$$
which comes from the compatibility of the Deligne pairing with the base change by $y'\to S$. Here $y'\in S$ is the closed point corresponding to $y$. 

By a result of Bost \cite{Bos2} (cf. \cite[App.]{GR} or \cite[\S1.3]{JS1}), we have an explicit uniform lower bound 
$$h_{\rm Fal}(X_y)\geq - g\log (\sqrt2 \pi).$$
Then it suffices to take
$$c\geq g\log (\sqrt2 \pi)+1$$ 
to get 
Theorem \ref{fiberwise11}. 
\end{proof}

\section{Uniform Bogomolov conjecture}
\label{sec uniform}

Classical Diophantine problems aim to find rational solutions or integral solutions of polynomial equations with rational coefficients. In modern terminology, the goal of Diophantine geometry is to study rational points or integral points of algebraic varieties over number fields. 
The celebrated Mordell conjecture, proved by Faltings in 1983, is the following statement. 

\begin{theorem}[Faltings]
Let $C$ be a geometrically integral smooth projective curve of genus $g\geq2$ over a number field $K$. 
Then $C(K)$ is finite. 
\end{theorem}

The conjecture was raised by Mordell \cite{Mor22} when he proved the finite generation of the group of the rational points of elliptic curves over $\QQ$ in 1922. Faltings' proof of the conjecture in \cite{Fal} signified a milestone in the history of Diophantine geometry. 
{By analogy with} the Thue--Siegel--Roth theorem in Diophantine approximation, Vojta \cite{Voj} gave a new proof of the Mordell conjecture in 1991. 
Inspired by Faltings' proof, Lawrence--Venkatesh \cite{LV20}
gave a third proof of the Mordell conjecture in terms of $p$-adic Hodge theory.
We will come back to Vojta's proof later, since it will be used in our main results. 

Once we know that the set $C(K)$ is finite, a natural question is to give a suitable upper bound on the {cardinality} of this set in terms of natural geometric and arithmetic invariants of $C$ over $K$. 
In this direction, we have the following uniform version of the Mordell conjecture, which solves a problem proposed by Mazur \cite[p. 234]{Maz}.

\begin{theorem}[Vojta \cite{Voj}, Dimitrov--Gao--Habegger \cite{DGH}, K\"uhne \cite{Kuh}]
\label{uniform Mordell}
Let $g\geq 2$ be an integer. 
Then there are positive constants $c_1(g)$ and $c_2(g)$ depending only on $g$ such that for any geometrically integral smooth projective curve $C$ of genus $g$ over a number field
$K$, we have
$$
 |C(K)| \leq c_1(g) c_2(g)^{r}.
$$ 
Here $r=\rank\, J(K)$ is the Mordell--Weil rank of the Jacobian variety $J$ of $C$ over $K$.
\end{theorem}

The theorem is proved by a combination of the following bounds. 
\begin{enumerate}[(1)]
\item (Large points)
Vojta's proof of the Mordell conjecture actually gives a deep inequality concerning {the distribution} of rational points of large heights, which {in particular} implies an upper bound on the number of rational points of large heights. Note that this is sufficient for the Mordell conjecture by the Northcott property of heights.
The upper bound was refined by de Diego \cite{dDi} and R\'emond \cite{Rem00a, Rem00b}.
\item (Small points)
The works \cite{DGH, Kuh} prove a uniform Bogomolov conjecture, which gives an upper bound on the number of rational points of small heights. 
By a simple argument of sphere packing, the uniform Bogomolov conjecture also bounds the number of all rational points {in the complement of the points in (1)}. 
\end{enumerate}
In \cite{Yua21}, Yuan gave a different proof of the uniform Bogomolov conjecture of \cite{DGH, Kuh} by the theory of adelic line bundles of Yuan--Zhang \cite{YZ}. In fact, Yuan introduced the admissible canonical bundle over relative curves and proved its bigness as in \S\ref{sec admissible}, and then deduced the uniform Bogomolov conjecture by this bigness property. 

The goal of this section is to explain some background and ideas of the proof of Theorem \ref{quant Mordell}, with an emphasis on the approach of \cite{Yua21} {to} the uniform Bogomolov conjecture.
We will also compare our approach with that of \cite{DGH, Kuh} briefly.

Our topic is closely related to the previous surveys of Gao \cite{Gao21} and Habegger \cite{Hab22}. 
While \cite{Gao21} mainly concerns the proof of the uniformity result in Theorem \ref{uniform Mordell} through the approach of \cite{DGH, Kuh}, and \cite{Hab22} gives a panoramic view of various related problems and results on the Mordell conjecture, 
our current survey mainly concerns the approach of \cite{Yua21}.
We refer {readers} to these two surveys for more aspects of the topic.

\subsection{Uniform Bogomolov conjecture}

In this subsection, we introduce the Bogomolov conjecture, the uniform Bogomolov conjecture, and its application to the uniform Mordell conjecture.

\subsubsection{The Bogomolov conjecture for single curves}
\label{sec Bog single curves}

The Bogomolov conjecture, proposed by Bogomolov in \cite{Bog81}, is as follows. 

\begin{thm}[Ullmo]
\label{bog conj}
Let $C$ be a geometrically integral smooth projective curve of genus $g\geq 2$ over a number field $K$. 
Let $\alpha$ be a line bundle on $C_{\ol K}$ of degree 1. 
Then there is a constant $c>0$ such that 
$$
\#\{x\in C(\ol K): \hat h(x-\alpha)\leq c\}<\infty.
$$ 
\end{thm}

Recall that the N\'eron--Tate height 
$\hat h:J(\ol K)\to \RR$ is introduced in \S\ref{sec NT curve}. 
Here $J$ denotes the Jacobian variety of $C$ over $K$. 

The case $c=0$ of the conjecture is the Manin--Mumford conjecture for curves, which actually inspired Bogomolov's formulation. Shortly after, the Manin--Mumford conjecture for curves {was proved} by Raynaud \cite{Ray83}. 

The Bogomolov conjecture was first proved by Ullmo \cite{Ull}, and then we have a second proof combining the works {of} Zhang \cite{Zha1,Zha3}, Cinkir \cite{Cin1} and de Jong \cite{dJo3}. 
Let us review the proofs briefly. 

The proof of Ullmo \cite{Ull} is based on the celebrated equidistribution theorem of Szpiro--Ullmo--Zhang \cite{SUZ},  the original version of Theorem \ref{equi3}. 
In fact, {if} the Bogomolov conjecture fails, then there is a
generic sequence $\{x_m\}_m$ on $C(\ol K)$ with 
$\hat h(x_m-\alpha)\to 0$. 
By enlarging $K$ if necessary, we can assume that $\alpha$ is a divisor on $C$. 
Consider the generically finite morphism 
$$
\psi: C^g\lra J, \quad (t_1, \dots, t_g)\longmapsto t_1+\cdots+t_g-g\alpha. 
$$
From the sequence $\{x_m\}_m$, we obtain a generic and small sequence $\{y_n\}_n$ in $C^g(\ol K)$, where the components of $y_n$ are from the sequence $\{x_m\}_m$, and the image $\{\psi(y_n)\}_n$ is a generic and small sequence in $J(\ol K)$.
Applying the equidistribution theorem to both generic and small sequences at an archimedean place $v$, we obtain two limit measures $\mu_{C^g,v}$ and $\mu_{J,v}$, both of which are given by strictly positive smooth $(g,g)$-forms. The compatibility of the sequences implies 
$\mu_{C^g,v}=\psi^*\mu_{J,v}$. 
Note that $\psi$ is not finite, so it maps some sub-curve of $C^g$ to a point of $J$, and thus $\psi^*\mu_{J,v}$ is 0 along this curve.
This {contradicts} the strict positivity of $\mu_{C^g,v}$. 

For the second proof, note the identity 
$$
[K:\QQ]\, \hat h(x-\alpha)=h_{\OL_\alpha}(x),\quad x\in C(\ol K) 
$$ 
for the adelic line bundle 
$$\OL_\alpha=i_\alpha^* \CO(\ol\theta_{\alpha_0})$$
on $C$. 
Here $i_\alpha:C\to J$ is the Abel--Jacobi embedding in terms of the base divisor $\alpha$, the divisor $\theta_{\alpha_0}$ is as in \S\ref{sec NT curve}, and $\ol\theta_{\alpha_0}$ is the $[2]$-invariant adelic extension for the dynamical system $[2]:J\to J$. 
Here we have assumed that $\alpha$ and $\alpha_0$ are divisors on $C$ by enlarging $K$. 
By Zhang's fundamental inequality in Theorem \ref{minima2}, it suffices to prove the self-intersection number $\OL_\alpha^2>0$ on $C.$
By a calculation of Zhang \cite{Zha1}, we essentially have 
$$\OL_\alpha^2\geq \OL_{\alpha_0}^2=g(g-1)\, {\ol\omega_{C/K,a}^2},$$
where {$\ol\omega_{C/K,a}$} is Zhang's admissible canonical bundle of $C/K$ in \S\ref{sec admissible single curves}.  
Then it remains to prove the admissible volume  
{$\ol\omega_{C/K,a}^2$}
is strictly positive. 
This {was proved in the works of} 
Zhang \cite{Zha3}, 
de Jong \cite{dJo3}, and Cinkir \cite{Cin1}. 
See \S\ref{sec admissible single curves} for more details.

\subsubsection{Uniform Bogomolov conjecture}
\label{sec uniform bog}

Recall that the Faltings height of a curve is defined in \S\ref{sec Faltings height}.

Finally, we have the following result from \cite[Thm. 1.1]{Yua21}, which can be viewed as a uniform version of the Bogomolov conjecture (Theorem \ref{bog conj}). 
It strengthens and generalizes {the new gap principle} in \cite[Thm. 4.1]{Gao21} obtained as a combination of \cite[Prop. 7.1]{DGH} and \cite[Thm. 3]{Kuh}. 

\begin{thm} [uniform Bogomolov] \label{uniform bog}
Let $g\geq 2$ be an integer. 
Then there are positive constants $c_1,c_2$ depending only on $g$ such that for any geometrically integral smooth projective curve $C$ of genus $g$ over a {number field} $K$, and for any 
line bundle $\alpha\in \Pic(C_{\ol K})$ of degree 1, 
we have
$$
\#\left\{x\in C(\ol K): \hat h(x-\alpha)\leq c_1\big(h_\Fal^+(C) +
\hat h((2g-2)\alpha-\omega_{C/K})\big) \right\} \leq c_2.
$$ 
\end{thm}

It is worth noting that DeMarco--Krieger--Ye \cite{DKY} previously proved the new gap principle in the case that $g=2$, {$\alpha$ is represented by a Weierstrass point}, and $C$ has a morphism of degree two to an elliptic curve.

Our theorem is stronger and more general than the new gap principle of \cite{DGH,Kuh} by the following aspects:
\begin{enumerate}[(1)]
\item it has an extra non-negative term $\hat h((2g-2)\alpha-\omega_{C/K})$ in the formula;
\item it allows $\alpha$ to be in $\Pic^1(C_{\overline K})$ instead of just in $C(\overline K)$.
\end{enumerate}
Our proof is very different from that of \cite{DGH,Kuh}, and the key ingredient is a bigness result of adelic line bundles on the universal curve.
We will come back to that in the next subsection.

Let us compare these two approaches briefly. 
The key result of the approach of \cite{DGH} is a non-degeneracy result of the Faltings--Zhang morphism for families of curves, which follows from a mixed Ax--Schanuel theorem for the universal family of abelian varieties by \cite{Gao1}, and the proof of the latter uses the real-analytic Betti map and the o-minimality theory.
This implies a weak version of the uniform Bogomolov conjecture by their height inequality.
The work \cite{Kuh} strengthens the weak version. 
Its strategy is a family version of the proof of the original Bogomolov conjecture by Ullmo \cite{Ull}. 
In fact, \cite{Kuh} proves a version of the equidistribution theorem \ref{equi4} for abelian schemes independently. 
To apply the equidistribution theorem, {one still needs} the non-degeneracy result of \cite{DGH} for the Faltings--Zhang map.

The proof of \cite{Yua21} has the following three steps:
\begin{enumerate}[(1)]
\item Extend Zhang's construction of admissible metrics from projective curves to {families} of projective curves. 
\item Prove the bigness of the admissible canonical bundle of the universal family over the moduli space of curves introduced in (1).
\item Prove the uniform Bogomolov conjecture from the bigness result in (2).
\end{enumerate}
Note that our bigness theorem also implies the non-degeneracy {of the} Faltings--Zhang map. 
As mentioned above, the work of \cite{Kuh} is a family version of the proof of the Bogomolov conjecture by \cite{SUZ,Ull}, but the work of \cite{Yua21}
 is a family version of the proof of the Bogomolov conjecture by \cite{Zha1, Zha3, Cin1, dJo3}.

Before introducing the proof of Theorem \ref{uniform bog}, let us first sketch how the theorem and Vojta's inequality imply the uniform Mordell conjecture.

\subsubsection{Uniform Vojta inequality}

Recall the notation for distance and angles between points of curves from \S\ref{sec NT curve}. 
To prove the uniform Mordell conjecture, we also need the following uniform version of Vojta's inequality and Mumford's inequality. 

\begin{thm}[Vojta's inequality] \label{vojta}
Let $g\geq 2$ be an integer. 
Then there are constants $\lambda_1, \lambda_2, \lambda_3$, depending only on $g$ and with $\lambda_3> \lambda_2>1$, {such} that for any geometrically integral smooth projective curve $C$ of genus $g$ over a {number field} $K$, and for any distinct points 
$x,y\in C(\ol K)$ satisfying 
$$
|x|\geq |y| \geq \lambda_1\sqrt{h_\Fal^+(C)}
$$
and 
$$
\angle(x,y) \leq  \arccos (3/4),
$$
we have
$$
\lambda_2\, |y|\leq  |x|\leq \lambda_3\, |y|.
$$
\end{thm}

Without uniformity of the constants,  the existence of the constant 
$\lambda_2$ (without the angle condition) is essentially due to Mumford  \cite[\S3, Cor. 1]{Mum}, and the existence of the constant $\lambda_3$ is essentially due to Vojta \cite{Voj}. In fact, Vojta's inequality is inspired by Mumford's inequality, and it gives the second proof of the Mordell conjecture {using techniques from} Arakelov geometry and ideas from Diophantine approximation. 
The uniformity of the constants {was obtained in later works} of de Diego \cite{dDi} and R\'emond \cite{Rem00a, Rem00b}.

We will not sketch a proof of Theorem \ref{vojta} in this paper, but we will {give} a few remarks and references about it.
First, Bombieri \cite{Bom} (based on \cite{Fal3}) wrote a variant of Vojta's proof, which replaced the use of Arakelov geometry by relatively elementary techniques.
We refer to the textbooks \cite{HS,BG06,IKM22} for Bombieri's variant of the proof.
Second, Yuan \cite{Yua25} gave a second variant of Vojta's proof, which replaced the use of the arithmetic Riemann--Roch theorem of Gillet--Soul\'e \cite{GS2} by the arithmetic Siu inequality (Theorem \ref{bigness}) for arithmetic varieties. 
Third, a version of Theorem \ref{vojta} with explicit constants was recently obtained by Yu--Yuan--Zhou \cite{YYZ} following the method of \cite{Yua25}. 
We will return to this version in \S\ref{sec more recent}.

\subsubsection{Uniform Mordell conjecture: sphere packing}

With the uniform Bogomolov conjecture (Theorem \ref{uniform bog}),  and the uniform versions of Vojta's inequality and Mumford's inequality 
(cf. Theorem \ref{vojta}), we prove the uniform Mordell conjecture (Theorem \ref{uniform Mordell}) in the following.

Denote $r=\rank\, J(K)$. The real vector space $V=J(K)_\RR$, endowed with the inner product $\pair{\cdot,\cdot}$ induced by the N\'eron--Tate height, is isometric to the Euclidean space $\RR^r$.
We first bound the set of large points given by 
$$
C(K)_{\rm large}:=\left\{x\in C(K): |x|\geq \lambda_1\, h_\Fal^+(C)^{1/2}\right\}.
$$

For every nonzero vector $x\in V$, consider the cone
$$
\mathrm{cone}(x):=\left\{y\in V: \angle(x,y)\leq  \frac12\arccos (3/4)\right\}.
$$
Then $V$ is covered by finitely many cones 
$\mathrm{cone}(x_1), \dots, \mathrm{cone}(x_n)$, as a consequence of the compactness of the unit sphere of $V$.
By basic sphere packing, we can take $n\leq 7^r$. 
For every $i$, denote by $\Sigma_i$ the set of points of $C(K)$ whose images in $V$ lie in $\mathrm{cone}(x_i)$. 
Order points of $\Sigma_i$ as a sequence $(y_1,y_2,\dots)$, such that $(|y_1|, |y_2|,\dots)$ is increasing. 
By Theorem \ref{vojta}, we have 
$$ |y_1|\leq |y_j| \leq \lambda_3\, |y_1|, \quad 
 |y_{j+1}|\geq  \lambda_2\, |y_j|.$$
 As a consequence, we have 
 $$
\# \Sigma_i \leq \log_{\lambda_2}(\lambda_3)+1.
 $$
It follows that 
$$
\# C(K)_{\rm large}\leq 7^r (\log_{\lambda_2}(\lambda_3)+1).
$$
This bounds the number of large points. 

By Theorem \ref{uniform bog}, the {cardinality} of the set 
$$
C(K)_{\rm small}:=\left\{x\in C(K): |x|\leq \sqrt{c_1}\, h_\Fal^+(C)^{1/2} 
 \right\}
 $$
 is at most $c_2$. 
If $\sqrt{c_1}\geq \lambda_1$, then 
$C(K)_{\rm small}$ and $C(K)_{\rm large}$ cover the whole set $C(K)$. 
However, we cannot expect this  to happen here. 
Therefore, we have to bound the moderate points 
$$
C(K)_{\rm moderate}:=\left\{x\in C(K): |x|\leq \lambda_1\, h_\Fal^+(C)^{1/2} 
 \right\}. 
$$
For this we need to use the extra $\alpha$ in Theorem \ref{uniform bog}. 

For convenience, denote 
$$
R_1= \sqrt{c_1}\, h_\Fal^+(C)^{1/2}, \quad 
R_2=\lambda_1\, h_\Fal^+(C)^{1/2}. 
$$
Enumerate points of $C(K)_{\rm moderate}$ as $\{x_1,x_2,\dots\}$. 
For every $i$, let $B_i$ be the open ball of center $x_i$ and radius 
$\ds\frac{R_1}{2}$ in $V$. 
By Theorem \ref{uniform bog}, any ball $B_i$ intersects at most $c_2$ other such balls. 
Then the natural map 
$$\coprod_i B_i\lra B$$
to the ball 
$$B= \left\{x\in V:|x|\leq R_2+\frac{R_1}{2}\right\}$$
has fibers of {order} at most $c_2$. 
Comparing the volumes, we have
$$
\vol(B_i)\cdot \# C(K)_{\rm moderate}\leq c_2\cdot \vol(B).
$$
It follows that
$$
\# C(K)_{\rm moderate} \leq c_2\frac{\vol(B)}{\vol(B_i)}  
=   c_2 \left(\frac{2R_2}{R_1}+1\right)^r=   
c_2 \left(\frac{2\lambda_1}{\sqrt{c_1} }+1\right)^r. 
$$
This finishes {the proof of Theorem} \ref{uniform Mordell}.

\subsection{Proof of the uniform Bogomolov conjecture}
 
The goal of this subsection is to sketch a proof of the uniform Bogomolov conjecture (Theorem \ref{uniform bog}).  
We first introduce a general framework concerning potential bigness, which gives us a general mechanism to bound algebraic points of small heights, and then {specialize it to} our particular setting to prove the theorem.

\subsubsection{Potential bigness}

\kkk
Let $S$ be a quasi-projective and flat normal integral scheme over $k$. 
Let $\pi:X\to S$ be a smooth relative curve. 
For $m>0$, we denote by
$$
X_{/S}^m=X\times_SX\times_S\cdots \times_S X
$$
the $m$-fold fiber product of $X$ over $S$. We take the convention $X_{/S}^0=S$.

Let $\OL$ be an adelic line bundle on $X/k$. 
{In additive notation},  \emph{the box tensor}
$$
m_{\boxtimes}\OL=p_1^*\OL+p_2^*\OL+\cdots+p_m^*\OL
$$
is an adelic line bundle on $X_{/S}^m$. Here $p_i:X_{/S}^m\to X$ is the projection to the $i$-th component.
We say that $\OL$ is \emph{potentially big on $X/S$} if there is a positive integer $m$ such that the adelic line bundle $m_{\boxtimes}\OL$ is big on 
$X_{/S}^m$. 

The notion of \emph{potential bigness} is from \cite[\S 4.1]{Yua21}, and has some similarity with the notion of \emph{correlation} in the previous work of Caporaso--Harris--Mazur \cite{CHM}.

Concerning potential bigness, we have the following key criterion from 
\cite[Thm. 4.1]{Yua21}. The proof is not hard, but the result is surprising in that bigness of the Deligne pairing can be converted to bigness of a total space. 

\begin{thm}[potential bigness] \label{potentially big}
Let $\OL$ be an adelic line bundle on $X/k$.
Assume that $\OL$ is nef on $X$, and the degree of $L$ on fibers of $\pi:X\to S$ is strictly positive.
Then the following are equivalent:
\begin{enumerate}[(1)]
\item $\pi_*\pair{\OL,\cdots,\OL}$ is big on $S$;
\item $m_\boxtimes\OL$ is big on $X_{/S}^m$ for all $m\geq \dim S$;
\item $\OL$ is potentially big on $X/S$; i.e., $m_\boxtimes\OL$ is big on $X_{/S}^m$ for some $m\geq 1$.
\end{enumerate}
\end{thm}

\begin{proof}
We only sketch the idea, and refer to \cite[Thm. 4.1]{Yua21} for a serious proof. 
The implication $(2)\Rightarrow(3)$ is trivial. 
Denote $d=\dim S$ in the following.

For the implication $(1)\Rightarrow(2)$, if $m\geq d$, expand the top self-intersection number 
$$
\OL_m^{d+m}
=(p_1^* \OL + p_2^* \OL + \cdots + p_m^* \OL)^{d+m}.
$$
The expansion includes the term
$$
(p_1^* \OL)^{2}  \cdots (p_d^* \OL)^{2}
\cdot (p_{d+1}^* \OL)  \cdots (p_m^* \OL)
=a^{m-d} (\pi_*\pair{\OL,\dots,\OL})^d.
$$
Here $a$ denotes the degree of $L$ on fibers of $\pi:X\to S$. 

For the implication $(3)\Rightarrow(1)$,
assume that for some $m\geq 1$, $m_\boxtimes\OL$ is big on $X_{/S}^m$.
Then the Deligne pairing 
$$\ON=(\pi_m)_*\pair{m_\boxtimes\OL,\dots, m_\boxtimes\OL}$$ 
is nef and big on $S$.
Here $\pi_m:X_{/S}^m\to S$ is the structure morphism. 
A calculation shows that $\ON$ is a positive multiple of 
$\pi_*\pair{\OL,\cdots,\OL}$.
\end{proof}

Note that the above notation
$X_{/S}^m$ and $m_{\boxtimes}\OL$, the notion of potential bigness, and 
the theorem extend naturally to quasi-projective varieties $S$ over number fields.
To apply the above theorem, we need the following key result relating potential bigness to {the distribution} of small points from \cite[Thm. 4.2]{Yua21}.

\begin{thm} \label{small points}
Let $S$ be a quasi-projective variety over a number field $K$.
Let $\pi:X\to S$ be a smooth relative curve. 
Let $\OL$ be a nef adelic line bundle on $X/\ZZ$, and $\OM$ be an adelic line bundle on $S/\ZZ$.
If $\OL$ is potentially big on $X/S$, then there are a non-empty Zariski open subvariety $U\subset S$ and constants  
$c_1,c_2>0$, such that 
for any $y\in U(\ol K)$,
$$
\#\{x\in X(\ol K): \pi(x)=y, \, h_\OL(x)\leq c_1\, h_\OM(y)\} \leq c_2.
$$ 
\end{thm}

\begin{proof}
By assumption, $\OL_m=m_\boxtimes\OL$ is big on $X_m=X_{/S}^m$ for some $m\geq 1$. Fix such an $m$. 
By Proposition \ref{height inequality}, there are a Zariski closed subset $Z\subsetneq X_m$ and $\epsilon>0$ such that 
$$
\{x\in X_m(\ol K):h_{\OL_m}(x)\leq \epsilon h_\OM(\pi_m(x))  \} \subset Z(\ol K).
$$ 
Here $\pi_m:X_m \to S$ is the structure morphism. 

For any $y\in S(\ol K)$, denote 
$$
\Sigma(y):=\{x \in X(\ol K): \pi(x)=y,\, h_{\OL}(x)\leq \frac{\epsilon}{m} h_\OM(y)\}.  
$$
The key is the inclusion 
$$\Sigma(y)^m \subset Z(\ol K).$$
This follows from the height identity
$$
h_{\OL_m}(x)=h_{\OL}(x_1)+\cdots +h_{\OL}(x_m)
$$
for any $x\in X_m(\ol K)$ represented by $(x_1,\cdots, x_m)$ under the expression 
$$X_m(\ol K)=\{(x_1,\cdots, x_m)\in X(\ol K)^m: \pi(x_1)=\cdots=\pi(x_m)\}.$$ 

Let $U$ be a non-empty open subscheme of $S$ such that $Z$ is flat over $U$. 
By a simple counting argument from \cite[Lem. 4.3]{Yua21} (cf.  \cite[Lem. 6.3]{DGH} and \cite[Lem. 7.3]{Gao21}), the inclusion $\Sigma(y)^m \subset Z_y(\ol K)$ 
forces $\Sigma(y)$ to be finite and bounded uniformly in $y\in U(\overline K).$
\end{proof}

\subsubsection{Proof of the uniform Bogomolov conjecture}

Now we are ready to sketch a proof of Theorem \ref{uniform bog}.
Recall that the theorem asserts that for any geometrically integral smooth projective curve $C$ of genus $g\geq 2$ over a {number field} $K$, and for any 
line bundle $\alpha\in \Pic(C_{\ol K})$ of degree 1, 
we have
$$
\#\left\{x\in C(\ol K): \hat h(x-\alpha)\leq c_1\big(h_\Fal^+(C) +
\hat h((2g-2)\alpha-\omega_{C/K})\big) \right\} \leq c_2.
$$ 
Here the positive constants $c_1,c_2$ {depend} only on $g$. 

\paragraph{Canonical case}
To illustrate the idea, we first treat the canonical case where $\alpha$ satisfies 
$(2g-2)\alpha=\omega_{C/K}$. In this case, the result {can be written simply as} 
$$
\#\left\{x\in C(\ol K): \hat h(x)\leq c_1\, h_\Fal^+(C)  \right\} \leq c_2.
$$ 
Then we {explain how to revise the proof to obtain uniformity} in the general situation. 

Take $S$ to be a fine moduli space of smooth curves of genus $g$ over 
$\QQ$ with a suitable full level-$N$ structure, and take $\pi:X\to S$ to be the universal curve.
Denote by $J$ the relative Jacobian variety of $X$ over $S$. 
Via the relative dualizing sheaf $\omega=\omega_{X/S}$, we have an
Abel--Jacobi map 
$$
i_\omega: X\lra J, \quad x\longmapsto (2g-2)x-\omega.
$$
By Theorem \ref{canonical alternative definition}, we have
$$
\pi_*\pair{i_\omega^*\OTheta,i_\omega^*\OTheta}
={16g(g-1)^3}\,\pi_*\pair{\overline\omega_{X/S,a},\overline\omega_{X/S,a}}
$$
as adelic line bundles on $S/\ZZ$. 
Here the adelic line bundle $\OTheta$ is constructed from the dynamical system $[2]: J\to J$ over $S$ right before the theorem.

By Theorem \ref{bigness55} for bigness, 
$\pi_*\pair{\overline\omega_{X/S,a},\overline\omega_{X/S,a}}$ is big on $S/\ZZ$, and thus $\pi_*\pair{i_\omega^*\OTheta,i_\omega^*\OTheta}$ is big on $S/\ZZ$. 
By Theorem \ref{potentially big} for potential bigness, 
$i_\omega^*\OTheta$ is potentially big on $X/S$. 
By Theorem \ref{small points}, for any adelic line bundle $\OM$ on $S/\ZZ$, there are a non-empty Zariski open subvariety $U\subset S$ and constants  
$c_1,c_2>0$, such that 
for any $y\in U(\ol \QQ)$,
$$
\#\{x\in X(\ol \QQ): \pi(x)=y, \, h_{{i_\omega^*\OTheta}}(x)\leq c_1\, h_\OM(y)\} \leq c_2.
$$ 
Replacing $S$ by irreducible components of $S\setminus U$,  replacing $X\to S$ {by} the corresponding base change, and performing the above argument successively, we conclude that the above inequality holds for all points $y\in S(\ol \QQ)$ (instead of $y\in U(\ol \QQ)$).

To convert the inequality to individual curves, we choose the adelic line bundle $\OM$ on $S/\ZZ$ by
$$
\OM= \ol\lambda_{S} + \ol\CO(c).
$$
This is already defined in \S\ref{sec fiberwise bigness}.
For any point $y\in S(\ol \QQ)$, we simply have
$$
h_{\OM}(y)
= h_{\rm Fal}(X_y) +c.
$$
By the result of Bost \cite{Bos2} (cf. \cite[App.]{GR} or \cite[\S1.3]{JS1}), 
$$h_{\rm Fal}(X_y)\geq - g\log (\sqrt2 \pi).$$
Taking 
$$c= g\log (\sqrt2 \pi)+1,$$ 
we have 
$$h_{\OM}(y)\geq  h_{\rm Fal}^+(X_y). $$
This proves Theorem \ref{uniform bog} in the canonical case
$(2g-2)\alpha=\omega_{C/K}$. 

\paragraph{General case}
Now we {discuss} the proof of Theorem \ref{uniform bog} for general 
$\alpha$. 
Let $\pi:X\to S$ be as above. 
Denote 
$$X_J=J\times_S X, \quad J_J=J\times_S J,$$
 viewed as $J$-schemes via the first projections 
 $$q_1:X_J\to J, \quad p_1:J_J\to J.$$ 
Note that $J_J$ is canonically isomorphic to the Jacobian scheme of $X_J$ over $J$.
We are going to apply the above mechanism to $q_1:X_J\to J$ 
(instead of $\pi:X\to S$).

Consider the morphism 
$$\tau:J\times_SX \lra J\times_SJ, \quad (y,x)\longmapsto (y,y+(2g-2)x-\omega_{X/S}).$$
This morphism agrees with the $J$-morphism
$$
i_{\omega-Q}: X_J\lra J_J,\quad x\longmapsto (2g-2)x-(\omega_{X_J/J}-Q).
$$
Here 
$Q$ is a universal line bundle on $J\times_S X$ in a suitable sense.
Denote by $\OTheta_J=p_2^*\OTheta$ the adelic line bundle on $J_J$ by the pull-back via the second projection $p_2:J\times_SJ\to J$.
Then we obtain an adelic line bundle $\tau^*(\OTheta_J)$ on $X_J$ by pull-back. 

By a calculation, we have the identity 
$$
q_{1*}\pair{\tau^*(\OTheta_J), \tau^*(\OTheta_J)}= 
16(g-1)^3\OTheta+16g(g-1)^3\pi_J^*\pi_*\pair{\ol\omega_{X/S,a}\, ,\ol\omega_{X/S,a}}
$$
in $\wh\Pic(J/\ZZ)$.
Then Theorem \ref{bigness55}  implies that  $q_{1*}\pair{\tau^*(\OTheta_J),\tau^*(\OTheta_J)}$ is nef and big on $J$. 
Therefore, we are ready to apply Theorem \ref{potentially big} and Theorem \ref{small points}
to $q_1:X_J\to J$ and the adelic line bundle $\OL=\tau^*(\OTheta_J)$ on $X_J$. 
As a consequence of the above mechanism, for any adelic line bundle $\OM$ on $J/\ZZ$, there are constants  
$c_1,c_2>0$, such that 
for any $y\in J(\ol \QQ)$,
$$
\#\{x\in X_J(\ol \QQ): q_1(x)=y, \, h_\OL(x)\leq c_1\, h_\OM(y)\} \leq c_2.
$$ 

Now we choose the adelic line bundle $\OM$ on $J$ by
$$
\OM= \OTheta + \ol\lambda_{S} +\ol\CO(c).
$$
Here $\ol\lambda_{S}$ and  
$\ol\CO(c)$ are as above, but they are viewed as adelic line bundles on $J/\ZZ$ by pull-back.
For any point $y\in J(\ol \QQ)$ with image $s\in S(\ol \QQ)$, we have
$$
h_{\OM}(y)
= 2\, \hat h(y)
+ h_{\rm Fal}(X_s) +c.
$$
This proves the theorem by taking $c$ as above and taking
$y_s\in J_s(\ol\QQ)$ to be  
$(2g-2)\alpha-\omega_{X_s/\ol\QQ}$ on $C_{\ol K}=X_s$.

\subsection{More recent works: quantitativity}
\label{sec more recent}

In the end, we state a quantitative Mordell conjecture and some related results recently proved by Yu--Yuan--Zhou \cite{YYZ}.  We refer to the original paper and the survey \cite{Yuan ICM} for more details of the results. 

The following theorem, a simplified version of \cite[Thm. 1.3]{YYZ}, is a quantitative version of the uniform Mordell conjecture (Theorem \ref{uniform Mordell}). 

\begin{theorem}[quantitative Mordell]
\label{quant Mordell}
For any geometrically integral smooth projective curve $C$ of genus $g\geq 2$ over a number field $K$, we have
$$
 |C(K)| \leq 10^{13} g^8 \big(1+\frac{3\log g}{g}\big)^{r}.
$$ 
Here $r=\rank\, J(K)$ is the Mordell--Weil rank of the Jacobian variety $J$ of $C$ over $K$.
\end{theorem}

The proof of Theorem \ref{quant Mordell} still consists of bounding points by different heights.
For large points, we have the following quantitative version of Vojta's inequality and Mumford's inequality (Theorem \ref{vojta}) from \cite[Thm. 1.9, Thm. 1.10]{YYZ}.

\begin{theorem}[quantitative Vojta inequality] \label{quant vojta}
Let  $C$  be a geometrically integral smooth projective curve of genus $g\geq 2$ over a number field $K$.
Let $x$ and $y$ be distinct points of $C(\overline K)$ such that 
$$
|x|\geq |y| \geq 1.2\cdot 10^{9}\cdot g^{\frac{7}{3}}
\sqrt{[\bar\omega_{C/K,a}^2]_\QQ}
$$
and 
$$
\angle(x,y) \leq \arccos \sqrt{\frac{1.01}{g}}. 
$$
Then 
$$
1.15\, |y|\leq |x| \leq 10^5 g^{\frac52}\, |y|.
$$
\end{theorem}

Recall that Theorem \ref{fiberwise11} asserts that $h_\Fal^+(C)$ and $[\ol\omega_{C/K,a}^2]_\QQ$ are equivalent invariants for our purpose.  
The proof of Theorem \ref{quant vojta} follows the variant of Vojta's proof by Yuan \cite{Yua25}, with precise calculations in terms of admissible adelic line bundles on curves. 

For small points, we have the following quantitative version of the uniform Bogomolov conjecture (Theorem \ref{uniform bog}) from \cite[Thm. 1.11]{YYZ}.

\begin{theorem}[quantitative Bogomolov]
\label{quant bog}
Let  $C$  be a geometrically integral smooth projective curve of genus $g\geq 2$ over a number field $K$, and $\alpha$ be a divisor of degree 1 on $C_{\ol K}$. Then
$$\#\left\{x\in C(\ol K): |x-\alpha|\leq \frac{1}{\sqrt{16g}} 
\sqrt{[\bar\omega_{C/K,a}^2]_\QQ}\right\}
\leq 6.5\cdot 10^{11}\cdot g^{\frac{17}{3}}.
$$
\end{theorem}

Previously, Looper--Silverman--Wilms \cite{LSW} proved a quantitative Bogomolov conjecture over function fields with explicit constants. 
Applying this result, Jiawei Yu \cite{Yu26} proved a version of Theorem \ref{quant Mordell} over function fields of characteristic 0. 

The proof of \cite{LSW} is based on estimates of heights on single curves, while the approaches of \cite{DGH, Kuh, Yua21} are based on estimates on moduli spaces. 
The approach of \cite{YYZ} {follows} the framework of Looper--Silverman--Wilms \cite{LSW} {and overcomes} all the archimedean obstacles by proving many desired inequalities on invariants of compact Riemann surfaces.  
For example, Theorem \ref{positive lower bound} has the following quantitative version from 
\cite[Thm. 1.13]{YYZ}.

\begin{thm} 
For any connected compact Riemann surface $C$ of genus $g\geq2$, we have
$$\varphi(C)\geq 10^{-7}g^{-\frac53}.$$
\end{thm}

To obtain a sharp bound on medium points, \cite{YYZ} also 
proves and applies a quantitative variant related to the extra $\hat h((2g-2)\alpha-\omega_{C/K})$ in Theorem \ref{uniform bog} obtained by \cite{Yua21}.


\end{document}